\documentclass[12pt]{report}   
\usepackage{graphicx}  
\usepackage[letterpaper, left=1in, right=1in, top=1in, bottom=1in]{geometry}
\usepackage{setspace}  
\usepackage{times}  
\usepackage[explicit]{titlesec}  
\usepackage[titles]{tocloft}  
\usepackage[backend=bibtex, sorting=none, bibstyle=ieee]{biblatex}  
\usepackage{appendix}  
\usepackage{rotating}  
\usepackage[normalem]{ulem}  
\usepackage{textcomp} 
\usepackage{indentfirst} 
\usepackage{booktabs,array,arydshln} 
\usepackage{amsmath} 
\usepackage[T1]{fontenc} 
\usepackage[utf8]{inputenc} 
\usepackage{newtxmath} 
\usepackage[hidelinks]{hyperref}

\usepackage{lipsum}
\usepackage{amsfonts}
\usepackage{epstopdf}
\usepackage{algorithmic}
\usepackage{multirow}
\usepackage{algorithm}
\usepackage{subcaption}

\usepackage{amsthm}
\usepackage[capitalize]{cleveref}
\usepackage{tablefootnote}
\usepackage{enumitem}
\usepackage{threeparttable}

\theoremstyle{definition}
\newtheorem{example}{Example}[chapter]
\newtheorem{definition}[example]{Definition}
\newtheorem{remark}[example]{Remark}
\theoremstyle{plain}
\newtheorem{theorem}[example]{Theorem}

\newtheorem{lemma}[example]{Lemma}
\newtheorem{proposition}[example]{Proposition}

\crefname{assumption}{Assumption}{Assumptions}
\crefname{section}{Section}{chapter}

\usepackage{tikz}
\usepackage{pgfplots}
\usetikzlibrary{intersections, pgfplots.fillbetween}

\allowdisplaybreaks

\bibliography{references}

\AtEveryBibitem{\clearfield{issn}}
\AtEveryBibitem{\clearlist{issn}}

\AtEveryBibitem{\clearfield{language}}
\AtEveryBibitem{\clearlist{language}}

\AtEveryBibitem{\clearfield{doi}}
\AtEveryBibitem{\clearlist{doi}}

\AtEveryBibitem{\clearfield{url}}
\AtEveryBibitem{\clearlist{url}}

\AtEveryBibitem{%
  \ifentrytype{online}
    {}
    {\clearfield{urlyear}\clearfield{urlmonth}\clearfield{urlday}}}

\begin{document}
\singlespacing  


\begin{titlepage}
\begin{center}

\begin{singlespacing}
\vspace*{6\baselineskip}
Stability of first-order methods in tame optimization\\
\vspace{3\baselineskip}
Lexiao Lai\\
\vspace{18\baselineskip}
Submitted in partial fulfillment of the\\
requirements for the degree of\\
Doctor of Philosophy\\
under the Executive Committee\\
of the Graduate School of Arts and Sciences\\
\vspace{3\baselineskip}
COLUMBIA UNIVERSITY\\
\vspace{3\baselineskip}
\the\year
\vfill

\end{singlespacing}

\end{center}
\end{titlepage}


\begin{titlepage}
\begin{singlespacing}
\begin{center}

\vspace*{35\baselineskip}

\textcopyright  \,  \the\year\\
\vspace{\baselineskip}	
Lexiao Lai\\
\vspace{\baselineskip}	
All Rights Reserved
\end{center}
\vfill

\end{singlespacing}
\end{titlepage}

\pagenumbering{gobble}

\begin{titlepage}
\begin{center}

\vspace*{5\baselineskip}
\textbf{\large Abstract}
\vspace*{5mm}

Stability of first-order methods in tame optimization

\vspace*{5mm}

Lexiao Lai
\end{center}

Modern data science applications demand solving large-scale optimization problems. The prevalent approaches are first-order methods, valued for their scalability. These methods are implemented to tackle highly irregular problems where assumptions of convexity and smoothness are untenable. 

Seeking to deepen the understanding of these methods, we study first-order methods with constant step size for minimizing locally Lipschitz tame functions. To do so, we propose notions of discrete Lyapunov stability for optimization methods. Concerning common first-order methods, we provide necessary and sufficient conditions for stability. We also show that certain local minima can be unstable, without additional noise in the method. Our analysis relies on the connection between the iterates of the first-order methods and continuous-time dynamics.
\vspace*{\fill}
\end{titlepage}

\pagenumbering{roman}
\setcounter{page}{1} 
\renewcommand{\cftchapdotsep}{\cftdotsep}  
\renewcommand{\cftchapfont}{\normalfont}  
\renewcommand{\cftchappagefont}{}  
\renewcommand{\cftchappresnum}{Chapter }
\renewcommand{\cftchapaftersnum}{:}
\renewcommand{\cftchapnumwidth}{5em}

\setlength{\cftbeforechapskip}{2ex}
\setlength{\cftbeforesecskip}{2ex}
\setlength{\cftbeforesubsecskip}{2ex}
\setlength{\cftbeforesubsubsecskip}{1ex}

\newcommand{\diag}{\mathop{\rm diag}}
\newcommand{\prox}{{\rm prox}}
\newcommand{\dom}{\mathop{\rm dom}} 
\newcommand{\graph}{\mathop{\rm graph}} 
\newcommand{\epi}{\mathop{\rm epi}} 
\newcommand{\conv}{\mathop{\rm conv}}
\newcommand{\cconv}{\overline{\mathop{\rm conv}}~}
\newcommand{\intr}{\mathop{\rm int}}
\newcommand{\rintr}{\mathop{\rm relint}}
\newcommand{\gra}{{\rm graph}}
\newcommand{\ones}{\mathbf 1}
\newcommand{\aff}{{\rm aff}}
\newcommand{\rank}{\mathop{\rm rank}}

\titleformat{\chapter}[display]
{\normalfont\bfseries\filcenter}{\chaptertitlename\ \thechapter}{0pt}{\large{#1}}

\renewcommand\contentsname{Table of Contents}

\tableofcontents
\setlength{\cftparskip}{\baselineskip}
\listoffigures
\listoftables

\clearpage

\phantomsection
\addcontentsline{toc}{chapter}{Acknowledgments}

\clearpage
\begin{center}

\vspace*{5\baselineskip}
\textbf{\large Acknowledgements}
\end{center}

First and foremost, I would like to express my deepest gratitude to my advisor, C\'edric Josz. This thesis would have not been possible without his unlimited support and guidance. He taught me everything about research, and I have always been amazed by his creative mind. I am also grateful to Kaizheng Wang, Daniel Bienstock, Donald Goldfarb, and Jiawang Nie for serving on my dissertation committee and for their valuable feedback. I would also like to thank Henry Lam and Anthony Man-Cho So for their support on my job search.

I would also like to thank all my fellow Ph.D. students at Columbia IEOR for the incredible moments that we shared together, from attending classes, doing sports, to exploring restaurants in the city. Special thanks to my collaborator Xiaopeng Li, who is always ready to answer any technical questions. I would also like to thank all the IEOR administrative staff who were of great help.

I am grateful to my partner Jingyu Chen for her unconditional support and belief.

Finally, I would like to thank my parents, Weihong Xie and Chunyang Lai, for everything. I will always be indebted to them for their boundless love and help in every stage of my life.

\clearpage





\clearpage
\pagenumbering{arabic}
\setcounter{page}{1} 

\phantomsection
\addcontentsline{toc}{chapter}{Introduction}

\begin{center}
\vspace*{5\baselineskip}
\textbf{\large Introduction}
\end{center}

The last decade has witnessed the fast growth of data science, with successful applications including image recognition \cite{he2016deep,dosovitskiy2020image}, natural language processing \cite{vaswani2017attention,devlin2018bert}, and image/language synthesis \cite{brown2020language,rombach2022high}. Central in this success is the effective training of large-scale models, which requires solving optimization problems with suitable algorithms. The prevailing algorithms are the first-order methods \cite{robbins1951stochastic,kingma2014adam,ruder2016overview}, namely the optimization algorithms that require only first-order derivatives. These algorithms have low per-iteration costs and can be easily distributed among GPUs \cite{dean2012large}.  However, their application typically lacks guarantees, raising concerns regarding robustness and explainability. This issue arises because the objectives are poorly structured, which renders them outside the scope of classical optimization theory.

In this thesis, we seek to provide guarantees for minimizing a general class of locally Lipschitz functions with first-order methods. We focus on the methods with constant step sizes, while most of the results continue to hold with variable step sizes that are nonsummable, except for those in \cref{ch:instability}. In order to analyze the iterates of these methods, we propose notions of discrete Lyapunov stability (\cref{def:discrete_lyapunov,def:global_stability}), akin to the Lyapunov stability in dynamical systems \cite{liapounoff1907probleme} \cite[Equation (5.6)]{sastry2013nonlinear}. Informally, a point is (locally) stable if the iterates generated by first-order methods stay close to it, given close enough initialization and small enough step size; A set is globally stable if the iterates eventually stay close to it for arbitrary initialization and small enough step size. These notions are suitable for the study of first-order methods with non-diminishing step sizes, as (subsequential) convergence to critical points can be unrealistic.

The results presented in this thesis rely on the premise that the objective function is tame (in an o-minimal structure on the real field) \cite{van1998tame,ioffe2009}. This assumption is crucial, as without it, the function may exhibit erratic behavior despite its local Lipschitz continuity  \cite{wang2005subdifferentiability}. Examples of such functions are can be found in \cref{ch:local_stability}.  The class of tame functions is both well suited for the study of first-order methods and wide enough to capture most applications. On one hand, they enjoy desirable properties, which we recall in \cref{sec:o-minimal}. On the other hand, seemingly all locally Lipschitz objective functions of interest nowadays are tame, including all those appearing in the statistical learning textbook \cite{friedman2001elements} by Friedman, Hastie, and Tibshirani.

This thesis is organized as follows. \cref{ch:pre} contains the preliminaries, where we recall the Clarke subdifferential, o-minimal structures, implementation of first-order methods, and review the existing literature. In \cref{ch:approx}, we show that iterates of certain first-order methods are approximated by subgradient trajectories, enabling a unified analysis for these methods in later chapters. In \cref{ch:local_stability}, we study the stability of local minima with respect to those first-order methods. We show that for a point to be stable, it is necessary for it to be a local minimum (\cref{thm:necessary}) and it suffices for it to be a strict local minimum (\cref{thm:sufficient_strict}).  In \cref{ch:global_stability}, we prove that the set of critical points is globally stable (\cref{thm:global_stability}), if the objective is additionally coercive. In \cref{ch:instability}, we study the subgradient method and propose sufficient conditions for local minima to be unstable (\cref{thm:suff_unstable}). These conditions can be verified in data science applications.


\titleformat{\chapter}[display]
{\normalfont\bfseries\filcenter}{}{0pt}{\large\chaptertitlename\ \large\thechapter : \large\bfseries\filcenter{#1}}  
\titlespacing*{\chapter}
  {0pt}{0pt}{30pt}	
  
\titleformat{\section}{\normalfont\bfseries}{\thesection}{1em}{#1}

\titleformat{\subsection}{\normalfont}{\thesubsection}{0em}{\hspace{1em}#1}



\chapter{Preliminaries}\label{ch:pre}
We begin by recalling some standard notations. Let $\mathbb{N}:=\{0,1,\ldots\}$ be the natural numbers. Let $\|\cdot\|$ be the induced norm of an inner product $\langle \cdot, \cdot\rangle$ on $\mathbb{R}^n$. Let $B(a,r)$ and $\mathring{B}(a,r)$ respectively denote the closed ball and the open ball of center $a\in \mathbb{R}^n$ and radius $r>0$. Let $B(S,r):= S + B(0,r)$ where $S \subset \mathbb{R}^n$. Given $x\in \mathbb{R}^n$ and $S \subset \mathbb{R}^n$, let $d(x,S) := \inf \{\|x - y\| : y \in S\}$ and $P_S(x):= \mathrm{argmin} \{y \in S:\|x - y\|\}$. 

\section{Clarke subdifferential}\label{sec:clarke}
 Recall that a function $f:\mathbb{R}^n\rightarrow\mathbb{R}^m$ is locally Lipschitz if for all $a \in \mathbb{R}^n$, there exist $r>0$ and $L>0$ such that 
$\|f(x)-f(y)\| \leqslant L \|x-y\|$ for all $x,y \in B(a,r)$. Since locally Lipschitz functions might not be differentiable everywhere, we consider the following notion of generalized gradients due to Clarke \cite{clarke1975}.

\begin{definition}{\cite{clarke1975} \cite[Chapter 2]{clarke1990}}
    \label{def:Clarke}
    Let $f:\mathbb{R}^n \rightarrow \mathbb{R}$ be a locally Lipschitz function. The Clarke subdifferential is the set-valued mapping $\partial f:\mathbb{R}^n\rightrightarrows\mathbb{R}^n$ defined for all $x \in \mathbb{R}^n$ by $\partial f(x) := \{ s \in \mathbb{R}^n : f^\circ(x,d) \geqslant \langle s , d \rangle, \forall d\in \mathbb{R}^n \}$ where
\begin{equation}\label{eq:gen_dir}
    f^\circ(x,d) := \limsup_{\tiny\begin{array}{c} y\rightarrow x \\
    t \searrow 0
    \end{array}
    } \frac{f(y+td)-f(y)}{t}.
\end{equation}
\end{definition}
  The Clarke subdifferential agrees with the convex subdifferential if $f$ is convex \cite[(1.2) Proposition]{clarke1975}, and agrees with the gradient if $f$ is continuously differentiable \cite[(1.13) Proposition]{clarke1975}. By Fermat's rule \cite[2.3.2 Proposition]{clarke1990}, $x\in \mathbb{R}^n$ is a local minimum of $f$ only if $0\in \partial f(x)$. We say that $x \in \mathbb{R}^n$ is a critical point of $f$ if $0 \in \partial f(x)$, and that $v \in \mathbb{R}$ is a critical value of $f$ if there exists $x\in \mathbb{R}^n$ such that $0 \in \partial f(x)$ and $v = f(x)$.
  
   A set-valued mapping $F:\mathbb{R}^n \rightrightarrows \mathbb{R}^n$ is upper semicontinuous if for all $x\in \mathbb{R}^n$ and $\epsilon>0$, there exists $\delta>0$ such that $F(B(x,\delta)) \subset B(F(x),\epsilon)$ \cite[p. 29]{clarke1990}.  We next recall some topological properties of the Clarke subdifferential.
\begin{proposition}\cite[Chapter 2]{clarke1990}\label{prop:clarke}
	Let $f:\mathbb{R}^n \rightarrow \mathbb{R}$ be a locally Lipschitz function and let $\partial f:\mathbb{R}^n \rightrightarrows \mathbb{R}^n$ be the Clarke subdifferential of $f$. Then it holds that
	\begin{enumerate}[label=(\alph*)]
		\item $\partial f(x)$ is a nonempty, convex, and compact subset of $\mathbb{R}^n$ for any $x\in \mathbb{R}^n$,
		\item $\partial f$ is upper semicontinuous.
	\end{enumerate}
\end{proposition}

In the remainder of this section, we define subgradient trajectories governed by the Clarke subdifferential. They play an important role for the analysis in this thesis (see \cref{ch:approx}). In order to do so, we recall the definition of absolutely continuous functions.
\begin{definition}
    \label{def:absolutely_continuous} 
    \cite[Definition 1 p. 12]{aubin1984differential}
    Given $a,b\in \mathbb{R}$ such that $a<b$, a function $x:[a,b] \rightarrow \mathbb{R}^n$ is absolutely continuous if for all $\epsilon>0$, there exists $\delta>0$ such that, for any finite collection of disjoint intervals $[a_1,b_1],\hdots,[a_m,b_m]$ of $[a,b]$ such that $\sum_{i=1}^m b_i-a_i \leqslant \delta$, we have $\sum_{i=1}^m \|x(b_i) - x(a_i)\| \leqslant \epsilon$.
\end{definition}
By virtue of \cite[Theorem 20.8]{nielsen1997introduction}, a function $x:[a,b]\rightarrow\mathbb{R}^n$ is absolutely continuous if and only if it is differentiable almost everywhere on $(a,b)$, its derivative $x'(\cdot)$ is Lebesgue integrable, and $\forall t\in [a,b], x(t) - x(a) = \int_a^t x'(t)dt$. Let $I$ be an interval in $\mathbb{R}$, we say that $x:I\rightarrow \mathbb{R}^n$ is absolutely continuous if it is absolutely continuous on any compact interval of $I$. We next define subgradient trajectories of locally Lipschitz functions.
\begin{definition}\label{def:subg_traj}
	Let $f:\mathbb{R}^n \rightarrow \mathbb{R}$ be a locally Lipschitz function and $I$ be an interval of $[0,\infty)$. We say that $x:I\rightarrow \mathbb{R}^n$ is a subgradient trajectory of $f$ if $x(\cdot)$ is absolutely continuous and $x'(t)\in -\partial f(x(t))$ for almost every $t\in I$.
\end{definition}
We defer a chain rule for subgradient trajectories of tame functions to the next section. 

\section{O-minimal structures}\label{sec:o-minimal}

O-minimal structures (short for order-minimal) were originally considered by van den Dries, Pillay and Steinhorn \cite{van1984remarks,pillay1986definable}. They are founded on the observation that many properties of semialgebraic sets can be deduced from a few simple axioms \cite{van1998tame}. Recall that a subset $A$ of $\mathbb{R}^n$ is semialgebraic \cite{bochnak2013real} if it is a finite union of basic semialgebraic sets, which are of the form $\{ x \in \mathbb{R}^n : p_i(x) = 0, ~ i = 1,\hdots,k; ~  p_i(x) > 0, ~ i = k+1,\hdots,m \}$ where $p_1,\hdots,p_m \in \mathbb{R}[X_1,\hdots,X_n]$ (i.e., polynomials with real coefficients). 

\begin{definition}\cite[Definition p. 503-506]{van1996geometric}
\label{def:o-minimal}
An o-minimal structure on the real field is a sequence $S = (S_k)_{k \in \mathbb{N}}$ such that for all $k \in \mathbb{N}$:
\begin{enumerate}
    \item $S_k$ is a boolean algebra of subsets of $\mathbb{R}^k$, with $\mathbb{R}^k \in S_k$;
    \item $S_k$ contains the diagonal $\{(x_1,\hdots,x_k) \in \mathbb{R}^k : x_i = x_j\}$ for $1\leqslant i<j \leqslant k$;
\item If $A\in S_k$, then $A\times \mathbb{R}$ and $\mathbb{R}\times A$ belong to $S_{k+1}$;
    \item If $A \in S_{k+1}$ and $\pi:\mathbb{R}^{k+1}\rightarrow\mathbb{R}^k$ is the projection onto the first $k$ coordinates, then $\pi(A) \in S_k$;
    \item $S_3$ contains the graphs of addition and multiplication;
    \item $S_1$ consists exactly of the finite unions of open intervals and singletons. 
\end{enumerate}
\end{definition}

Note that $S_1$ are the semialgebraic subsets of $\mathbb{R}$ and by \cite[2.5 Examples (3)]{van1996geometric}, $S_k$ contains the semialgebraic subsets of $\mathbb{R}^k$. A subset $A$ of $\mathbb{R}^n$ is definable in an o-minimal structure $(S_k)_{k\in\mathbb{N}}$ if $A \in S_n$. Let $A \subset \mathbb{R}^n$ and $B\subset \mathbb{R}^m$, a function $f:A\rightarrow B$ is definable in an o-minimal structure if its graph, that is to say $\{(x,t) \in A \times B : f(x)=t \}$, is definable in that structure. A set $C \subset \mathbb{R}^n$ is tame \cite{ioffe2009} in an o-minimal structure $(S_k)_{k\in \mathbb{N}}$ if
    $$
    \forall x\in \mathbb{R}^{n}, ~ \forall r>0, ~~~ C \cap B(x,r) \in S_{n},
    $$
    and a function $f:A\rightarrow B$ is tame if its graph is tame. Throughout this thesis, we fix an arbitrary o-minimal structure on the real field, and say that the sets or functions are tame (resp. definable) if they are tame (resp. definable) in this structure.
    
 Tame functions enjoy nice properties that enable the analysis in this thesis. In the remainder of this section, we recall several such results. We first state the \L{}ojasiewicz inequality for continuous tame functions, generalized from the inequality for semialgebraic functions \cite[\S 2]{lojasiewicz1958} (see also \cite[\S 17]{lojasiewicz1959}, \cite[(2.1)]{hormander1958division}).
 \begin{theorem}\label{thm:l}
 	\cite[Theorem 0]{kurdyka1998gradients} Let $A$ be a compact subset of $\mathbb{R}^n$. Let $f,g:A \rightarrow \mathbb{R}$ be continuous tame functions. If $f^{-1}(0) \subset g^{-1}(0)$, then there exists a strictly increasing definable function $\sigma:[0,\infty) \rightarrow [0,\infty)$ that is continuously differentiable on $(0,\infty)$ with $\sigma(0) = 0$ such that for any $x\in A$, we have $|f(x)| \geqslant \sigma(|g(x)|)$.
 \end{theorem}
 
 Recall that the Kurdyka-\L{}ojasiewicz inequality \cite[Theorem 1]{kurdyka1998gradients} lower bounds the gradient norm of continuously differentiable definable functions, after precomposing a desingularizing function. We next state a version of Kurdyka-\L{}ojasiewicz inequality for locally Lipschitz tame functions, which is a consequence of the projection formula of Clarke subdifferentials of definable functions \cite[Corollary 9]{bolte2007clarke}. 
 \begin{theorem}\label{thm:kl}
 	\cite[Theorem 14]{bolte2007clarke}  Let $f:\mathbb{R}^n\rightarrow \mathbb{R}$ be a locally Lipschitz tame function. Then for any $r>0$, there exist $\rho>0$ and a strictly increasing definable function $\psi:[0,\rho)\rightarrow[0,\infty)$ that is continuously differentiable on $(0,\rho)$ with $ \psi(0) = 0$ such that $\psi$ is concave and
 	\begin{equation*}
 		d(0,\partial \left(\psi \circ |f|\right)(x)) \geqslant 1
 	\end{equation*}
 	for any $x\in B(0,r)$ such that $0<|f(x)|<\rho$.
 \end{theorem}
The fact that $\psi$ can be taken to be concave is due to the monotonicity theorem \cite[(1.2) p. 43]{van1998tame} and was observed in \cite{attouch2010proximal}. A uniform version of the inequality was proposed recently in \cite{josz2023global} for proving the global convergence of the gradient method. \cref{thm:kl} implies that a locally Lipschitz tame function has finitely many critical values over any bounded set. The definable Morse-Sard theorem shows that in fact definable functions admit finitely many critical values over their domains \cite[Corollary 9]{bolte2007clarke}.

Finally, we recall a chain rule for the subgradient trajectories of locally Lipschitz tame functions. The same property was studied for semialgebraic functions in \cite{drusvyatskiy2015curves}.

\begin{proposition}\label{prop:chain} \cite[Lemma 5.2]{davis2020stochastic} Let $f:\mathbb{R}^n\rightarrow \mathbb{R}$ be a locally Lipschitz tame function and $I$ be an interval of $[0,\infty)$. If $x:I \rightarrow \mathbb{R}$ is a subgradient trajectory of $f$, then $f\circ x$ is differentiable almost everywhere on $I$ with
        \begin{equation*}
            \left(f\circ x\right)'(t) = -\|x'(t)\|^2 \quad\text{and}\quad \|x'(t)\| = d\left(0,\partial f(x(t))\right). 
        \end{equation*}
	\end{proposition}
	
An immediate consequence of the chain rule is that the function value decreases along any subgradient trajectory, until a critical point is reached.

\section{First-order methods}\label{sec:first-order}
In this thesis, we consider unconstrained optimization problems with possible composite structures, namely

\begin{equation}
\label{eq:obj}
	\inf\limits_{x\in \mathbb{R}^n} f(x) := \frac{1}{N}\sum\limits_{i= 1}^N f_i(x)
\end{equation}
where $f_i:\mathbb{R}^n \rightarrow \mathbb{R}$ is locally Lipschitz for $i = 1,2, \ldots, N$. Such problems are central in machine learning applications such as empirical risk minimization \cite{bottou2018optimization}, low-rank matrix recovery \cite{li2019,ma2022global,zhang2022accelerating}, and the training of deep neural networks \cite{lecun2015deep}.  As mentioned in the introduction, we study first-order methods with constant step sizes for solving \eqref{eq:obj}. While \cref{ch:instability} focuses on the subgradient method, other parts of this thesis consider a general class of first-order methods that satisfy \cref{def:approx_flow_new}. We outline several examples of such methods here. 

 The subgradient method (\cref{alg:sg}) is one of the most classical first-order methods, which is a nonsmooth adaptation of the steepest descent due to Cauchy \cite{cauchy}. The subgradient method with momentum (\cref{alg:mag}) is a generalization of the framework proposed in the work of Kovachki and Stuart \cite[(7)]{kovachki2021continuous} from differentiable functions to locally Lipschitz functions. \cref{alg:mag} reduces to the heavy ball method \cite{polyak1964some} when $\gamma = 0$ and to Nesterov's momentum method \cite[equation (2.2.22)]{nesterov2018introductory} when $\beta = \gamma$. It also includes the vanilla subgradient method as a special case when $\beta = \gamma= 0$. The random reshuffling with momentum (\cref{alg:rrm}) is an extension of \cref{alg:mag} which exploits the composite nature of the objective function \eqref{eq:obj}. Its update is the same as \cref{alg:mag} except that each step concerns only one component $f_i$, which is chosen at a random order at every iteration (epoch). This is exactly how stochastic subgradient method with momentum is implemented in practice (see for e.g., documentations from TensorFlow\footnote{\url{https://www.tensorflow.org/api\_docs/python/tf/keras/optimizers/SGD}}, PyTorch\footnote{\url{https://pytorch.org/docs/stable/generated/torch.optim.SGD.html}} and scikit-learn\footnote{\url{https://scikit-learn.org/stable/modules/sgd.html}}). Last, Algorithm \ref{alg:rcd} is the random-permutations cyclic coordinate descent method, where $\nabla_i f(x):= [\nabla f(x)]_i e_i$, $[\nabla f(x)]_i$ is the $i$\textsuperscript{th} entry of $\nabla f(x)$, and $e_i$ is the $i$\textsuperscript{th} vector in the canonical basis of $\mathbb{R}^n$. Similar to \cref{alg:rrm}, Algorithm \ref{alg:rcd} chooses a permutation of all the coordinates at every iteration and cycles through them.
\begin{algorithm}[ht]
\caption{Subgradient method}\label{alg:sg}
\begin{algorithmic}
\STATE{\textbf{choose} step size $\alpha>0$ and $x_0 \in \mathbb{R}^n$}
\FOR{$k = 0,1, \ldots$}
    \STATE{ $x_{k+1}  \in x_k  - \alpha \partial f (x_k)$ } 
\ENDFOR
\end{algorithmic}
\end{algorithm}

\begin{algorithm}[ht]
\caption{Subgradient method with momentum}\label{alg:mag}
\begin{algorithmic}
\STATE{\textbf{choose} step size $\alpha>0$, momentum parameters $\beta \in (-1,1)$, $\gamma \in \mathbb{R}$, constant $\delta>0$, $x_{-1},x_0 \in \mathbb{R}^n$ with $\|x_{-1}-x_0\|\leqslant \delta \alpha$}
\FOR{$k = 0,1, \ldots$}
    \STATE{ $y_k = x_k + \gamma (x_k - x_{k-1})$}
    \STATE{ $x_{k+1}  \in x_k + \beta (x_k - x_{k-1})  - \alpha \partial f (y_k)$ } 
\ENDFOR
\end{algorithmic}
\end{algorithm}

\begin{algorithm}[!ht]
\caption{Random reshuffling with momentum}\label{alg:rrm}
\begin{algorithmic}
\STATE{\textbf{choose} step size $\alpha>0$, momentum parameters $\beta \in (-1,1)$, $\gamma \in \mathbb{R}$, constant $\delta>0$, $x_{-1,N-1}, x_0 \in \mathbb{R}^n$ with $\|x_{-1,N-1}- x_0\| \leqslant \delta \alpha$}
\FOR{$k = 0,1, \ldots$}
    \STATE{ $x_{k,0} = x_k$}
    \STATE{$x_{k,-1} = x_{k-1,N-1}$}
    \STATE{choose a permutation $\sigma^k$ of $\{1,2,\ldots, N\}$}
    \FOR{$i = 1,2, \ldots, N$}
        \STATE{ $y_{k,i} = x_{k,i-1} + \gamma (x_{k,i-1} - x_{k,i-2})$}
        \STATE{ $x_{k,i} \in x_{k,i-1} + \beta (x_{k,i-1} - x_{k,i-2})  - \alpha \partial f_{\sigma^k_i} (y_{k,i})$}
    \ENDFOR
    \STATE{ $x_{k+1} = x_{k,N}$}
\ENDFOR
\end{algorithmic}
\end{algorithm}

\begin{algorithm}[ht]
\caption{Random-permutations cyclic coordinate descent method}\label{alg:rcd}
\begin{algorithmic}
\STATE{\textbf{choose}  $x_0 \in \mathbb{R}^n$, step size $\alpha>0$}
\FOR{$k = 0,1, \ldots$}
    \STATE{choose a permutation $\sigma^k$ of $\{1,2,\ldots, n\}$}
    \STATE{ $x_{k,0} = x_k$}
    \FOR{$i = 1,2, \ldots, n$}
        \STATE{ $x_{k,i} = x_{k,i-1} - \alpha \nabla_{\sigma^k_i} f(x_{k,i-1})$}
    \ENDFOR
    \STATE{ $x_{k+1} = x_{k,n}$}
\ENDFOR
\end{algorithmic}
\end{algorithm}

\section{Literature review}\label{sec:literature}
We next review the existing analysis of \cref{alg:sg,alg:mag,alg:rrm,alg:rcd}. The subgradient method was proposed by Shor in 1961 for minimizing piecewise linear convex functions that arise when taking the dual of network transportation problems \cite{shor1962application,shor1964structure}. If $f$ is convex and the Euclidean norm of its subgradients is bounded above by a constant $c$, then $\liminf f(x_k) -\inf f \leqslant c^2\alpha/2$ provided that the infimum is reached \cite[Proposition 3.2.3]{bertsekas2015convex}. In order to get within $\epsilon$ accuracy of that bound, $\lfloor d(x_0,X)^2/(\alpha \epsilon) \rfloor$ iterations suffice where $\lfloor \cdot \rfloor$ denotes the floor of a real number and $d(x_0,X)$ is the distance between the initial iterate $x_0$ and the set of minimizers $X \subset \mathbb{R}^n$ \cite[Proposition 3.2.4]{bertsekas2015convex}. If the objective function grows quadratically (at least as fast as $t\in \mathbb{R} \mapsto\beta t^2$ for some $\beta>0$) around the set of minimizers, then the iterates asymptotically get within $c\sqrt{\alpha}/\sqrt{2\beta}$ distance to the set of minimizers if $\alpha \in (0,1/(2\beta)]$ \cite[Proposition 3.2.5]{bertsekas2015convex}. If we relax the boundedness assumption on the subgradients to $\|s\| \leqslant c \sqrt{1+d(x,X)^2}$ for all $(x,s)$ in the graph of $\partial f$, then we get the slightly weaker bound $\liminf f(x_k) -\inf f \leqslant c^2\alpha/2(1+d(x_0,X))$ \cite[Exercise 3.6]{bertsekas2015convex}.

 The gradient method with momentum with $\gamma = 0$ was introduced by Polyak \cite{polyak1964some}. It admits a nearly optimal local convergence rate for twice continuously differentiable strongly convex functions \cite[Theorem 9]{polyak1964some}. Nesterov showed that it admits a globally optimal convergence rate \cite[Theorem 2.1.13]{nesterov2018introductory} if one chooses $\beta = \gamma$ in an appropriate manner. With variable momentum parameters, it also has an optimal rate for convex functions with Lipschitz gradients whose infimum is attained \cite{nesterov1983method}. If one relaxes the convexity assumption, then with a suitable choice of parameters $\alpha,\beta,$ and $\gamma$, the gradients $\nabla f(x_k)$ converge to zero \cite[Lemmas 1,2,3]{zavriev1993heavy} for any initial points $x_{-1},x_0 \in \mathbb{R}^n$. If in addition $f$ is coercive and satisfies the Kurdyka-\L{}ojasiewicz inequality \cite{kurdyka1998gradients} at every point and $x_{-1}=x_0$, then the iterates have finite length \cite[Theorem 4.9]{ochs2014ipiano}. In the nonsmooth setting that we consider in this thesis, there seems to be no results to the best of our knowledge.

The incremental subgradient method is a special case of the stochastic subgradient method with random reshuffling where the components are visited in a fixed order. It can be traced back to the Widrow-Hoff least mean squares method \cite{widrow1960adaptive} for minimizing a finite sum of convex quadratics in 1960. It was pointed out later by Kohonen that with sufficiently small constant step sizes, the limit points of the iterates of the least mean squares method are close to a minimum of the objective function \cite{kohonen1974adaptive}. With diminishing step sizes that are not summable but square summable, the least mean squares method converges to a minimum of the problem \cite{luo1991convergence}. For convex objectives, the incremental subgradient method with constant step size $\alpha>0$ satisfies $\liminf_{k \rightarrow \infty} f(x_k) \leqslant \inf_{ \mathbb{R}^n} f + C\alpha$ for some $C>0$ \cite[Proposition 2.1]{nedic2001incremental}, provided that $\inf_{\mathbb{R}^n} f>-\infty$ and the subgradients of the components $f_i$ are uniformly bounded. We refer the readers to the survey paper \cite{bertsekas2011incremental}, the textbook \cite{bertsekas2015convex}, and references therein for a more detailed discussion on the subject.

The stochastic subgradient method with random reshuffling is a stochastic version of the incremental subgradient method. It was shown recently that the stochastic gradient method with random reshuffling outperforms the incremental gradient method in expectation on strongly convex functions with quadratic components \cite[Theorem 2]{gurbuzbalaban2021random}, under certain choices of diminishing step sizes. If the objective function is strongly convex and differentiable with Lipschitz gradients among other assumptions, then the iterates and the corresponding function values of the stochastic gradient method with random reshuffling and constant step size eventually lie in a neighborhood of the minimizer \cite[Theorem 1]{mishchenko2020random} and a neighborhood of the minimum \cite[Theorem 1]{nguyen2021unified} respectively, both in expectation. By relaxing the strong convexity assumption to mere convexity, the function values evaluated at the average iterates $\hat{x}_k := (\sum_{l = 0}^k x_l)/k$ eventually lie in a neighborhood of the minimum in expectation \cite[Theorem 3]{mishchenko2020random} \cite[Remark 1]{nguyen2021unified}. By further removing the convexity assumption, the minimum norm of the gradients eventually lies in a neighborhood of zero in expectation \cite[Theorem 4]{mishchenko2020random} \cite[Corollary 1, Corollary 3]{pauwels2021incremental} (see also \cite[Theorem 4]{nguyen2021unified} for a similar result). If in addition assuming a certain Kurdyka-\L{}ojasiewicz property, bounded iterates of random reshuffling with certain diminishing step sizes converge to critical points \cite{li2023convergence}. The long-term behavior of the iterates for nonconvex and nonsmooth objective functions has so far remained elusive.

Despite the empirical success of incorporating momentum into random reshuffling \cite{sutskever2013importance}, the theoretical understanding of such methods is limited. So far, the only guarantees available are for modified versions \cite{tran2021smg,tran2022nesterov}. The work of Tran \textit{et al.} \cite{tran2021smg} in 2021 studied a modified version of stochastic gradient method with random reshuffling and heavy ball. The momentum is constant within every iteration (epoch) and is equal to the average of the gradients evaluated in the previous epoch. With the modification, the norm of gradients of the average iterates $\hat{x}_k$ eventually lie in a neighborhood of zero in expectation \cite[Corollary 1]{tran2021smg}, under various assumptions \cite[Assumption 1]{tran2021smg}. A modified stochastic gradient method with random reshuffling and Nesterov's momentum was studied recently \cite{tran2022nesterov}. The momentum is only applied at the level of the outer loop, at the end of each iteration (epoch). In this setting, the function values eventually lie a neighborhood of the minimum when the component functions are convex \cite[Theorem 1]{tran2022nesterov}, among other assumptions.

Coordinate descent methods are the object of the survey paper \cite{wright2015coordinate} by Wright in 2015. The idea of coordinate descent methods is to optimize with respect to one variable at a time. It was first studied under the framework of univariate relaxation \cite[Section 14.6]{ortega1970}. 
With exact line search and almost cyclic rule or Gauss-Southwell rule for cycling over the coordinates, the coordinate descent method converges linearly to a minimizer of a strongly convex objective that is twice differentiable \cite[Theorem 2.1]{luo1992convergence}. More recently, global convergence of random coordinate descent method was established for convex objectives with Lipschitz continuous partial derivatives \cite{nesterov2012efficiency}. In contrast to cyclic coordinate descent methods, random coordinate descent methods choose a coordinate randomly at each iteration instead of following a cycling rule. Similar to the stochastic subgradient method with random reshuffling, the random-permutations cyclic coordinate descent method (\cref{alg:rcd}) considered in this thesis is easier to implement than the random coordinate descent method as it requires only sequential access of the data \cite{gurbuzbalaban2020randomness}. Using \cite[Lemma 3.3, remark 3.2]{beck2013convergence}, the convergence of the random-permutations cyclic coordinate descent method can be deduced for coercive functions with locally Lipschitz gradients. The superior performance of the random-permutations cyclic coordinate descent method was observed in numerical experiments, and was supported by analysis for convex quadratic objectives \cite{lee2019random,gurbuzbalaban2020randomness}. For objective functions without a locally Lipschitz gradient, the study of the method appears to be absent from the literature.


\chapter{Approximation of first-order methods by subgradient trajectories}
\label{ch:approx}
In this chapter, we study the connection between the iterates of first-order methods and the subgradient trajectories. In order to do so, we propose a notion of approximation that can be shown to hold for many first-order methods, under method-dependent regularity assumptions. This approximation is central for characterizing local and global stability in the subsequent chapters. The material of this chapter is based on the following articles:\vspace*{3mm}

\noindent C\' edric Josz, Lexiao Lai, Global stability of first-order methods for coercive tame functions,\emph{ Mathematical Programming}, 2023 [\href{https://arxiv.org/abs/2308.00899}{preprint}] [\href{https://doi.org/10.1007/s10107-023-02020-9}{journal doi}]
 \vspace*{3mm}
 
 \noindent C\' edric Josz, Lexiao Lai, Xiaopeng Li, Proximal random reshuffling under local Lipschitz continuity,\emph{ arXiv preprint}, 2024 [\href{https://arxiv.org/abs/2408.07182}{preprint}] \vspace*{3mm}

We refer to an iterative method with constant step size as a set-valued mapping $\mathcal{M}: \mathbb{R}^{(\mathbb{R}^n)} \times (0,\infty) \times 2^{(\mathbb{R}^n)} \times \mathbb{N}\rightrightarrows (\mathbb{R}^n)^\mathbb{N}$ which, to an objective function $f:\mathbb{R}^n \rightarrow \mathbb{R}$, a constant step size $\alpha \in (0,\infty)$, a set $X_0\subset \mathbb{R}^n$, and a natural number $k_0$ associates a set of sequences in $\mathbb{R}^n$ whose $k_0$\textsuperscript{th} term is contained in $X_0$. Many optimization algorithms can be viewed as such set-valued mappings, including \cref{alg:sg,alg:mag,alg:rrm,alg:rcd}.

We define the precise meaning of approximation in \cref{def:approx_flow_new}. This definition is inspired by a series of works that resort to continuous-time dynamics to analyze discrete-time dynamics. The idea that discrete dynamics resemble their continuous counterpart dates back to Euler \cite{euler1792institutiones,blanton2006foundations}. He proposed discretizing ordinary differential equations to find approximate solutions. This technique is also used to prove the existence of solutions via the Cauchy-Peano theorem \cite[Theorem 1.2]{coddington1955theory}. Ljung \cite{ljung1977analysis} and Kushner \cite{kushner1977general,kushner1977convergence} established a connection between the asymptotic behavior of discrete and continuous dynamics with noise, which is particularly useful when they are governed by conservative fields. Benaïm \textit{et al.} \cite{benaim2005stochastic,benaim2006stochastic,benaim2006dynamics} strengthened this connection by relaxing some assumptions and incorporating set-valued dynamics. Due to its importance in analyzing optimization algorithms in recent years \cite{borkar2009stochastic,duchi2018stochastic,davis2020stochastic,bolte2020conservative,salim2018random,pauwels2021incremental,xiao2024adam}, we elaborate on their contribution.

Benaïm, Hofbauer, and Sorin consider a closed set-valued mapping $F:\mathbb{R}^n \rightrightarrows \mathbb{R}^n$ with nonempty convex compact values for which there exists $C>0$ such that $\sup \{ \|s\| : s \in F(x)\} \leqslant C(1+\|x\|)$ for all $x\in \mathbb{R}^n$. They show that discrete trajectories of $F$ can be approximated by its continuous trajectories in the following sense. Let $(x_k)_{k\in \mathbb{N}}$ be a bounded sequence such that $x_{k+1} \in x_k + \alpha_k F(x_k)$ for all $k\in\mathbb{N}$ where $\alpha_k>0$, $\sum_{k=0}^\infty \alpha_k = \infty$, and $\sum_{k=0}^\infty \alpha_k^2 < \infty$ (Ljung and Kushner also assume this, following Robbins and Monro \cite{robbins1951stochastic}). Let $t_0 := 0$ and $t_k := \alpha_0+\hdots+\alpha_{k-1}$ for $k\geqslant 1$. Consider the linear interpolation defined by 
\begin{equation*}
    x(t) := x_k + \frac{t-t_k}{t_{k+1}-t_k}(x_{k+1}-x_k), ~~~\forall t\in[t_k,t_{k+1})
\end{equation*}
as well as the time shifted interpolations $x^{\tau}(\cdot) := x(\tau + \cdot)$ where $\tau\geqslant 0$. The key insight is that for any sequence $\tau_k\rightarrow \infty$, the shifted interpolations $x^{\tau_k}(\cdot)$ subsequentially converge to a solution to the differential inclusion $x'(t) \in F(x(t))$ for almost every $t>0$ in the topology of uniform convergence on compact intervals \cite[Theorem 4.2]{benaim2005stochastic}. When one specifies that $F:= -\partial f$ where $f:\mathbb{R}^n\rightarrow\mathbb{R}$ is a locally Lipschitz function, this can be used to derive asymptotic properties of some important algorithms in optimization.

We are interested in the constant step size regime for which it is hopeless to try to establish uniform convergence over compact intervals (for a fixed discrete trajectory, as above). We thus ask for something weaker from the algorithms we analyze, namely, that the continuous and discrete time dynamics are close in uniform norm up to any given time. We ask that this holds for a set of time shifted trajectories to account for multistep methods. We next give a precise meaning to this notion. We will use $\lfloor t \rfloor$ to denote the floor of a real number $t$ which is the unique integer such that $\lfloor t \rfloor \leqslant t < \lfloor t \rfloor + 1$.

\begin{definition}
\label{def:approx_flow_new}
An iterative method $\mathcal{M}$ is approximated by subgradient trajectories of a locally Lipschitz function $f:\mathbb{R}^n \rightarrow \mathbb{R}$ (up to a positive multiplicative constant) if there exists $c>0$ such that for any compact sets $X_0, X_1 \subset \mathbb{R}^n$ and for any $\epsilon,T>0$, there exists $\bar{\alpha}>0$ such that for all $\alpha \in (0,\bar{\alpha}]$, $k_0 \in \mathbb{N}$, and $(x_k)_{k \in \mathbb{N}} \in \mathcal{M}(f,\alpha,X_0,k_0)$ for which $x_0, \ldots, x_{k_0} \in X_1$, there exists an absolutely continuous function $x:[0,T]\rightarrow \mathbb{R}^n$ such that
\begin{equation}
    \label{eq:ivp}
    x'(t) \in -c\partial f(x(t)),~~~\text{for almost every}~t\in [0,T],~~~x(0)\in X_0,
\end{equation}
and $\|x_{k}-x((k-k_0)\alpha)\|\leqslant \epsilon$ for $k=k_0,\hdots,k_0+\lfloor T/\alpha\rfloor$.
\end{definition}

 \begin{table}[ht]
\caption{Algorithms that satisfy \cref{def:approx_flow_new} with the corresponding scalings. The standing assumption is that $f$ is both tame and lower bounded.}
\vspace*{-3mm}
\begin{center}
\begin{tabular}{|c|c|c|} 
\hline
Algorithm & Assumption & Scaling \\
\hline
\multirow{3}{5cm}{Subgradient method with momentum} & \multirow{3}{5cm}{$f$ locally Lipschitz} & 
\multirow{3}{3cm}{$c = 1/(1-\beta)$} 
\\ 
&  &\\ 
&  &\\ 
\cline{1-3}
\multirow{3}{5cm}{Random reshuffling with momentum} & \multirow{3}{5cm}{$f_i$ locally Lipschitz and subdifferentially regular\tablefootnote{A locally Lipschitz function is subdifferentially regular \cite[2.3.4 Definition]{clarke1990} if its generalized directional derivative \eqref{eq:gen_dir} \cite[p. 25]{clarke1990} agrees with the classical directional derivative.}} & \multirow{3}{3cm}{$c = N/(1-\beta)$}\\ 
& &\\ 
& &\\ 
\cline{1-3}
\multirow{3}{5cm}{Random-permutations cyclic coordinate descent method} & \multirow{3}{5cm}{$f$ continuously differentiable} & \multirow{3}{3cm}{$c = 1$}\\ 
& &\\ 
& &\\ 
\hline
\end{tabular}
\end{center}\label{tab:appr_alg}
\end{table}
\cref{def:approx_flow_new} is different from \cite[Definition 3]{josz2023globalstability} as it asks the approximation to hold over any given time interval, instead of over a certain one. While this was unnecessary in the context of \cite{josz2023globalstability} (see also \cref{ch:global_stability}) for establishing global stability guarantees for coercive tame functions, it becomes handy in studying local stability (see \cref{ch:local_stability}). Despite the fact that we need the approximation to hold in a stronger sense, \cref{def:approx_flow_new} continues to be satisfied by the algorithms considered in \cite{josz2023globalstability}, with the additional assumption that the objective function is tame and lower bounded.
 In the remainder of this chapter, we prove the subgradient method (\cref{alg:sg}) satisfies \cref{def:approx_flow_new}. Using the techniques in \cite{josz2023globalstability}, the same proof can be adapted to show that \cref{alg:mag,alg:rrm,alg:rcd} are also approximated by subgradient trajectories under method-dependent regularity assumptions. We summarize them in Table \ref{tab:appr_alg}.
 
We next state the main theorem of this chapter.
\begin{theorem}\label{thm:subg_approx}
	The subgradient method is approximated by subgradient trajectories of any locally Lipschitz tame function that is lower bounded.
\end{theorem}

In order to prove \cref{thm:subg_approx}, we need to first show the uniform boundedness of iterates for any given $T>0$, which is the object of the following proposition.
\begin{proposition}
\label{prop:bounded}
    Let $f:\mathbb{R}^n\rightarrow \mathbb{R}$ be a locally Lipschitz tame that is lower bounded. Then for any bounded set $X_0 \subset \mathbb{R}^n$ and $T>0$, there exist $\bar{\alpha}, r>0$ such that for all iterates $(x_k)_{k\in \mathbb{N}}$ generated by the subgradient method with constant step size $\alpha \in (0, \bar{\alpha}]$ and $x_0\in X_0$, we have that $x_0, \ldots, x_{\lfloor T/\alpha\rfloor+1} \in B(0,r)$
\end{proposition}
\begin{proof}
   We assume without loss of generality that $X_0 \subset \mathbb{R}^n$ is nonempty and compact. Fix $T>0$. There exists $Q>0$ such that for any subgradient trajectory $x:[0,T]\rightarrow \mathbb{R}^n$ of $f$ with $x_0 \in X_0$,
it holds that $d(x(\tilde{T}),X_0) \leqslant \|x(\tilde{T})-x(0)\| \leqslant \int_0^{\tilde{T}} \|x'(t)\|~dt \leqslant \int_0^{T} \|x'(t)\|~dt \leqslant Q$ for all $\tilde{T} \in [0,T]$. Indeed, we have
\begin{subequations}
\label{eq:cont_len}
    \begin{align}
        \int_{0}^T \|x'(t)\|~dt &\leqslant \sqrt{T} \sqrt{\int_{0}^T \|x'(t)\|^2~dt}\label{eq:cont_len_a} \\
        &= \sqrt{T} \sqrt{\int_{0}^T -(f\circ x)'(t)~dt} \label{eq:cont_len_b}\\
        &= \sqrt{T} \sqrt{f(x(0)) - f(x(T))} \label{eq:cont_len_c}\\
        &\leqslant \sqrt{T} \sqrt{\sup_{x_0\in X_0}f(x_0) - \inf_{y\in \mathbb{R}^n}f(y)}=: Q. \label{eq:cont_len_d}
    \end{align}
\end{subequations}
Above, \eqref{eq:cont_len_a} follows from the Cauchy-Schwarz inequality and \eqref{eq:cont_len_b} follows from the chain rule of subgradient trajectories (\cref{prop:chain}).

We next reason by contradiction and assume that for any $r>0$, there exist a positive sequence $\alpha_m \rightarrow 0$ and sequences of iterates  $(x_k^m)_{k\in \mathbb{N}}$ generated by the subgradient method with constant step size $\alpha_m$ and $x_0^m \in X_0$ such that $\max\{\|x_k^{m}\|:k = 0, \ldots,\lfloor T/\alpha_m\rfloor+1\}>r$ for any $m\in \mathbb{N}$. Let $R>0$ such that $X_0 \subset B(0,R)$. Recall that by \cref{prop:clarke}, $\partial f$ is upper semicontinuous with nonempty, compact and convex values. Thus, there exists $L>0$ such that $\partial f(B(0,R+Q+1)) \subset B(0,L)$ by \cite[Proposition 3 p. 42]{aubin1984differential}. Take $r = R+Q+1$ and we may assume that $\bar{\alpha}_m \leqslant \min\{T, 1/(2L)\}$ without loss of generality. For each $m\in \mathbb{N}$, let $k_m := \min\{k\in \mathbb{N}:x_{k}^m \in B(0,r), x_{k+1}^m \not\in B(0,r)\}$, we have that $k_m \leqslant \lfloor T/\alpha_m\rfloor$ following our assumption.  Thus $ k_m\alpha_m \in [0,T]$ for any $m \in \mathbb{N}$ and $\bar{T}:= \liminf_{m \rightarrow \infty} k_m\alpha_m \in [0, T]$. By taking a subsequence if necessary, we assume that $\lim_{m \rightarrow \infty} k_m\alpha_m = \bar{T}$.

For each sequence $x_0^m, x_1^m, \ldots, x_{k_m}^m$, consider the (extended) linear interpolation $\bar{x}^m(\cdot)$ defined from $[0,\max\{ k_m\alpha_m,\bar{T}\}]$ to $\mathbb{R}^n$ by
	\begin{equation*}
		\bar{x}^m(t) = x_k^m + (t -  k\alpha_m) \frac{x_{k+1}^m - x_{k}^m}{\alpha_m}
	\end{equation*}
	for any $t \in [ k\alpha_m, \alpha_m (k+1)]$ and $k \in\{ 0,1, \ldots, k_m - 1\}$. Also, $\bar{x}^m(t) = x_{k_m}^m$ for $t \in [ k_m\alpha_m, \bar{T}]$ if $ k_m\alpha_m < \bar{T}$. As $x_k^m \in B(0,r)$ for $k = 0,\ldots, k_m$, we know that $\bar{x}^m(t) \in B(0,r)$ for all $t\in [0,\max\{ k_m\alpha_m,\bar{T}\}]$ by the convexity of $B(0,r)$. For any $t \in ( k\alpha_m, \alpha_m (k+1))$ and $k \in\{ 0,1, \ldots k_m - 1\}$, it holds that $(\bar{x}^m)'(t) = (x_{k+1}^m - x_{k}^m)/\alpha_m \in -\partial f(x_k^m) \subset -\partial f(B(0,r)) \subset B(0,L)$. Also, $(\bar{x}^m)'(t) = 0$ for any $t \in ( k_m\alpha_m, \bar{T})$ if $ k_m\alpha_m < \bar{T}$. By successively applying the Arzel\`a-Ascoli and the Banach-Alaoglu theorems (see \cite[Theorem 4 p. 13]{aubin1984differential}), there exists a subsequence of $(\bar{x}^m(\cdot))_{m\in \mathbb{N}}$ (again denoted $(\bar{x}^m(\cdot))_{m\in \mathbb{N}}$) and an absolutely continuous function $x:[0,\bar{T}]\rightarrow \mathbb{R}^n$ such that $\bar{x}_{|[0,\bar{T}]}^{m}(\cdot)$ converges uniformly to $x(\cdot)$ and $(\bar{x}_{|[0,\bar{T}]}^{m})'(\cdot)$ converges weakly to $x'(\cdot)$ in $L^1([0,\bar{T}],\mathbb{R}^n)$. In addition, for almost every $t \in (0,\bar{T})$, since $ k_m\alpha_m \rightarrow \bar{T}$, it holds that   $t\in ( k\alpha_m, (k+1)\alpha_m)$ for any sufficiently large $m$ and some $k \in\{ 0,1, \ldots k_m - 1\}$. Fix any such $t$, for all sufficiently large $m$, it holds that
\begin{subequations}
\begin{align*}
    (\bar{x}^{m}(t),(\bar{x}^{m})'(t)) 
    & = \left(x_k^{m} + (t- k\alpha_m)\frac{x_{k+1}^{m}-x_k^{m}}{\alpha_m},\frac{x_{k+1}^{m}-x_k^{m}}{\alpha_m}\right) \\
    & \in \left(\{x_k^{m}\} - (t- k\alpha_m)\partial f(x_k^{m})\right) \times \left(-\partial f(x_k^{m})\right)  \\
    & = \{x_k^{m}\} \times \left(-\partial f(x_k^{m})\right)  + ( k\alpha_m-t)\partial f(x_k^{m}) \times \{0\}  \\
    & \subset \mathrm{graph}(-\partial f) + B(0,L\alpha_m)\times \{0\}.
\end{align*}
\end{subequations}
According to \cite[Convergence Theorem p. 60]{aubin1984differential}, it follows that $x'(t) \in -\partial f(x(t))$ for almost every $t\in (0,\bar{T})$. Thus, $x(\cdot)$ is a subgradient trajectory of $f$. The sequence of initial points $(x^{m}(0))_{m\in\mathbb{N}}$ lies in the compact set $X_0$, hence its limit $x(0)$ lies in $X_0$ as well. Therefore, $\int_{0}^{\bar{T}} \|x'(t)\|~dt \leqslant Q$, and thus $x(\bar{T}) \in B(X_0,Q) \subset B(0,R+Q)$. Also, it holds that $\lim_{m\rightarrow \infty} \bar{x}^m(\bar{T}) = x(\bar{T}) \in B(0,R+Q)$ and $\|\bar{x}^m(\bar{T}) - x_{k_m}^m\| = \|\bar{x}^m(\bar{T}) - \bar{x}^m( k_m\alpha_m)\| \leqslant L|\bar{T} -  k_m\alpha_m| \rightarrow 0$ as $m\rightarrow \infty$. Thus, $x_{k_m}^m \in B(0,R+Q+0.5)$ for all sufficiently large $m$. Meanwhile, we have that $\|x_{k_m+1}^m  - x_{k_m}^m\| \leqslant \alpha_m L \leqslant 0.5$, contradicting the assumption that $x_{k_m+1}^m \not\in B(0,R+Q+1)$.
\end{proof}

We are now ready to prove \cref{thm:subg_approx}.
\begin{proof}[Proof of \cref{thm:subg_approx}]
Since for any iterates $(x_k)_{k \in \mathbb{N}}$ generated by the subgradient method and $k_0 \geqslant 1$, $(x_k)_{k\geqslant k_0}$ is unrelated to $x_0,\ldots, x_{k_0 - 1}$ if given $x_{k_0}$, it suffices to show that the subgradient method satisfies \cref{def:approx_flow_new} with $k_0 := 0$ and $X_1 := X_0$.

 Let $f:\mathbb{R}^n\rightarrow \mathbb{R}$ be a locally Lipschitz tame function that is lower bounded. Without loss of generality, let  $X_0\subset \mathbb{R}^n$ be nonempty and compact and $T>0$.  Invoking \cref{prop:bounded}, there exist $\bar{\alpha}, r>0$ such that for all iterates $(x_k)_{k\in \mathbb{N}}$ generated by the subgradient method with constant step size $\alpha \in (0, \bar{\alpha}]$ and $x_0\in X_0$, we have that $x_0, \ldots, x_{\lfloor T/\alpha\rfloor+1} \in B(0,r)$. Let $(\alpha_m)_{m \in \mathbb{N}}$ denote a sequence of positive numbers that converges to zero. Without loss of generality, we may assume that the sequence is bounded above by $\bar{\alpha}$. To each term in the sequence, we attribute a sequence $(x_k^{m})_{k\in\mathbb{N}}$ generated by the subgradient method with step size $\alpha_m$ and initialized in $X_0$. Thus, $x_0^{m},\hdots,x_{\lfloor T/\alpha_m \rfloor+1}^{m} \in B(0,r)$. Consider the linear interpolation of those iterates, that is to say, the function $\bar{x}^{m}:[0,T]\rightarrow\mathbb{R}^n$ defined by $\bar{x}^{m}(t) := x_k^{m} + (t- k\alpha_m)(x_{k+1}^{m}-x_k^{m})/\alpha_m$ for all $t\in [ k\alpha_m , \min\{(k+1)\alpha_m,T\}]$ and $k\in \{0,\hdots,\lfloor T/\alpha_m \rfloor\}$. Following similar arguments as in the proof of \cref{prop:bounded}, a subsequence of $(\bar{x}^m(\cdot))_{m\in\mathbb{N}}$ converges uniformly to a subgradient trajectory of $f$ that initialized in $X_0$.

The conclusion of the theorem now follows. To see why, one can reason by contradiction and assume that there exists $\epsilon>0$ such that for all $\hat{\alpha} \in (0,\bar{\alpha}]$, there exist $\alpha \in (0,\hat{\alpha}]$ and a sequence $(x_k)_{k \in \mathbb{N}}$ generated by the subgradient method with step size $\alpha$ and initialized in $X_0$ such that, for any subgradient trajectory $x(\cdot)$ of $f$ initialized in $X_0$, it holds that $\|x_k - x(k\alpha) \| > \epsilon$ for some $k\in \{0,\ldots, \lfloor 0,T/\alpha \rfloor\}$. We can then generate a sequence $(\alpha_m)_{m \in \mathbb{N}}$ of positive numbers converging to zero with attributed sequences $(x_k^m)_{k\in \mathbb{N}}$ generated by the subgradient method with step size $\alpha_m$ such that, for any subgradient trajectory $x(\cdot)$ of $f$ initialized in $X_0$, it holds that $\|x_k - x(k\alpha) \| > \epsilon$ for some $k\in \{0,\ldots, \lfloor 0,T/\alpha \rfloor\}$. We obtain a contradiction with the conclusion of the previous paragraph.
\end{proof}

\begin{remark}
\label{remark:dontchev}

When proving both \cref{prop:bounded} and \cref{thm:subg_approx}, we use the same strategy which consists in taking sequences generated by the subgradient method with smaller and smaller constant step size, and show that a subsequence of their linear interpolations converges uniformly to a subgradient trajectory up to a finite time. The same strategy can be used to show the approximation of other first-order methods \cite[Section 4]{josz2023globalstability}.

Several discretization methods of initial value problems with differential inclusions were studied in \cite{taubert1981converging,aubin1984differential,clarke1990,filippov2013differential} (see also a survey on the subject by Dontchev and Lempio \cite{dontchev1992difference}). Assume that the set-valued mapping underlying the differential inclusion is upper semicontinuous with nonempty compact convex values, such that the norm of their elements are upper bounded by a linear function of the norm of the argument. Then over any finite time horizon, a subsequence of linear interpolations of the Euler method with smaller and smaller step sizes converges uniformly to a solution to the initial value problem \cite[Theorem 2.2]{dontchev1992difference}. If in addition the set-valued mapping is bounded, then a class of linear multistep methods has the same convergence property as above \cite[p. 127, Theorem]{taubert1981converging} (see also \cite[Convergence Theorem 3.2]{dontchev1992difference}). We build on the techniques developed in the above works when checking Definition \ref{def:approx_flow_new}. We adapt them so that they can handle the case where the set of initial points is a compact set and the case where $\partial f$
is not accessible (as in \cref{alg:rrm,alg:rcd}).

\end{remark}


\chapter{Local stability of first-order methods}\label{ch:local_stability}
In this chapter, we introduce a notion of discrete Lyapunov stability and propose necessary and sufficient conditions for stability, in order to analyze the behavior of first-order methods in the vicinity of a local minimum.  The material of this chapter is based on the following article:\vspace*{3mm}

\noindent C\' edric Josz, Lexiao Lai, Lyapunov stability of the subgradient method with constant step size, \emph{Mathematical Programming}, 2023 [\href{https://arxiv.org/abs/2211.14850}{preprint}] [\href{https://doi.org/10.1007/s10107-023-01936-6}{journal doi}]
 \vspace*{3mm}
 
 We consider the first-order methods with constant step size for minimizing a locally Lipschitz function $f:\mathbb{R}^n \rightarrow \mathbb{R}$. While these methods are often used in practice to solve nonconvex and nonsmooth problems, there is little theoretical understanding of the behavior of the iterates, as reviewed in \cref{sec:literature}. Given the absence of theoretical results, in this chapter we take a first step by investigating the behavior of first-order methods in the vicinity of a local minimum of $f$. In order to do so, we propose a notion of stability akin to Lyapunov stability in dynamical systems \cite{liapounoff1907probleme} \cite[Equation (5.6)]{sastry2013nonlinear}. Recall from \cref{ch:approx} that we may denote an iterative method by a set-valued mapping  $\mathcal{M}$. For notational convenience, let $f:\mathbb{R}^n \rightarrow \mathbb{R}$, $\alpha \in (0,\infty)$, and $X_0\subset \mathbb{R}^n$, we denote by $\mathcal{M}(f,\alpha,X_0):= \mathcal{M}(f,\alpha,X_0,0)$ the set of sequences generated by the method with constant step size $\alpha$ and initialized in $X_0$.

\begin{definition}
\label{def:discrete_lyapunov}
We say that $x^*\in \mathbb{R}^n$ is an $\mathcal{M}$-stable point of a locally Lipschitz function $f:\mathbb{R}^n\rightarrow\mathbb{R}$ if for all $\epsilon>0$, there exist $\delta>0$ and $\bar{\alpha}>0$ such that
\begin{equation*}
	(x_k)_{k\in \mathbb{N}} \in \mathcal{M}\left(f,(0,\bar{\alpha}],B(x^*,\delta)\right)\implies \{x_k\}_{k\in \mathbb{N}}\subset B(x^*,\epsilon).
\end{equation*}
\end{definition}

Informally, a point is $\mathcal{M}$-stable if all of the iterates of the iterative method $\mathcal{M}$ remain in any neighborhood of it, provided that the initial point is close enough to it and that the step size is small enough. Denote the subgradient method (\cref{alg:sg}) by $\mathcal{M}_S$, we illustrate the notion of stability in the following example.

\begin{example}[Local minima of sharp and weakly convex functions are stable]
\label{eg:sharpness}
Consider $x^* \in \mathbb{R}^n$ and a locally Lipschitz function $f:\mathbb{R}^n\rightarrow \mathbb{R}$ such that, for all $x \in \mathbb{R}^n$ and $s \in \partial f(x)$, we have
\begin{equation}
\label{eq:sharp_weak}
  \mu \|x-x^*\| \leqslant  f(x)-f(x^*) \leqslant \langle x - x^* , s \rangle + \frac{\rho}{2} \|x-x^*\|^2
\end{equation}
for some positive constants $\mu$ and $\rho$. In the literature \cite{davis2018subgradient,davis2019stochastic,charisopoulos2021low}, the first inequality in \eqref{eq:sharp_weak} is referred to as sharpness \cite[Assumption A 2]{davis2018subgradient}, while the second inequality in \eqref{eq:sharp_weak} is referred to as weak convexity \cite[Assumption A 1]{davis2018subgradient}. For example, the absolute value function is sharp and weakly convex.

 In order to prove that $x^*$ is $\mathcal{M}_S$-stable, it suffices to prove the statement in \cref{def:discrete_lyapunov} for all $\epsilon>0$ sufficiently small. We may thus restrict ourselves to the case where $0<\epsilon<2\mu/\rho$. Given such a fixed $\epsilon$, let us choose $\delta := \epsilon$ and $\bar{\alpha} := \min\{ \delta/L , \delta (2\mu-\rho \delta)/L^2 \}$ where $L$ is a Lipschitz constant of $f$ on $B(0,2\epsilon)$. If $\alpha \in (0,\bar{\alpha}]$ and $x_0 \in B(x^*,\delta)$, then 
\begin{subequations}
    \begin{align}
        \|x_1-x^*\|^2 & = \|x_0 -\alpha s_0 - x^*\|^2 ~~~~~(\text{where}~s_0 \in \partial f(x_0))\\
        & = \|x_0-x^*\|^2 - 2\alpha \langle x_0-x^* , s_0 \rangle + \alpha^2\|s_0\|^2 \\
        & \leqslant \|x_0-x^*\|^2 - 2\alpha \left(\mu \|x_0-x^*\|-\frac{\rho}{2} \|x_0-x^*\|^2\right) + \alpha^2 L^2 \\
        & = (1+\alpha \rho)\|x_0-x^*\|^2 - 2\alpha \mu \|x_0-x^*\| + \alpha^2 L^2 \label{eq:convex} \\
        & \leqslant \delta^2.
    \end{align}
\end{subequations}
The last inequality holds because the expression in \eqref{eq:convex} is a convex function of $\|x_0-x^*\|$ and $\alpha$. It is thus bounded above by its evaluations at $(0,0)$, $(0,\bar{\alpha})$, $(\delta,0)$ and $(\delta,\bar{\alpha})$.
By induction, it follows that all the iterates lie in $B(x^*,\delta) = B(x^*,\epsilon)$, which completes the proof of stability. 
\end{example}

We remark that the inequalities \eqref{eq:sharp_weak} may not hold in practice \cite[p. 121]{davis2020stochastic} \cite[2.3.6 Proposition]{clarke1990}, and can be difficult to check \cite[Conjecture 8.7]{charisopoulos2021low}. While we attempt to characterize stability for locally Lipschitz functions, \cref{eg:stable_not_local,eg:local_not_stable} below show that  stability and local optimality are actually decorrelated at such generality.
 This is perhaps not surprising since it is also the case for the gradient trajectories of infinitely differentiable functions \cite[Proposition 2]{absil2006stable}.

\begin{example}[Stability $\centernot\implies$ local optimality]
\label{eg:stable_not_local}
Let $\langle \cdot,\cdot \rangle$ be the Euclidean inner product on $\mathbb{R}$. $0$ is an $\mathcal{M}_S$-stable point of $f:\mathbb{R}^n \rightarrow \mathbb{R}$ defined by
\begin{equation}
   f(x) := \left\{
   \begin{array}{ll}
   x^2\sin(1/x) & \text{if}~ x\neq 0, \\
   0 & \text{else}.
   \end{array}
   \right.
\end{equation}
However, $0$ is not a local minimum of $f$. In Figure \ref{fig:stable_critical}, we apply the subgradient method with constant step size $\alpha = 0.01$ and initialized at $x = -0.01$. One can see that that first 100 iterates remain in a neighborhood of the origin. For a proof of stability, see Section \ref{subsec:Proof of stability in Example}.
\end{example}

\begin{figure}[ht]
\centering
\hspace*{-10mm}
  \includegraphics[width=.6\textwidth]{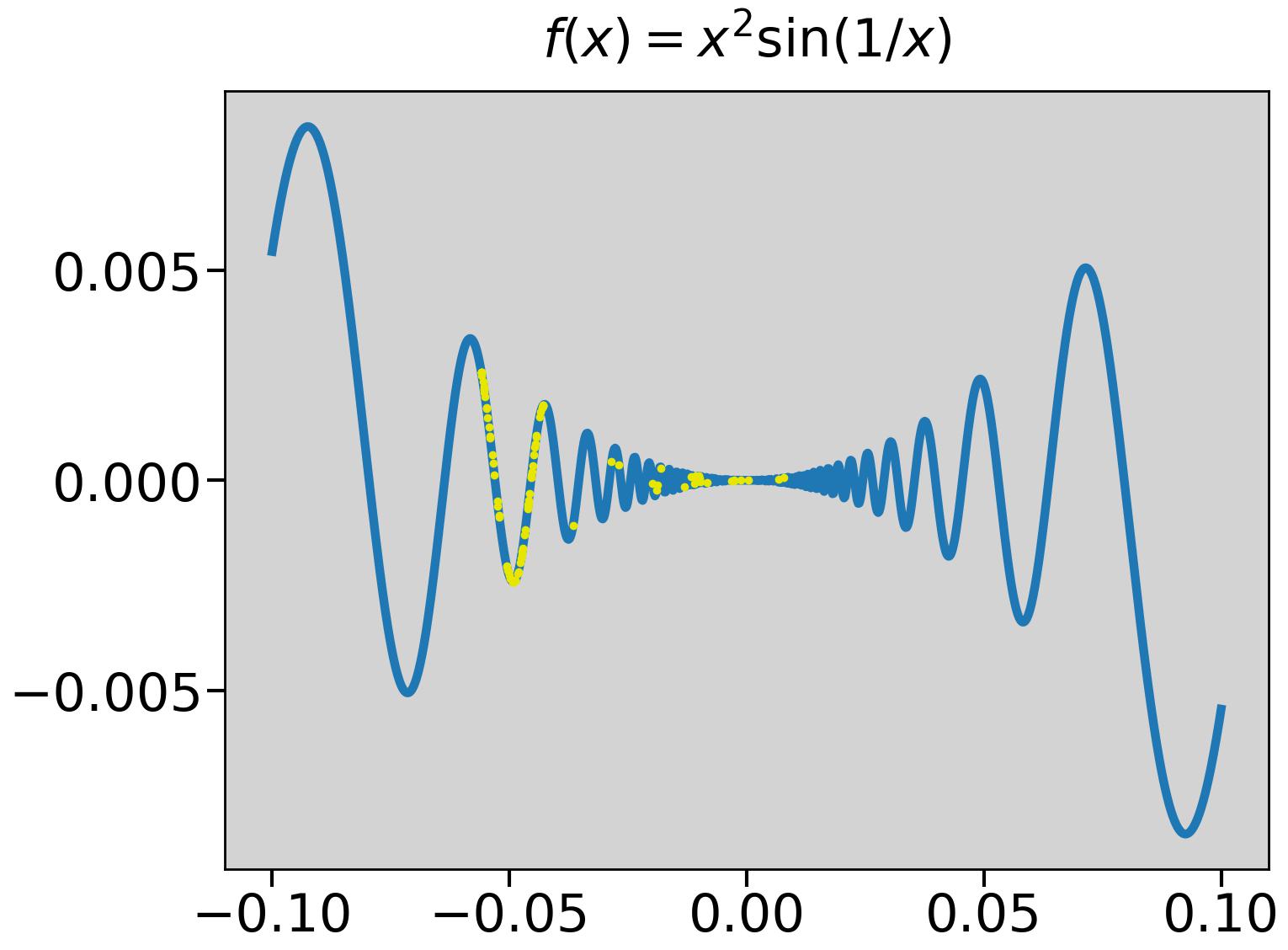}
\caption{Iterates of the subgradient method with constant step size.}
\label{fig:stable_critical}
\end{figure}
 Recall that a point $x^* \in \mathbb{R}^n$ is a strict local minimum of a function $f:\mathbb{R}^n\rightarrow\mathbb{R}$ if there exists a positive constant $\epsilon$ such that $f(x^*) < f(x)$ for all $x \in B(x^*,\epsilon)\setminus \{x^*\}$.
 
\begin{example}[Local optimality $\centernot\implies$ stability]
\label{eg:local_not_stable}
Consider the function $f:\mathbb{R} \rightarrow \mathbb{R}$ due to Rockafellar \cite{rockafellar1981favorable} (see also \cite[Proposition (1.9)]{lebourg1979generic}) defined by 
\begin{equation}
\label{eq:rockafellar}
    f(x) := \int_0^x g(t)dt~~~\text{where}~~~ g(t):= \left\{\begin{array}{rl}
        1 & \text{if}~ t\in A, \\
        -1 & \text{if}~ t \notin A.
    \end{array}\right.
\end{equation}
Above, $A$ is a measurable subset of $\mathbb{R}$ such that the measure of $A\cap I$ and $A\setminus I$ are positive for all non-empty open interval $I$. See \cref{fig:rock} for illustration of such a function. Since the function is nowhere monotonic, the set of local minima is dense in the real line \cite[Lemma 3.3]{wang2005subdifferentiability}. Also, the function is Lipschitz on the real line and $\partial f(x) = [-1,1]$ for all $x\in \mathbb{R}$ according to \cite[Equation (1.12)]{rockafellar1981favorable}. Any local minimum is not $\mathcal{M}_S$-stable since the update rule of the subgradient method with constant step size $\alpha>0$ is given by $x_{k+1} \in x_k + [-\alpha,\alpha]$ for $k = 0,1,2,\hdots$ Furthermore, given a local minimum $x^*$ of $f$, consider $\hat{f}:\mathbb{R} \rightarrow \mathbb{R}$ defined by $\hat{f}(x) := f(x) + |x-x^*|/2$. Clearly, $x^*$ is a strict local minimum of $\hat{f}$. Also, we have that $\partial \hat{f}(x) = [-3/2,1/2]$ for all $x<x^*$, $\partial \hat{f}(x) = [-1/2,3/2]$ for all $x>x^*$ 
due to \cite[p. 39, Corollary 1]{clarke1990}. The strict local minimum $x^*$ is thus not $\mathcal{M}_S$-stable.
\end{example}

\begin{figure}[ht]
\centering
  \includegraphics[width=.6\textwidth]{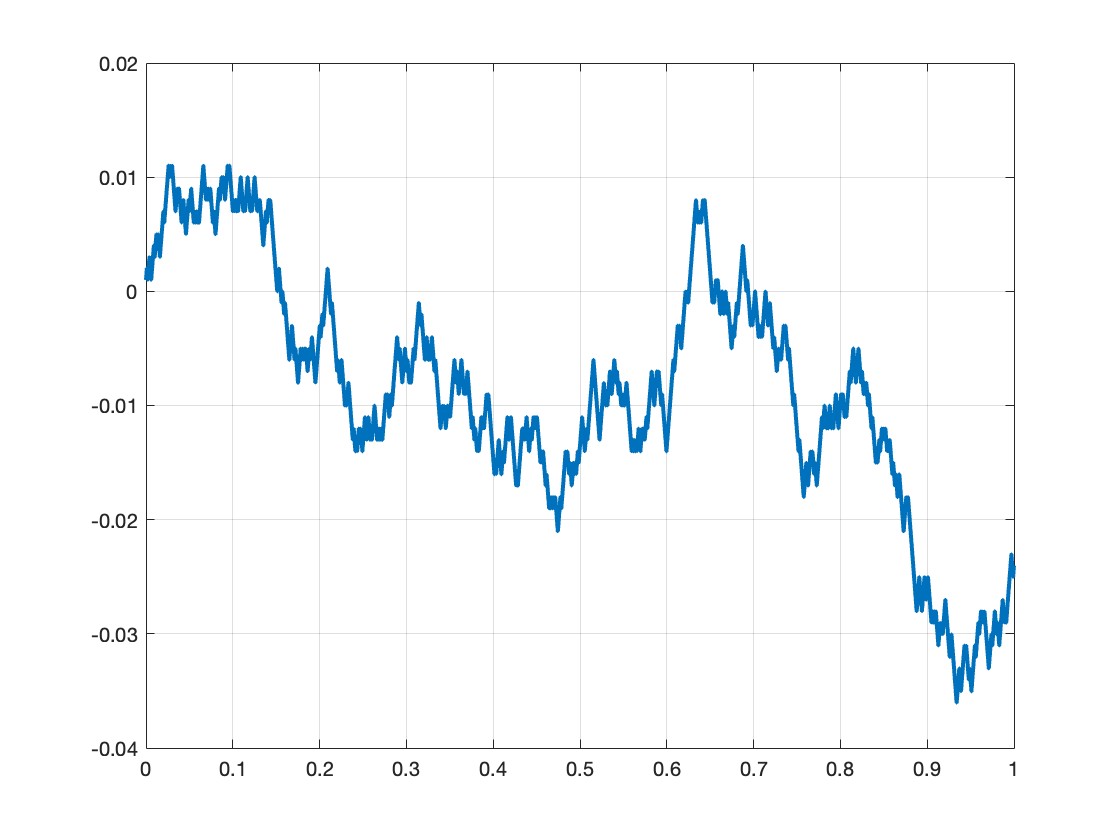}
\caption{Illustration of a Rockafellar function.}
\label{fig:rock}
\end{figure}

The two examples above call for additional structures on the objective function in order to characterize stability of local minima.  We thus confine our investigation to locally Lipschitz tame functions. When $f$ is locally Lipschitz and tame, the set of stable points and local minima coincide in two cases. The first is when $f$ is continuously differentiable with a locally Lipschitz gradient. The fact that local minima are $\mathcal{M}_S$-stable in this regime can be deduced using arguments from \cite[Proposition 3.3]{absil2005convergence}. The converse is a consequence of one of our results (Theorem \ref{thm:necessary}). The second case is that of subgradient trajectories (\cref{def:subg_traj}), where instead of iterates one considers absolutely continuous solutions to the differential inclusion $x' \in -\partial f(x)$. This is a simple generalization of \cite[Theorem 3]{absil2006stable} which holds for real analytic functions. Much of modern numerical optimization however falls outside the scope of these two cases, namely that of smooth objective functions and continuous-time dynamics. It is thus important to determine in the discrete case and when $f$ is merely locally Lipschitz and tame, whether local minima are stable, and conversely, whether stable points are local minima.

As discussed in \cref{ch:approx}, many first-order methods satisfy \cref{def:approx_flow_new}, namely they are approximated by subgradient trajectories of $f$ under mild assumptions. In the remainder of this chapter, we aim to characterize stability for these methods. We say that a point $x^*$ is a stable point of a locally Lipschitz function $f:\mathbb{R}^n\rightarrow \mathbb{R}$ if it is an $\mathcal{M}$-stable point of $f$, for any iterative method $\mathcal{M}$ approximated by subgradient trajectories of $f$. We show that for a point to be stable, it is necessary for it to be a local minimum (\cref{thm:necessary}) and it suffices for it to be a strict local minimum (\cref{thm:sufficient_strict}). 


\begin{theorem}
\label{thm:necessary}
Stable points of locally Lipschitz tame functions are local minima.
\end{theorem}

\begin{proof}
Let $\mathcal{M}$ be an iterative method approximated by subgradient trajectories of $f$ (\cref{def:approx_flow_new}) with $c>0$. We in fact show that if 
 $x^* \in \mathbb{R}^n$ is  $\mathcal{M}$-stable point of a locally Lipschitz tame function $f:\mathbb{R}^n\rightarrow\mathbb{R}$, then it is a local minimum of $f$. We reason by contradiction and assume that $x^*$ is not a local minimum of $f$. According to the Kurdyka-\L{}ojasiewicz inequality (\cref{thm:kl}), there exist $r,\rho>0$ and a strictly increasing concave continuous definable function $\psi:[0,\rho) \rightarrow [0,\infty)$ that is continuously differentiable on $(0,\rho)$ with $\psi(0) = 0$ such that $d(0,\partial f(x)) \geqslant 1/\psi'(|f(x) - f(x^*)|)$ for all $x \in B(x^*,r)$ whenever $0<|f(x)-f(x^*)|< \rho$. After possibly reducing $r$, the inequality holds for all $x \in B(x^*,r)$ such that $f(x) \neq f(x^*)$.

Let $\epsilon \in (0,r/2)$ and $L>0$ be a Lipschitz constant of $f$ on $B(x^*,2r)$. By the definition of stability (\cref{def:discrete_lyapunov}), there exist $\delta>0$ and $\bar{\alpha} \in (0,\min\{1,r/(4cL)\}]$ such that $\{x_k\}_{k\in \mathbb{N}}\subset B(x^*,\epsilon)$ for any $(x_k)_{k\in \mathbb{N}} \in \mathcal{M}(f,(0,\bar{\alpha}],B(x^*,\delta))$. Since $x^*$ is not a local minimum, there exists $x_0 \in B(x^*,\delta)$ such that $f(x_0)<f(x^*)$. Let $T_0:= (f(x_0) - \min_{B(x^*,2\epsilon)} f)\psi'(f(x^*) - f(x_0))^2/c\geqslant 0$ and $T:= T_0 + 1$. Let $(x_k)_{k\in \mathbb{N}} \in \mathcal{M}(f,(0,\bar{\alpha}],x_0)$. By \cref{def:approx_flow_new}, there exists an absolutely continuous function $x:[0,T]\rightarrow \mathbb{R}^n$ such that
\begin{equation*}
    x'(t) \in -c\partial f(x(t)),~~~\text{for almost every}~t\in [0,T],~~~x(0) = x_0,
\end{equation*}
and $\|x_{k}-x(k\alpha)\|\leqslant \epsilon/2$ for $k=0,\hdots,\lfloor T/\alpha\rfloor$, after possibly reducing $\bar{\alpha}$. It follows that $\|x(k\alpha) - x^*\| \leqslant \|x(k\alpha) - x_k\|+ \|x_k - x^*\| \leqslant \epsilon/2 + \epsilon = 3\epsilon/2 < 3r/4$ for $k=0,\hdots,\lfloor T/\alpha\rfloor$. Thus, $x([0,T]) \subset B(x^*,r)$. Indeed, assume the contrary that $x(\tilde{t}) \not\in B(x^*,r)$ for some $\tilde{t}\in (k\alpha,\min\{(k+1)\alpha,T\})$ and $k$. Then by the continuity of $x(\cdot)$, there exists $t' \in (k\alpha,\tilde{t})$ such that $x([k\alpha,t']) \subset B(x^*,r)$ and $\|x(t') - x^*\| = r$. Therefore, $ r/4 =  r - 3r/4< \|x(t') - x^*\|- \|x(k\alpha) - x^*\| \leqslant \|x(t') - x(k\alpha)\| \leqslant \int_{k\alpha}^{t'}\|x'(t)\| dt \leqslant cL(t' - k\alpha) \leqslant cL \alpha \leqslant r/4$, where the last inequality follows from the fact that $\alpha \leqslant \bar{\alpha} \leqslant r/(4cL)$. Contradiction occurs.

 Since $f$ is locally Lipschitz and tame, by the chain rule of subgradient trajectories (\cref{prop:chain}), it holds that
\begin{equation}
\label{eq:lyapunov}
    f(x(t)) - f(x(0)) = - \int_0^t cd(0,\partial f(x(\tau)))^2 d\tau, ~~~ \forall t \in [0,T].
\end{equation}
Therefore, $ f(x(\alpha \lfloor T/\alpha \rfloor)) - f(x(0)) = - \int_0^{\alpha \lfloor T/\alpha \rfloor} cd(0,\partial f(x(\tau)))^2 d\tau \leqslant - \int_0^{\alpha \lfloor T/\alpha \rfloor} c/\psi'(f(x^*) - f(x(\tau)))^2 d\tau \leqslant -\alpha \lfloor T/\alpha \rfloor c/\psi'(f(x^*) - f(x(0)))^2 < -T_0 c/\psi'(f(x^*) - f(x(0)))^2 = \min_{B(x^*,2\epsilon)} f - f(x(0))$. Above, the first inequality is due to the Kurdyka-\L{}ojasiewicz inequality, and the second inequality follows from \eqref{eq:lyapunov} and the concavity of $\psi$. Thus, $x(\alpha \lfloor T/\alpha \rfloor) \not\in B(x^*,2\epsilon)$ and $\|x_{\lfloor T/\alpha \rfloor} - x^*\| \geqslant \|x(\alpha \lfloor T/\alpha \rfloor) - x^*\| - \|x(\alpha \lfloor T/\alpha \rfloor) - x_{\lfloor T/\alpha \rfloor}\| > 2\epsilon - \epsilon/2 > \epsilon$, which contradicts with the stability of $x^*$.
\end{proof}

We also obtain the following sufficient condition for stability.

\begin{theorem}
\label{thm:sufficient_strict}
Strict local minima of locally Lipschitz tame functions are stable.
\end{theorem}
\begin{proof}
Let $\mathcal{M}$ be an iterative method approximated by subgradient trajectories of  a locally Lipschitz tame function $f:\mathbb{R}^n\rightarrow\mathbb{R}$ (\cref{def:approx_flow_new}) with $c>0$. Let $x^*\in \mathbb{R}^n$ be a strict local minimum of $f$. By the \L{}ojasiewicz inequality (\cref{thm:l}) and the Kurdyka-\L{}ojasiewicz inequality (\cref{thm:kl}), there exist $r, \rho>0$ and strictly increasing continuous definable functions $\sigma,\psi:[0,\rho)\rightarrow[0,\infty)$ that are continuously differentiable on $(0,\rho)$ with $\sigma(0) = \psi(0) = 0$ such that $\psi$ is concave, $f(x)-f(x^*) \geqslant  \sigma(\|x-x^*\|)$, and $d(0,\partial f(x)) \geqslant 1/\psi'(f(x) - f(x^*))$ for all $x \in B(x^*,r)$ whenever $0<f(x)-f(x^*)< \rho$. After possibly reducing $r$, the two inequalities above hold for all $x \in B(x^*,r)\setminus \{x^*\}$. In order to prove stability, it suffices to prove the statement in \cref{def:discrete_lyapunov} for all $\epsilon>0$ sufficiently small. We may thus restrict ourselves to the case where $0<\epsilon < r$. Given such a fixed $\epsilon$, we next describe a possible choice for $\delta$. 

Given $\Delta \geqslant f(x^*)$, let $[f\leqslant \Delta]_{x^*}$ denote the connected component of the sublevel set $[f\leqslant \Delta] := \{ x \in \mathbb{R}^n : f(x) \leqslant \Delta \}$ containing $x^*$. By taking $\Delta_\epsilon:= f(x^*) + \sigma(\epsilon/2)$, we find that $[f\leqslant \Delta_\epsilon]_{x^*}$ is contained in $B(x^*,\epsilon/2)$. Indeed, one can reason by contradiction and assume that there exists $x \in [f\leqslant \Delta_\epsilon]_{x^*} \setminus B(x^*,\epsilon/2)$. Then $x\notin B(x^*,r)$, otherwise $\sigma(\|x-x^*\|) \leqslant f(x)-f(x^*) \leqslant \sigma(\epsilon/2)$ and thus $\|x-x^*\| \leqslant \epsilon/2$. Therefore $[f\leqslant \Delta_\epsilon]_{x^*}$ is the disjoint union of $[f\leqslant \Delta_\epsilon]_{x^*} \cap \mathring{B}(x^*,r)$ and $[f\leqslant \Delta_\epsilon]_{x^*} \setminus B(x^*,r)$, both of which are nonempty and open in $[f\leqslant \Delta_\epsilon]_{x^*}$. This contradicts the connectedness of $[f\leqslant \Delta_\epsilon]_{x^*}$, which yields that $[f\leqslant \Delta_\epsilon]_{x^*} \subset B(x^*,\epsilon/2)$. By continuity of $f$ we may choose $\delta>0$ such that 

\begin{equation}
\label{eq:inclusions}
    B(x^*,\delta) ~ \subset ~ [f\leqslant \Delta_\epsilon]_{x^*} ~ \subset ~ B(x^*,\epsilon/2).
\end{equation}

We next describe a possible choice for $\bar{\alpha}$. Let $L>\sup\{\|s\|:s\in \partial f(x), x\in B(x^*,\epsilon)\}$ be a Lipschitz constant of $f$ on $B(x^*,\epsilon)$ and let $T := \epsilon/(3L)$, where the supremum is finite due to \cref{prop:clarke} and \cite[Proposition 3 p. 42]{aubin1984differential}. Since $f$ is continuous, the set of initial values $[f\leqslant \Delta_\epsilon]_{x^*}$ is closed. By virtue of the second inclusion in \eqref{eq:inclusions}, $[f\leqslant \Delta_\epsilon]_{x^*}$ is in fact a compact set. Let $\epsilon' := \min\{\epsilon L,\sigma(\epsilon/2),c\xi^2 T\}/(2L)>0$ where $\xi:= 1/\psi'(\sigma(\epsilon/2)/2)>0$.
Applying \cref{def:approx_flow_new}, there exists $\bar{\alpha} \in (0,T/2]$ such that, for any $(x_k)_{k\in\mathbb{N}} \in \mathcal{M}(f,(0,\bar{\alpha}],[f\leqslant \Delta_\epsilon]_{x^*})$, there exists a solution to 
\begin{equation}
\label{eq:ivp:levelset}
     x'(t) \in - c\partial f(x(t)), ~~~ \text{for a.e.}~ 
     t \in (0,T), ~~~ x(0) \in [f\leqslant \Delta_\epsilon]_{x^*}.
\end{equation}
 for which  it holds that
\begin{equation}
\label{eq:uniform_bound_iterates}
    \|x_k - x(k\alpha)\| \leqslant \epsilon',~~~ k = 0,\hdots,\lfloor T/\alpha \rfloor.
\end{equation}

Having chosen $\delta$ and $\bar{\alpha}$, let us fix some $\alpha \in (0,\bar{\alpha}]$ from now on. Consider a sequence $(x_k)_{k\in\mathbb{N}}$ generated by the subgradient method with constant step size $\alpha$ and initialized in $B(x^*,\delta)$. Our goal is to show that all the iterates lie in $B(x^*,\epsilon)$. We first show that $x_0,\hdots,x_K$ lie in $B(x^*,\epsilon)$ where $K:=\lfloor T/\alpha \rfloor$. Since $f$ is locally Lipschitz and tame, by the chain rule of subgradient trajectories (\cref{prop:chain}), it holds that
\begin{equation}
\label{eq:dist}
    f(x(t)) - f(x(0)) = - \int_0^t c d(0,\partial f(x(\tau)))^2 d\tau, ~~~ \forall t \in [0,T].
\end{equation}
As a result, $f(x(t)) \leqslant f(x(0)) \leqslant \Delta_\epsilon$ for all $t \in [0,T]$. We thus have that $x(0) \in [f\leqslant \Delta_\epsilon]_{x^*}\cap x([0,T])$, both of which are connected subsets of the sublevel set $[f\leqslant \Delta_\epsilon]$. By maximality of $[f\leqslant \Delta_\epsilon]_{x^*}$, it follows that $x([0,T]) \subset [f\leqslant \Delta_\epsilon]_{x^*} \subset B(x^*,\epsilon/2)$. Therefore, $\|x_k - x^*\| \leqslant \|x_k - x(k\alpha)\| + \|x(k\alpha) - x^*\| \leqslant \epsilon' + \epsilon/2 \leqslant \epsilon$ for $k = 0, \ldots,K$.

 In order to show that the ensuing iterates also lie in $B(x^*,\epsilon)$, we will show that $x_K \in [f\leqslant \Delta_\epsilon]_{x^*}$. The same argument used previously then yields that $x_{K+1},\hdots,x_{2K} \in B(x^*,\epsilon)$. Since $K = \lfloor T/\alpha \rfloor \geqslant \lfloor T/\bar{\alpha}  \rfloor \geqslant 2$, we may conclude by induction that all the iterates belong to $B(x^*,\epsilon)$. This is illustrated in Figure \ref{fig:inclusions}. 

\begin{figure}[ht]
\centering
    \begin{tikzpicture}[scale=2.4]
    \draw[blue,opacity=0.2] (4.4,0.4) -- (5.5,.508);
    \fill[gray!50] (5.5,.9) circle (38pt);
    \draw[line width=.2mm] (5.5,.9) circle (38pt);
    \draw[line width=.2mm, name path=A] plot [smooth,tension=0.689] coordinates {(4.4,0.4) (5.1,0.1) (5.8,0.48) (6.4967,0.5682)};
    \draw[line width=.2mm, name path=B] plot [smooth,tension=0.689] coordinates {(6.4967,0.5682)(6.5829,1.0505)(6.0727,1.4115) (5.3159,1.4086) (4.6,1) (4.4,0.4)};
    \tikzfillbetween[of=A and B]{blue, opacity=0.2};
    \fill (4.87,0.79)  node {\normalsize $B(x^*,\delta)$};
    \fill (5.9,1.57)  node {\normalsize $[f\leqslant \Delta_\epsilon]_{x^*}$};
    \draw[line width=.2mm] (5.5,.9) circle (10pt);
    \fill (4.6,2.2)  node {\normalsize $B(x^*,\epsilon/2)$};
    \filldraw (5.5,.9) circle (.2pt);
    \fill (5.61,.92)  node {\normalsize $x^*$};
    \draw[magenta,line width=.3mm] plot [smooth, tension=0.689] coordinates {(6.55,.95) (6.4,.85) (6.2918,.639)};
    \draw[magenta,line width=.3mm] plot [smooth, tension=0.689] coordinates {(6.3,.64) (6.1,.69) (5.85,.7)  (5.6,.6) (5.4,.64)};
    \draw[dashed,line width=.2mm] plot [smooth, tension=0.689] coordinates {(6.6,.8) (6.49,.69) (6.42,.44)};
    \draw[dashed,line width=.2mm] plot [smooth, tension=0.689] coordinates {(6.42,.44) (6.1,.51) (5.85,.52)  (5.6,.42) (5.4,.44)};
    \draw[dashed,line width=.2mm] plot [smooth, tension=0.689] coordinates {(6.5,1.1) (6.32,1) (6.2,.84)};
    \draw[dashed,line width=.2mm] plot [smooth, tension=0.689] coordinates {(6.2,.84) (6.1,.87) (5.85,.88)  (5.6,.78) (5.4,.82)};
    \filldraw[yellow] (6.5,.87) circle (.5pt);
    \filldraw[yellow] (6.33,.84) circle (.5pt);
    \filldraw[yellow] (6.37,.667) circle (.5pt);
    \filldraw[yellow] (6.38,.5) circle (.5pt);
    \filldraw[yellow] (6.23,.52) circle (.5pt);
    \filldraw[yellow] (6.06,.62) circle (.5pt);
    \filldraw[yellow] (5.88,.74) circle (.5pt);
    \filldraw[yellow] (5.65,.68) circle (.5pt);
    \filldraw[yellow] (5.5,.51) circle (.5pt);
    \fill (7.03,.803)  node {\normalsize $x_K$};
    \draw[->] (6.9,.82) -- (6.54,.864);
    \fill (5.0,.385)  node {\normalsize $x_{2K}$};
    \draw[->] (5.15,.405) -- (5.45,.493);
    \end{tikzpicture}
\caption{Induction step}
\label{fig:inclusions}
\end{figure}
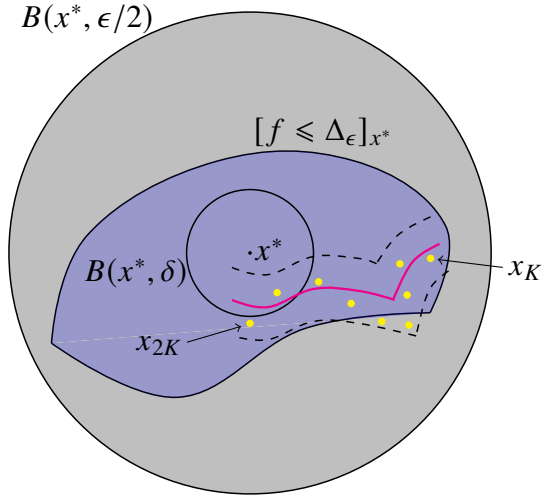
For the remainder of the proof, we seek to show that $x_K \in [f\leqslant \Delta_\epsilon]_{x^*}$. In order to do so, we prove that  $B(x(K\alpha),\epsilon')$ is a connected subset of $[f\leqslant \Delta_\epsilon]$ that has nonempty intersection with $[f\leqslant \Delta_\epsilon]_{x^*}$. Since $[f\leqslant \Delta_\epsilon]_{x^*}$ is a connected component of $[f\leqslant \Delta_\epsilon]$, by maximality and \eqref{eq:uniform_bound_iterates} we then have $x_K \in B(x(K\alpha),\epsilon') \subset [f\leqslant \Delta_\epsilon]_{x^*}$.
%
For all $\tilde{x} \in B(x(K\alpha),\epsilon')$, we have
\begin{subequations}
    \begin{align}
        	f(\tilde{x}) - f(x^*) & = f(\tilde{x}) - f(x(K\alpha)) + f(x(K\alpha)) - f(x^*) \label{level_a} \\
	        & \leqslant L\|\tilde{x} - x(K\alpha)\| + \max\{\sigma(\epsilon/2)/2,\sigma(\epsilon/2) - c\xi^2 T/2\} \label{level_b} \\
	        & \leqslant L\epsilon' + \sigma(\epsilon/2) - \min\{\sigma(\epsilon/2)/2,c\xi^2 T/2\} \label{level_c} \\
            & \leqslant \sigma(\epsilon/2). \label{level_d} 
    \end{align}
\end{subequations}
Indeed, $\tilde{x}$ and $x(K\alpha)$ belong to $B(x^*,\epsilon)$ so we may invoke the Lipschitz constant $L$ of $f$ on $B(x^*,\epsilon)$ in order to bound the first term in \eqref{level_a}. Recall that $x(K\alpha)\in [f\leqslant \Delta_\epsilon]_{x^*} \subset B(x^*,\epsilon/2)$ and, since $\epsilon' \leqslant \epsilon/2$, we have $\tilde{x} \in B(x(K\alpha),\epsilon') \subset B(x^*,\epsilon)$. As for the second term in \eqref{level_a}, if it is greater than or equal to $\sigma(\epsilon/2)/2$, then for all $t \in [0,K\alpha]$, we have $\sigma(\epsilon/2)/2 \leqslant  f(x(K\alpha)) - f(x^*) \leqslant  f(x(t)) - f(x^*)$ and thus $d(0,\partial f(x(t))) \geqslant 1/\psi'(f(x(t)) - f(x^*)) \geqslant 1/\psi'(\sigma(\epsilon/2)/2) = \xi$. By \eqref{eq:dist}, it follows that $f(x(K\alpha)) - f(x^*) \leqslant f(x(0)) - f(x^*) - \int_{0}^{K\alpha} c\xi^2 d\tau \leqslant f(x(0)) - f(x^*) - K\alpha c\xi^2 \leqslant \sigma(\epsilon/2) - c\xi^2 T/2$. The last inequality is due to the fact that $x(0) \in [f\leqslant \Delta_\epsilon]_{x^*}$ and $K\alpha = \lfloor T/\alpha\rfloor \alpha \geqslant T-\alpha \geqslant T-\bar{\alpha} \geqslant T/2$. In \eqref{level_c}, we use the fact that $\tilde{x} \in B(x(K\alpha),\epsilon')$ and rewrite the maximum into a minimum. Finally, \eqref{level_d} holds because $\epsilon' \leqslant \min\{\sigma(\epsilon/2),c\xi^2 T\}/(2L)$.
\end{proof}
We next illustrate \cref{thm:sufficient_strict} with the following example.

\begin{example}[A strict local minimum]
\label{eg:strict_local}
Consider the locally Lipschitz tame function defined from $\mathbb{R}^2$ to $\mathbb{R}$ by
\begin{equation}
    f(x_1,x_2):=\max\{-18x_1^2 + 12|x_2|,6x_1^2+3|x_2|\}.
\end{equation}
$(0,0)$ is a strict local minimum of $f$ because $f(x_1,x_2) \geqslant 6x_1^2+3|x_2| \geqslant 0$, with equality if and only if $x_1=x_2=0$. Thus, Theorem \ref{thm:sufficient_strict} can be applied and it guarantees the stability of $(0,0)$. See \cref{fig:strict} for a realization of the subgradient method. Note that $f$ is not sharp around $(0,0)$, and a slight modification (see \cref{subsec:modification}) of it yields a locally Lipschitz tame function with a strict local minimum that is neither sharp nor weakly convex. Recall that in order to show $\mathcal{M}_S$-stability for sharp and weakly convex functions in Example \ref{eg:sharpness}, it suffices to choose $\delta = \epsilon$ in \cref{def:discrete_lyapunov}. In contrast, one must choose $\delta < \epsilon$ in order to show stability in the current example. Indeed, no matter how small the step size is, there does not exist a ball centered at the origin that contain all iterates of the subgradient method that are initialized in the ball (see \cref{fig:goout} for an illustration and Section \ref{subsec:Proof of in Example} for a proof).
\end{example}
\begin{figure}
	\centering
	\includegraphics[width=.6\textwidth]{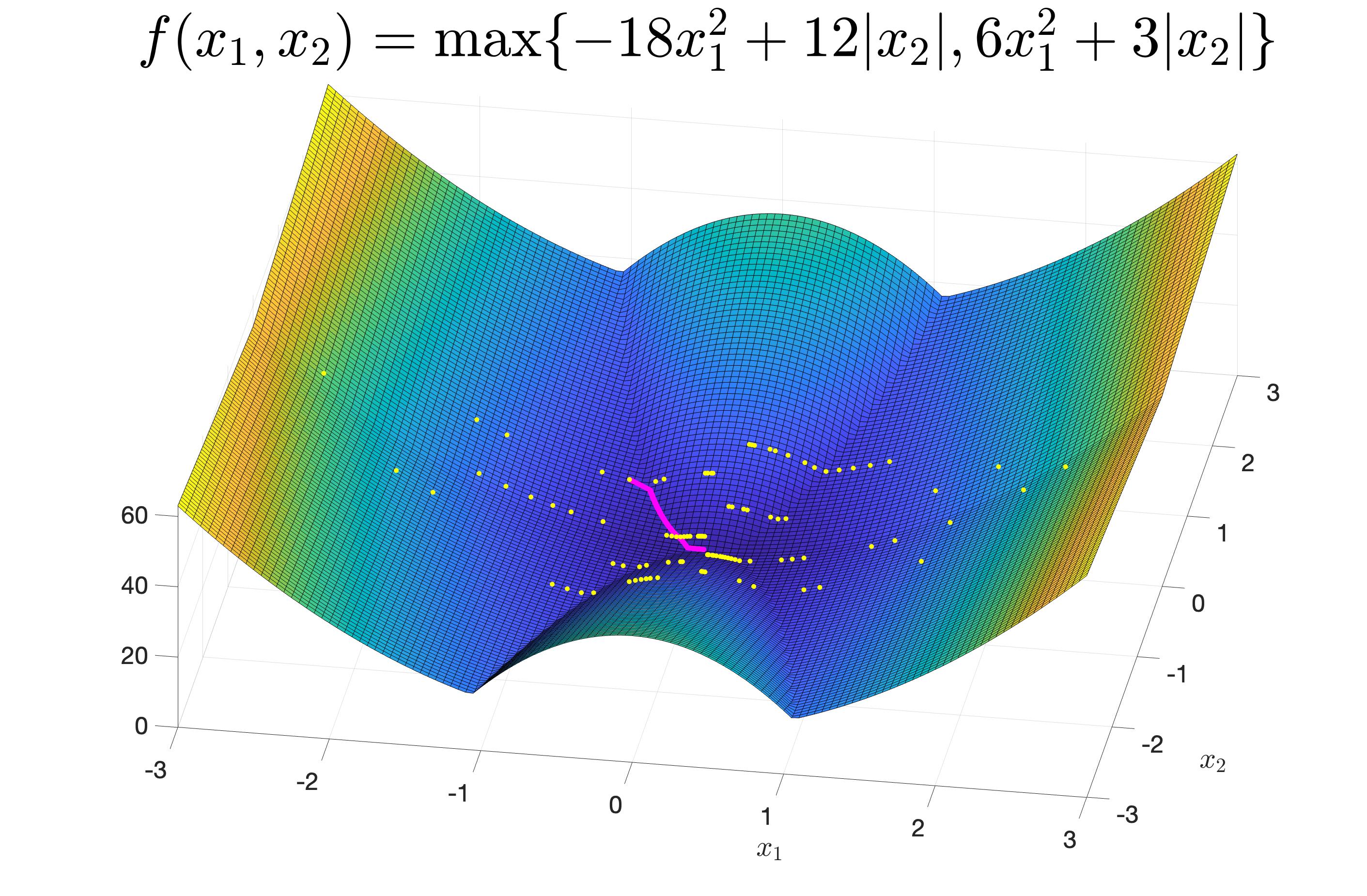}
	\caption{Continuous and discrete subgradient trajectories in magenta and yellow respectively.}
	\label{fig:strict}
\end{figure}
\begin{figure}
	\centering
 \title{asdasd}
	\includegraphics[width=.6\textwidth]{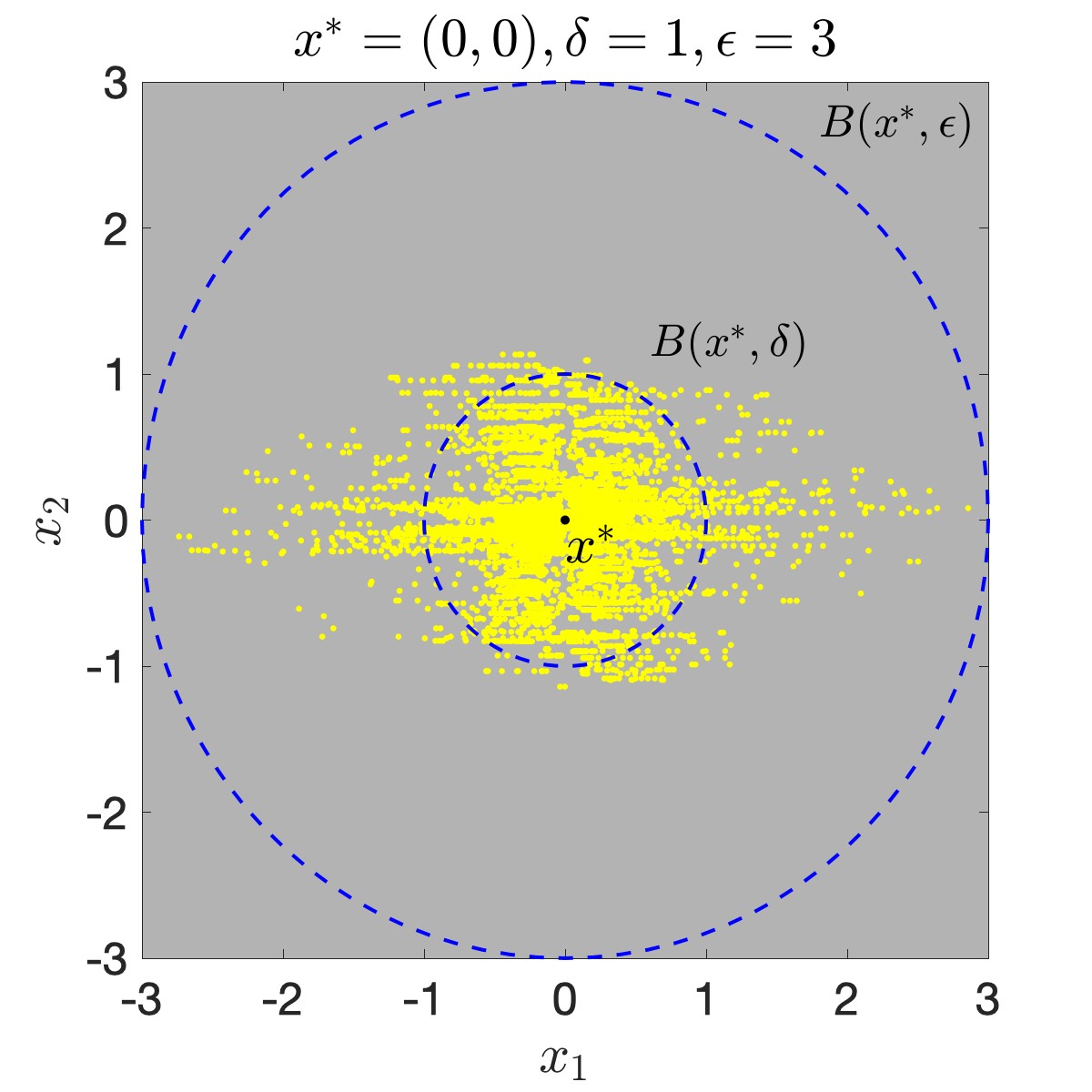}
	\caption{Subgradient method randomly initialized in the unit ball (100 trials with different step sizes).}
	\label{fig:goout}
\end{figure}

\section{Proof of stability in Example \ref{eg:stable_not_local}}
\label{subsec:Proof of stability in Example}
 The Clarke subdifferential of $f(x):=x^2\sin(1/x)$ is equal to
\begin{equation}
    \partial f(x) := 
    \left\{
   \begin{array}{ll}
        \{2x\sin(1/x) - \cos(1/x)\} & \text{if}~ x\neq 0,\\[.5mm]
        [-1,1] & \text{else}.
   \end{array}
   \right.
\end{equation}
Observe that if $x \neq 0$ is critical, then $\cos(1/x) \neq 0$, otherwise $|f'(x)| = |2x\sin(1/x) - \cos(1/x)| = 2|x||\sin(1/x)| = 2|x| \neq 0$. Apart from the origin, the critical points are thus the solutions to $\sin(1/x)/ \cos(1/x) = 1/(2x)$, that is to say, $\tan(1/x) = 1/(2x)$. Consider the function $\varphi$ defined from $\mathbb{R}\setminus \{ \pi/2 + k\pi ~|~ k \in \mathbb{Z}\}$ to $\mathbb{R}$ by $\varphi(t) := \tan(t) - t/2$. It is increasing on each interval $(\pi/2 + k\pi,\pi/2 + (k+1)\pi)$ where $k \in \mathbb{Z}$ since $\varphi'(t) = 1/\cos^2(t) - 1/2 >0$. In addition, its limits on either end of the interval are $-\infty$ and $+\infty$ respectively. As a result, $\varphi$ has exactly one root on each interval, say $t_k$. The positive roots of $\varphi$ are $t_0,t_1,t_2,\hdots$, hence the positive critical points of $f$ are $1/t_0,1/t_1,1/t_2,\hdots$ Since $f$ is odd, its negative critical points are $-1/t_0,-1/t_1,-1/t_2,\hdots$ This is illustrated in Figure \ref{fig:proof_stable_not_local}. As the figure suggests, the critical points of $f$ are alternatively local minima and local maxima. Indeed, when $k\in \mathbb{N}$, it holds that $f''(1/t_k) = 2\sin(t_k) - 2t_k \cos(t_k) - t_k^2 \sin(t_k) = \cos(t_k)[ 2\tan(t_k) -2t_k - t_k^2 \tan(t_k)] = \cos(t_k)[ 2(t_k/2) -2t_k - t_k^2 (t_k/2)] = \cos(t_k)(-t_k - t_k^3/2)$. The expression obtained is positive if $k$ is even, and negative if $k$ is odd.

In order to prove $\mathcal{M}_S$-stability, it suffices to prove the statement in \cref{def:discrete_lyapunov} for all $\epsilon>0$ sufficiently small. We may thus restrict ourselves to the case where $0<\epsilon<1/2$. Let $1/t_N$ denote the greatest critical point of $f$ in the interval $[-\epsilon,\epsilon]$, for some $N \in \mathbb{N}$. Let $\delta$ denote the second largest local minimum of $f$ less than or equal to $1/t_N$. Naturally, $\delta = 1/t_{N+2}$ if $N$ is even, and $\delta = 1/t_{N+3}$ if $N$ is odd. Without loss of generality, after possibly reducing $\epsilon$, we can assume that $N$ is even and thus define $\delta := 1/t_{N+2}$. We also define
\begin{equation}
    \bar{\alpha} := \frac{1}{2}\min\left\{~\frac{1}{t_{N}} - \frac{1}{t_{N+1}} ~,~ \frac{1}{t_{N+1}} - \frac{1}{t_{N+2}} ~,~ \frac{1}{t_{N+2}} - \frac{1}{t_{N+3}} ~ \right\}.
\end{equation}

Having chosen $\delta$ and $\bar{\alpha}$, let us fix some $\alpha \in (0,\bar{\alpha}]$ from now on. Consider a sequence $x_0,x_1,x_2,\hdots$ generated by the subgradient method with constant step size $\alpha$ and initialized in $[-\delta,\delta]$. Our goal is to show that all the iterates lie in $[-\epsilon,\epsilon]$. If $x_0,x_1,x_2,\hdots \in [-\delta,\delta]$, then there is nothing to prove since $\delta\leqslant \epsilon$. Otherwise, consider the smallest nonnegative integer $p$ for which $x_{p+1} \notin [-\delta,\delta]$. We next show that if $x_{p+1}>\delta$, then $x_{p+1},x_{p+2},x_{p+3},\hdots \in [1/t_{N+3} , 1/t_{N+1}] \subset [0,\epsilon]$. Similarly, if $x_{p+1}<-\delta$, then $x_{p+1},x_{p+2},x_{p+3},\hdots \in [-1/t_{N} , -1/t_{N+2}] \subset [-\epsilon,0]$. Putting these facts together, we obtain that $x_0,x_1,x_2, \hdots \in [-\epsilon,\epsilon]$.

Assume that $x_{p+1}>\delta$. We begin by checking that $x_{p+1} \in [1/t_{N+3} , 1/t_{N+1}]$, and then reason by induction. The lower bound follows from $\delta = 1/t_{N+2} \geqslant 1/t_{N+3}$. In order to derive the upper bound, recall that $0<\epsilon<1/2$. Hence, for all $x\in [-\epsilon,\epsilon]$ and for all $s \in \partial f(x)$, it holds that $|s|\leqslant 2$. Indeed, when $x=0$, we have $\partial f(x) = [-1,1]$ and when $x\neq 0$, we have $|f'(x)| \leqslant |2x\sin(1/x)-\cos(1/x)| \leqslant 2|x||\sin(1/x)|+|\cos(1/x)|\leqslant 2|x|+1 \leqslant 2$. Thus, for some $s_p \in \partial f(x_p)$, we have $x_{p+1} = x_p - \alpha s_p \leqslant x_p + 2\alpha \leqslant \delta + 2\alpha \leqslant \delta + 2\bar{\alpha} = 1/t_{N+2} + 2\bar{\alpha} \leqslant 1/t_{N+2} + 2(1/t_{N+1} - 1/t_{N+2})/2 = 1/t_{N+1}$. Now, assume that $x_k \in [1/t_{N+3} , 1/t_{N+1}]$ for some $k \geqslant p+1$. On the corresponding open interval, $f$ admits a unique critical point, $1/t_{N+2}$, which is a local minimum. Hence $f'(x)\leqslant 0$ when $x \in [1/t_{N+3},1/t_{N+2}]$ and $f'(x) \geqslant 0$ when $x \in [1/t_{N+2},1/t_{N+1}]$. If $x_k$ belongs to the first interval, then $1/t_{N+3} \leqslant x_k \leqslant x_k - \alpha f'(x_k) \leqslant x_k + 2\alpha \leqslant x_k + 2 \bar{\alpha} \leqslant 1/t_{N+2} + 2(1/t_{N+1} - 1/t_{N+2})/2 = 1/t_{N+1}$, that is to say, $x_{k+1} \in [1/t_{N+3},1/t_{N+1}]$. If $x_k$ belongs to the second interval, then $1/t_{N+3} \geqslant x_k \geqslant x_k - \alpha f(x_k) \geqslant x_k -2\alpha \geqslant 1/t_{N+2} - 2(1/t_{N+1} - 1/t_{N+2})/2 = 1/t_{N+1}$, that is to say, $x_{k+1} \in [1/t_{N+3},1/t_{N+1}]$. As a result, $x_{k+1} \in [1/t_{N+3} , 1/t_{N+1}]$, proving the induction. The case when $x_{p+1}<-\delta$ can be treated with similar arguments so is omitted.

\begin{figure}
  \centering
\begin{tikzpicture}[scale=7]
    \draw[->] (-.75,0) -- (.75,0) node[right] {$x$};
    \draw[->] (0,-.38) -- (0,.38) node[above] {$x^2\sin (1/x)$};
  \draw[blue,domain=0.01:.7,samples=5000] plot (\x, {\x*\x*sin((100/\x))});
  \draw[blue,domain=-.7:-0.01,samples=5000] plot (\x, {\x*\x*sin((100/\x))});
  \draw [densely dotted] (.408,-.15) -- (.408,0) node[above] {$\frac{1}{t_0}$};
  \draw [densely dotted] (-.408,0) node[below] {$-\frac{1}{t_0}$} -- (-.408,.15);
  \draw [densely dotted] (.2295,0) node[below] {$\frac{1}{t_1}$} -- (.2295,.051);
  \draw [densely dotted] (-.2295,-.051) -- (-.2295,0) node[above] {$-\frac{1}{t_1}$};
  \draw [densely dotted] (.16185,-.023) -- (.16185,0) node[above] {$\frac{1}{t_2}$};
  \draw [densely dotted] (-.16185,0) node[below] {$-\frac{1}{t_2}$} -- (-.16185,.023);
\end{tikzpicture}
 \caption{Visualization of the critical points.}
  \label{fig:proof_stable_not_local}
\end{figure}
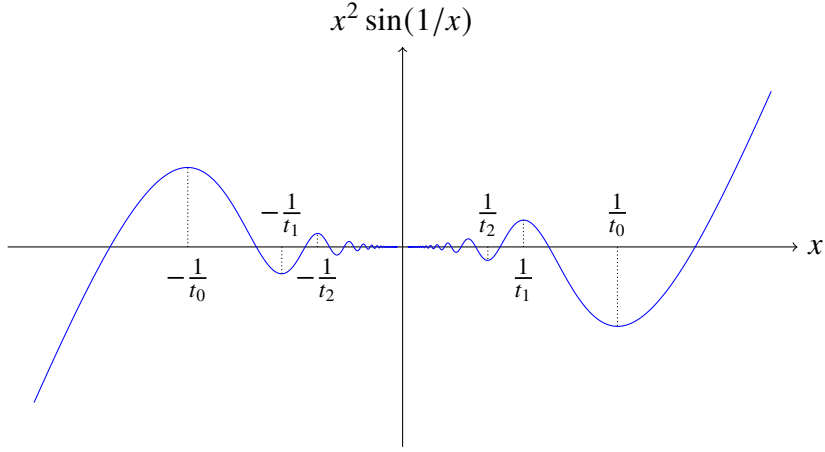

\section{Modification of Example \ref{eg:strict_local}}
\label{subsec:modification}

Consider the function $f$ defined from $\mathbb{R}^2$ to $\mathbb{R}$ by $f(x_1,x_2) := \max \{ -18|x_1|^{3/2} + 12|x_2|,6|x_1|^{3/2}+3|x_2| \}$. The origin is a strict local minimum because $f(x_1,x_2) \geqslant 6|x_1|^{3/2}+3|x_2| \geqslant 0$, with equality if and only if $x_1=x_2=0$. The function is not sharp because $f(x_1,0) = 6|x_1|^{3/2}$ for all $x_1\in \mathbb{R}$, violating the first inequality in \eqref{eq:sharp_weak}. The function is not weakly convex because the second inequality in \eqref{eq:sharp_weak} is violated by $(x_{1\rho},x_{2\rho}) := (3^{2/3}/\rho^2,8/\rho^3)$ for $\rho$ sufficiently large. Indeed, $8|x_{1\rho}|^{3/2} = 3|x_{2\rho}|$ and hence $f(x_{1\rho},x_{2\rho}) = -18|x_{1\rho}|^{3/2} + 12|x_{2\rho}| = 6|x_{1\rho}|^{3/2}+3|x_{2\rho}|$. Without loss of generality, let $\langle\cdot,\cdot\rangle$ be the Euclidean inner product on $\mathbb{R}^2$. It follows that $s_\rho := (-27 \mathrm{sign}(x_{1\rho})|x_{1\rho}|^{1/2}, 12 \mathrm{sign}(x_{2\rho}))^T \in \partial f(x_{1\rho},x_{2\rho})$. 
As a result, for $\rho$ sufficiently large, we have that $ f(x_{1\rho},x_{2\rho}) - \left\langle
        (x_{1\rho},x_{2\rho})^T , s_\rho \right\rangle - \rho(x_{1\rho}^2+x_{2\rho}^2)/2 = \hdots $
\begin{subequations}
    \begin{align}
         & = -18|x_{1\rho}|^{3/2} + 12|x_{2\rho}| - (-27 |x_{1\rho}|^{3/2} + 12 |x_{2\rho}|) - \frac{\rho}{2}(x_{1\rho}^2+x_{2\rho}^2) \\
         & = 9|x_{1\rho}|^{3/2} - \frac{\rho}{2}(x_{1\rho}^2+x_{2\rho}^2) \\
         & = 9(3^{2/3}/\rho^2)^{3/2} - \frac{\rho}{2}((3^{2/3}/\rho^2)^2+(8/\rho^3)^2) \\
         & = (27 - 3^{4/3}/2- 32/\rho^2)/\rho^3 > 0.
    \end{align}
\end{subequations}

\section{Proof of $\delta<\epsilon$ in Example \ref{eg:strict_local}}
\label{subsec:Proof of in Example}

Let $\langle\cdot,\cdot\rangle$ be the Euclidean inner product on $\mathbb{R}^2$. Recall that the objective function in Example \ref{eg:strict_local} is given by $f(x_1,x_2):=\max\{-18x_1^2 + 12|x_2|,6x_1^2+3|x_2|\}$. We reason by contradiction and assume that there exists $\epsilon>0$ and $\bar{\alpha}>0$ such that, for all $\alpha \in (0,\bar{\alpha}]$, the subgradient method with constant step size $\alpha$ initialized in $B(0,\epsilon)$ has all its iterates in $B(0,\epsilon)$. The ball $B(0,\epsilon)$ contains the following segment parametrized by $t \in [0,1]$:
\begin{equation}
    \begin{pmatrix}
    x_1(t) \\
    x_2(t)
    \end{pmatrix}
    =
    t \begin{pmatrix}
        \sqrt{-9/128 + \sqrt{81/16384+9/64\epsilon^2}} \\
        -3/16 + \sqrt{9/256+\epsilon^2}
    \end{pmatrix}.
\end{equation}
Indeed, $\|(x_1(t),x_2(t))\|^2 = t^2 (-9/128 + \sqrt{81/16384+9/64\epsilon^2}+9/256 - 3/8 \sqrt{9/256+\epsilon^2} + 9/256+\epsilon^2) = t^2 \epsilon^2$. Inside the segment, i.e. for $t\in (0,1)$, it holds that $f(x_1(t),x_2(t)) = -18x_1(t)^2 + 12|x_2(t)|>6x_1(t)^2+3|x_2(t)|$. Indeed, this is equivalent to saying that $3|x_2(t)| > 8x_1(t)^2$. If we can show that $3|x_2(1)|=8x_1^2$, then $3|x_2(t)|-8x_1(t)^2 = 3|tx_2(1)|-8(tx_1)^2 = t(3|x_2(1)|-8tx_1^2) = t(1-t)3|x_2(1)|>0$. The assumption $3|x_2(1)|=8x_1(1)^2$ is valid because $3|x_2(1)|-8x_1(1)^2 = -9/16 + \sqrt{81/256+9\epsilon^2} + 8 \times 9/128 - 8\sqrt{81/16384+9/64\epsilon^2} = 0$. As a result, the function $f$ is differentiable inside the segment and its gradient is equal to $\nabla f(x_1(t),x_2(t)) = (-36x_1(t),12\mathrm{sign}(x_2(t)))^T$. Let us apply one iteration of the subgradient method with constant step size $\alpha \in (0,\bar{\alpha}]$ to $(x_1(t),x_2(t))$ and compute the resulting norm:
\begin{subequations}
    \begin{align}
        \left\|
    \begin{pmatrix}
    x_1(t) \\
    x_2(t)
    \end{pmatrix} - \alpha \nabla f(x_1(t),x_2(t))
    \right\|^2 & = \left\| \begin{pmatrix}
    x_1(t) \\
    x_2(t)
    \end{pmatrix}
    \right\|^2
    - 2 \alpha \left\langle \begin{pmatrix}
    x_1(t) \\
    x_2(t)
    \end{pmatrix} ,
    \nabla f(x_1(t),x_2(t))
    \right\rangle + \\
    & ~~~~ \alpha^2 \left\| \nabla f(x_1(t),x_2(t))
    \right\|^2 \\
    & = t^2\epsilon^2 - 2\alpha (-36x_1(t)^2+12|x_2(t)|) + \alpha^2 (1296x_1(t)^2 + 144) \\
    & \xrightarrow{t\rightarrow 1} \epsilon^2 - 2\alpha (-36x_1(1)^2+12|x_2(1)|) + \alpha^2 (1296x_1(1)^2 + 144) \\
    & = \epsilon^2 + 72\alpha(1+18\alpha)x_1(1)^2 - 24\alpha |x_2(1)|  + 144\alpha^2 \\
    & = \epsilon^2 + 72\alpha(1+18\alpha)3|x_2(1)|/8 - 24\alpha |x_2(1)|  + 144\alpha^2 \\
    & = \epsilon^2 + \alpha(3+486\alpha)|x_2(1)|  + 144\alpha^2 \\
    & > \epsilon^2.
    \end{align}
\end{subequations}
Hence, for all $t<1$ sufficiently close to one, the subgradient method with constant step size $\alpha$ initialized at $(x_{1}(t),x_{2}(t))$, which belongs to $B(0,\epsilon)$, exits $B(0,\epsilon)$ after one iteration. This yields a contradiction. 


\chapter{Global stability of first-order methods for coercive functions}
\label{ch:global_stability}
In this chapter, we consider first-order methods with constant step size for minimizing locally Lipschitz coercive tame functions. We prove that if the method is approximated by subgradient trajectories, then the iterates eventually remain in a neighborhood of the set of critical points. The material of this chapter is based on the following article:\vspace*{3mm}

\noindent C\' edric Josz, Lexiao Lai, Global stability of first-order methods for coercive tame functions,\emph{ Mathematical Programming}, 2023 [\href{https://arxiv.org/abs/2308.00899}{preprint}] [\href{https://doi.org/10.1007/s10107-023-02020-9}{journal doi}]
 \vspace*{3mm}

In \cref{ch:local_stability}, we studied the local behavior of first-order methods around local minima of locally Lipschitz tame functions, rooted in the notion of (local) stability (\cref{def:discrete_lyapunov}). In this chapter, we turn our attention to their global behavior, namely the behavior when the initialization is not necessarily close to a local minimum. While the first-order methods (for e.g., \cref{alg:sg,alg:mag,alg:rrm,alg:rcd}) are implemented by machine learning practitioners \cite{sutskever2013importance,tensorflow2015-whitepaper,paszke2019pytorch}, the analysis of their global behavior seems to be absent from the literature when the objective is neither convex nor differentiable with a locally Lipschitz gradient. We will review the existing analysis in \cref{sec:literature}. 

In order to analyze their global behavior, we next propose a notion of global stability. Recall from \cref{ch:approx,ch:local_stability} that we refer to an iterative method with constant step size as a set-valued mapping $\mathcal{M}: \mathbb{R}^{(\mathbb{R}^n)} \times (0,\infty) \times 2^{(\mathbb{R}^n)} \times \mathbb{N}\rightrightarrows (\mathbb{R}^n)^\mathbb{N}$ which, to an objective function $f:\mathbb{R}^n \rightarrow \mathbb{R}$, a constant step size $\alpha \in (0,\infty)$, a set $X_0\subset \mathbb{R}^n$, and a natural number $k_0$ associates a set of sequences in $\mathbb{R}^n$ whose $k_0$\textsuperscript{th} term is contained in $X_0$. Also, we denote by $\mathcal{M}(f,\alpha,X_0):= \mathcal{M}(f,\alpha,X_0,0)$ the set of sequences generated by the method with constant step size $\alpha$ and initialized in $X_0$.

\begin{definition}
\label{def:global_stability}
We say that $X^*\subset \mathbb{R}^n$ is a globally $\mathcal{M}$-stable set of a locally Lipschitz function $f:\mathbb{R}^n\rightarrow\mathbb{R}$ if for all $\epsilon>0$ and for all bounded $X_0 \subset \mathbb{R}^n$, there exists $\bar{\alpha}>0$ such that
\begin{equation*}
	(x_k)_{k\in \mathbb{N}} \in \mathcal{M}(f,(0,\bar{\alpha}],X_0)\implies \exists \bar{k}: \{x_k\}_{k\geqslant \bar{k}}\subset B(X^*,\epsilon).
\end{equation*}
\end{definition}
Informally, a set is globally stable if iterates eventually remain in any neighborhood of it, given that the constant step size is sufficiently small. In contrast to local stability (\cref{def:discrete_lyapunov}), the initial iterate is not required to be close to the set. We say that $X^*\subset \mathbb{R}^n$ is a globally stable set of a locally Lipschitz function $f:\mathbb{R}^n \rightarrow \mathbb{R}$ if it is a globally $\mathcal{M}$-stable set of the function, for any $\mathcal{M}$ approximated by subgradient trajectories of $f$ (\cref{def:approx_flow_new}).

We are now ready to state the main result of this chapter, whose object is the set of critical points of a locally Lipschitz coercive tame function. Recall that a function $f:\mathbb{R}^n\rightarrow\mathbb{R}$ is coercive if $\lim_{\|x\|\rightarrow \infty} f(x) =\infty$. Many objective functions arising in data science are coercive due to the use of regularizers \cite{krogh1991simple,tibshirani1996regression}. Some objectives are naturally coercive, such as in symmetric low-rank matrix recovery problems \cite{ge2016,li2019}.

\begin{theorem}\label{thm:global_stability}
	The set of critical points of any locally Lipschitz coercive tame function is globally stable.
\end{theorem}

\begin{remark}
\label{remark:coercive}
The assumption of coercivity in \cref{thm:global_stability} can be replaced by requiring the iterates to be uniformly bounded for all sufficiently small step sizes when initialized in $X_0$. In other words, we can ask that for any iterative method $\mathcal{M}$ that satisfies \cref{def:approx_flow_new} and for any bounded $X_0 \subset \mathbb{R}^n$, there exist $\bar{\alpha},r>0$ such that $\mathcal{M}(f,(0,\bar{\alpha}],X_0) \subset B(0,r)^{\mathbb{N}}$. Indeed, 
one can then apply \cref{thm:global_stability} to a coercive function $f_r:\mathbb{R}^n \rightarrow \mathbb{R}$ which coincides with a (possibly noncoercive) locally Lipschitz tame function $f:\mathbb{R}^n \rightarrow \mathbb{R}$ in $B(0,2r)$, namely $f_r(x):=  f(P_{B(0,2r)}(x)) + d(x,B(0,2r))$ for all $x \in \mathbb{R}^n$ where $P_{B(0,2r)}$ is the projection on $B(0,2r)$.
It is clear that $f_r$ is definable and coercive. In order to show that $f_r$ is Lipschitz continuous, it suffices to prove $g_r(x):= f(P_{B(0,2r)}(x))$ is Lipschitz continuous. Let $L>0$ denote a Lipschitz constant of $f$ in $B(0,2r)$. For all $x,y\in \mathbb{R}^n$, we have $\|g_r(x) - g_r(y)\| = \|f(P_{B(0,2r)}(x)) - f(P_{B(0,2r)}(y))\| \leqslant L\|P_{B(0,2r)}(x) - P_{B(0,2r)}(y)\| \leqslant L\|x - y\|$.
\end{remark}
The proof of \cref{thm:global_stability} can be found in \cref{sec:global_proof}. We illustrate \cref{thm:global_stability} with two examples. Recall that \cref{alg:mag,alg:rrm,alg:rcd} are approximated by subgradient trajectories under method-dependent regularity assumptions (\cref{tab:appr_alg}). The first (Figure \ref{fig:illustration}a) is nonsmooth and the second (Figure \ref{fig:illustration}b) is continuously differentiable. One can see that the iterates indeed track a subgradient trajectory up to a certain time, then go on to track another subgradient trajectory, after which they stabilize around a critical point.
\begin{figure}[ht!]
\centering
\begin{subfigure}{.49\textwidth}
  \centering
  \includegraphics[width=0.95\textwidth]{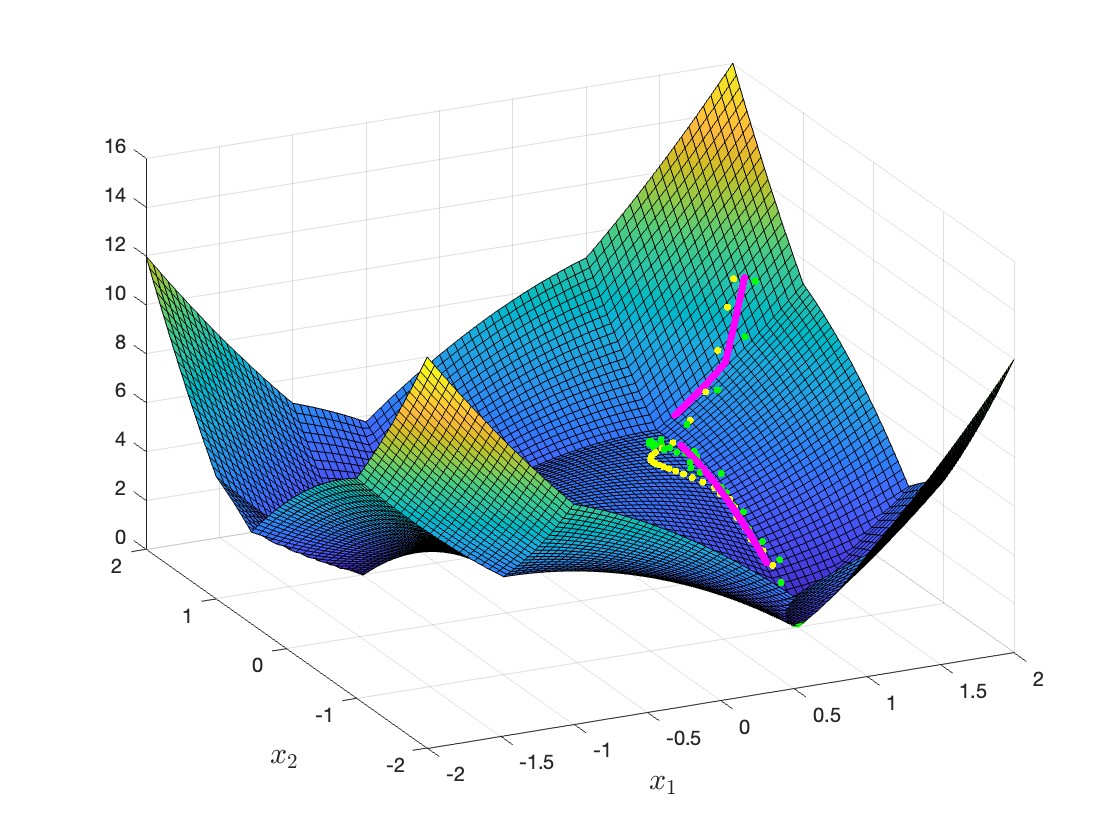}
  \caption{$f(x_1,x_2) = |x_1^2 - 1| + 2|x_1x_2+1| + |x_2^2 - 1|$.}
\end{subfigure}
 \begin{subfigure}{.49\textwidth}
  \centering
   \includegraphics[width=0.95\textwidth]{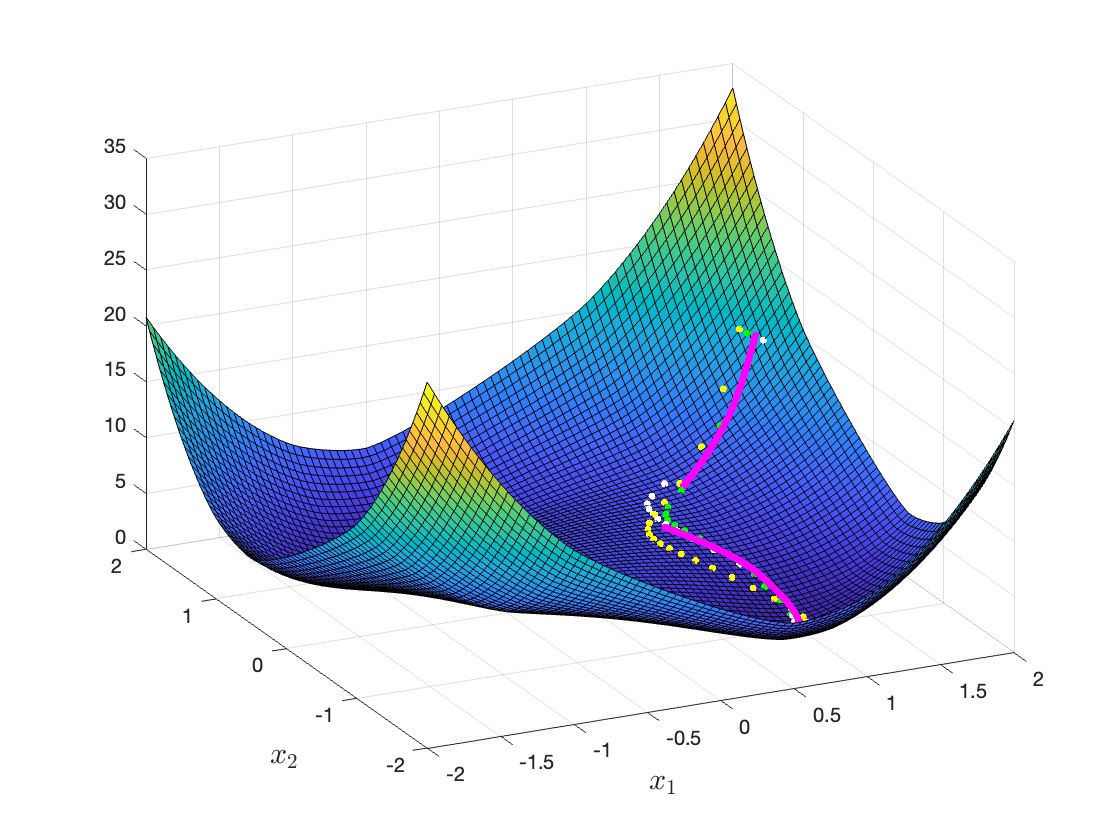}
  \caption{$f(x_1,x_2) = |x_1^2 - 1|^{3/2} + 2|x_1x_2+1|^{3/2} + |x_2^2 - 1|^{3/2}$.}
\end{subfigure}
\caption{The subgradient method with momentum, random reshuffling with momentum, and random-permutations cyclic coordinate descent method are in yellow, green, and white respectively. Subgradient trajectories are in magenta.}
\label{fig:illustration}
\end{figure}

\section{Proof of \cref{thm:global_stability}}\label{sec:global_proof}

We begin by stating two technical lemmas. The first relates a uniform neighborhood of a sublevel set with another sublevel set. The second is analogous to the descent lemma for smooth functions \cite[Lemma 1.2.3]{nesterov2018introductory} \cite[Lemma 5.7]{beck2017first}. We use $[f \leqslant \Delta]:= \{x\in \mathbb{R}^n: f(x) \leqslant \Delta\}$ to denote a sublevel set of a function $f:\mathbb{R}^n \rightarrow \mathbb{R}$ where $\Delta \in \mathbb{R}$.

\begin{lemma}
\label{lemma:containment}
    Let $f:\mathbb{R}^n \rightarrow \mathbb{R}$ be a locally Lipschitz function. Let $\Delta \in \mathbb{R}$ and let $L>0$ be a Lipschitz constant of $f$ in $[f\leqslant \Delta]$. For any $\epsilon'>0$, $B([f \leqslant \Delta - \epsilon'L], \epsilon') \subset [f \leqslant \Delta]$.
\end{lemma}
\begin{proof}
We show that $B(a,\epsilon') \subset [f\leqslant \Delta]$ for all $a \in [f\leqslant \Delta - \epsilon' L]$. Indeed, if $b \in B(a,\epsilon') \setminus [f\leqslant \Delta]$, then there exists $c$ in the segment $[a,b)$ such that $f(c) = \Delta$ and $ \epsilon' L = \Delta -(\Delta - \epsilon' L) \leqslant f(c)-f(a) \leqslant L\|c-a\| < \epsilon'L$. 
\end{proof}
\begin{lemma}
\label{lemma:decrease}
Let $f:\mathbb{R}^n \rightarrow \mathbb{R}$ be a locally Lipschitz tame function. Let $X\subset \mathbb{R}^n$ and $L$ be a Lipschitz constant of $f$ on $X$. For all $T, \epsilon', \alpha, c > 0$, $k_0 \in \mathbb{N}$, $(x_k)_{k\in \mathbb{N}} \in (\mathbb{R}^n)^\mathbb{N}$, and for any subgradient trajectory $x:[0,T]\rightarrow \mathbb{R}^n$ of $cf$ such that $x([0,T]) \subset X$, $x_k \in X$, and $\|x_k - x(\alpha (k-k_0))\|\leqslant \epsilon'$ for $k = k_0, \ldots, k_0+\lfloor T/\alpha\rfloor$, we have
	 \begin{equation*}
	     f(x_k) \leqslant f(x((k - k_0)\alpha)) + \epsilon' L \leqslant f(x_{k_0})- c\int_0^{(k - k_0)\alpha} d(0,\partial f (x(s)))^2~ds + 2\epsilon' L
	 \end{equation*}
	 for $k = k_0,\ldots, k_0 + \lfloor T/\alpha\rfloor$.
\end{lemma}
\begin{proof}
    For $k=k_0,\hdots,k_0+\lfloor T/\alpha\rfloor$, we have
	 \begin{subequations}
     \label{eq:fxk}
    	\begin{align}
		f(x_k) &\leqslant f(x((k - k_0)\alpha)) + \epsilon' L \label{eq:fxk_a}\\[2mm]
		&= f(x(0)) - (f(x(0)) - f(x((k - k_0)\alpha))) + \epsilon' L\label{eq:fxk_b}\\[2mm]
     &\leqslant f(x_0) - (f(x(0)) - f(x((k - k_0)\alpha))) + 2\epsilon' L\label{eq:fxk_c}\\[1mm]
		&= f(x_0)- c\int_0^{(k - k_0)\alpha} d(0,\partial f (x(s)))^2~ds + 2\epsilon' L. \label{eq:fxk_d}
    	\end{align}
     \end{subequations}
	 In \eqref{eq:fxk_a} and \eqref{eq:fxk_c}, we invoke the Lipschitz constant $L$ of $f$ on $X \ni x((k - k_0)\alpha),x_k$. \eqref{eq:fxk_d} is due to the chain rule of subgradient trajectories (\cref{prop:chain}).
\end{proof}

We next prove the final result needed for the proof of \cref{thm:global_stability}. We show that the function values evaluated at the iterates eventually stabilize around some critical value.

\begin{proposition}[Stability of function values]
\label{prop:fv_stab}
	Let $f:\mathbb{R}^n \rightarrow \mathbb{R}$ be a locally Lipschitz coercive tame function and let $\mathcal{M}$ be an iterative method with constant step size that is approximated by subgradient trajectories of $f$. For any bounded set $X_0 \subset \mathbb{R}^n$ and $\epsilon>0$, there exist $\bar{\alpha}, \Delta>0$ such that for all $(x_k)_{k \in \mathbb{N}} \in \mathcal{M}(f,(0,\bar{\alpha}],X_0,0)$, we have $f(x_k) \leqslant \Delta$ for all $k \in \mathbb{N}$ and there exist a critical value $f^*$ of $f$ and $\bar{k} \in \mathbb{N}$ such that $ |f(x_k) - f^*| \leqslant \epsilon$ for all $k \geqslant \bar{k}$.
\end{proposition}
\begin{proof}
Let $f:\mathbb{R}^n \rightarrow \mathbb{R}$ be a locally Lipschitz coercive tame function. Since $f$ is tame and coercive, there exists $\Delta>0$ such that $X_0 \subset [f\leqslant \Delta/2]$ and $\Delta$ is not a critical value of $f$. By the definable Morse-Sard theorem \cite[Corollary 9]{bolte2007clarke}, $f$ has finitely many critical values $f_1>\cdots>f_p$ in $[f\leqslant \Delta]$ (and it has at least one since $f$ is coercive and continuous). Since $f$ is coercive and continuous, the compact sublevel sets $[|f-f_i|\leqslant \epsilon]$, $i = 1,\hdots,p$, are pairwise disjoint after possibly reducing $\epsilon$, which we may do without loss of generality. We may also assume that $f_1+2\epsilon\leqslant\Delta$. According to the Kurdyka-\L{}ojasiewicz inequality (\cref{thm:kl}), there exist $\rho>0$ and a strictly increasing concave continuous definable function $\psi:[0,\rho) \rightarrow [0,\infty)$ that is continuously differentiable on $(0,\rho)$ with $\psi(0) = 0$ such that $d(0,\partial f(x)) \geqslant 1/\psi'(|f(x) - f_i|)$ for all $x \in [|f-f_i|\leqslant \epsilon]$ whenever $0<|f(x)-f_i|< \rho$ for $i = 1, \ldots, p$. Without loss of generality, we assume $\epsilon<\rho$ so that $d(0,\partial f(x)) \geqslant 1/\psi'(|f(x) - f_i|)$ for all $x \in [|f-f_i|\leqslant \epsilon]$ such that $f(x) \neq f_i$.

Consider a Lipschitz constant $L\geqslant 1$ of $f$ in $[f\leqslant \Delta]$ and the quantity
	 \begin{equation}
	 \label{eq:M}
	     M := \inf\{d(0,\partial f(x)):|f(x) - f_i|\geqslant \epsilon/2, ~ i =1, \ldots, p,~ f(x)\leqslant \Delta\}>0.
	 \end{equation}
	 Fix $T>0$. Since $\mathcal{M}$ is approximated by subgradient trajectories of $f$, by Definition \ref{def:approx_flow_new}, there exist $c>0$ and $\bar{\alpha} \in (0,T/2)$ such that
	 such that for all $\alpha \in (0,\bar{\alpha}]$, $k_0\in \mathbb{N}$, and $(x_k)_{k \in \mathbb{N}} \in \mathcal{M}(f,\alpha,[f\leqslant \Delta/2],k_0)$ for which $x_0, \ldots, x_{k_0} \in [f \leqslant \Delta]$, there exists a subgradient trajectory $x:[0,T]\rightarrow \mathbb{R}^n$ of $cf$  for which $x(0)\in [f\leqslant\Delta/2]$ and
	 $\|x_k - x(\alpha (k - k_0))\|\leqslant \epsilon'$ for $k = k_0, \ldots, k_0+ \lfloor T/\alpha\rfloor$ where
	 \begin{equation*}
	     \epsilon' := \min\left\{\frac{\Delta}{4L}, \frac{cM^2T}{24L}, \frac{\epsilon}{8L}, \frac{cT}{2L\psi'(\epsilon/2)^2}\right\}>0.
	 \end{equation*}
	 Since $[|f-f_1|\leqslant \epsilon],\ldots,[|f-f_p|\leqslant \epsilon]$ are compact, after possibly reducing  $\bar{\alpha}$, the statement still holds if one replaces the initial set $[f\leqslant\Delta/2]$ by $[|f-f_1|\leqslant \epsilon],[|f-f_2|\leqslant \epsilon],\hdots,$ or $[|f-f_p|\leqslant \epsilon]$.
	 
	 From now on, we fix a constant step size $\alpha \in (0,\bar{\alpha}]$.  Consider a sequence $(x_k)_{k \in \mathbb{N}} \in\mathcal{M}(f,\alpha,X_0,0) \subset \mathcal{M}(f,\alpha,[f\leqslant \Delta/2],0)$ along with an associated subgradient trajectory $x:[0,T]\rightarrow \mathbb{R}^n$ of $cf$ for which $x(0)\in [f\leqslant \Delta/2] \subset [f\leqslant \Delta - \epsilon' L]$ and
	 $\|x_k - x(\alpha (k-k_0))\|\leqslant \epsilon'$ for $k = k_0, \ldots, k_0+K$ where $k_0 = 0$ and $K:=\lfloor T/\alpha\rfloor$. By Lemmas \ref{lemma:containment} and \ref{lemma:decrease}, for $k=0,\hdots,K$, we have $f(x_k) \leqslant f(x(k\alpha)) + \epsilon' L \leqslant f(x(0)) + \epsilon'L \leqslant \Delta/2 + \epsilon' L \leqslant \Delta$ and
	 \begin{equation}
  \label{eq:decrease}
		f(x_k) \leqslant f(x_0)- c\int_0^{k\alpha} d(0,\partial f (x(s)))^2~ds + 2\epsilon' L. 
     \end{equation}
If $c\int_0^{K \alpha} d(0,\partial f (x(s)))^2~ds  \geqslant 3\epsilon' L$, then we have $f(x_{K}) \leqslant f(x_0) -3\epsilon' L +2\epsilon' L \leqslant \Delta/2$ so that we may apply Lemmas \ref{lemma:containment} and \ref{lemma:decrease} again with $k_0 = K$. Since the continuous function $f$ is bounded below on the compact set $[f\leqslant \Delta/2]$, this process with constant decrease can only be repeated finitely many times. Thus there exist $v\in \mathbb{N}$ and an absolutely continuous function (again denoted $x(\cdot)$) such that $f(x_k) \leqslant f(x_{vK})- c\int_0^{(k-vK)\alpha} d(0,\partial f (x(s)))^2~ds + 2\epsilon' L$ and $\|x_{k}-x(\alpha(k-vK))\| \leqslant \epsilon'$ for $k = vK, \ldots, (v+1)K$ where $c\int_0^{K \alpha} d(0,\partial f (x(s)))^2~ds  < 3\epsilon' L$. Hence there exists $t'\in [0,K \alpha]$ such that $d(0,\partial f(x(t')))^2 \leqslant 3\epsilon' L/(cK\alpha) \leqslant  3\epsilon' L/(cT/2) \leqslant M^2/4$, where we use the fact that $\epsilon' \leqslant cM^2T/(24L)$. Since $d(0,\partial f(x(t'))) \leqslant M/2$ and $f(x(t')) \leqslant \Delta$, by definition of $M$ in \eqref{eq:M} there exists $i\in \{1,\ldots,p\}$ such that $|f(x(t')) - f_i| < \epsilon/2$. We also have that $f(x(t')) \leqslant f(x(0)) \leqslant f(x_{vK}) + \epsilon'L \leqslant \Delta/2 + \epsilon'L$. Thus $f_i<f(x(t')) + \epsilon/2 \leqslant \Delta/2 + \epsilon'L + \epsilon/2 \leqslant \Delta/2+3\epsilon/8$. For $k' = vK,\ldots, (v+1)K$, we have
\begin{subequations}
\label{eq:fvc}
	\begin{align}
	|f(x_{k'}) - f_i| \leqslant & |f(x_{k'}) - f(x(\alpha(k' - vK)))| + |f(x(\alpha(k' - vK))) - f(x(t'))| + \label{eq:fvc1}\\[2mm]
	& |f(x(t')) - f_i| \label{eq:fvc2} \\[2mm]
		 \leqslant & L\|x_{k'} - x(\alpha(k' - vK))\| + |f(x(0)) - f(x(K \alpha))| + \epsilon/4\label{eq:fvc3}\\[2mm]
		 \leqslant & L\epsilon' + 3\epsilon'L + \epsilon/2\label{eq:fvc4}\\[2mm]
		 \leqslant & \epsilon/8+3\epsilon/8+\epsilon/2\label{eq:fvc5}\\[2mm]
		 = & \epsilon.
	\end{align}
\end{subequations}
Indeed, \eqref{eq:fvc1} is due to the triangular inequality. We invoke the Lipschitz constant $L$ of $f$ on $[f\leqslant \Delta]$ in order to bound the first term in \eqref{eq:fvc1}. In order to bound the second term in \eqref{eq:fvc1}, we use the fact that the composition $f\circ x$ is decreasing and $0 \leqslant \alpha(k' - vK) \leqslant t' \leqslant K\alpha$. \eqref{eq:fvc4} holds because $\|x_{k'} - x(\alpha(k' - vK))\| \leqslant \epsilon'$ and $|f(x(0)) - f(x(K \alpha))| = c\int_0^{K \alpha} d(0,\partial f (x(s)))^2~ds  < 3\epsilon' L$. \eqref{eq:fvc5} is due to $\epsilon' \leqslant \epsilon/(8L)$.

We next show that $f(x_k) \leqslant f_i + \epsilon$ for all $k \geqslant k':=vK$. Without loss of generality, we assume that $k'=0$ so that by \eqref{eq:fvc} we have $f(x_k) \leqslant f_i + \epsilon$ for $k = 0, \hdots, K$. We prove that $f(x_{K+1})\leqslant f_i+\epsilon$, hence $f(x_{k})\leqslant f_i + \epsilon$ for all $k\geqslant k'$ by induction. We distinguish two cases. If $f(x_1)<f_i - \epsilon$, then $f(x_{K+1}) \leqslant f(x_1) + 2\epsilon'L < f_i - \epsilon + \epsilon/4 \leqslant f_i + \epsilon$, where the first inequality follows from $x_1 \in [f\leqslant f_i-\epsilon] \subset [f\leqslant \Delta/2 + 3\epsilon/8 - \epsilon] \subset [f\leqslant \Delta/2]$ and Lemmas \ref{lemma:containment} and \ref{lemma:decrease}. If $x_1 \in [|f-f_i|\leqslant \epsilon]$, then let $x:[0,T]\rightarrow \mathbb{R}^n$ be an associated subgradient trajectory of $cf$  such that $\|x_k - x(\alpha (k-1))\|\leqslant \epsilon'$ for $k = 1,\ldots, K+1$ and $x(0) \in [|f-f_i|\leqslant \epsilon]$. Note that for any $t \in [0,K\alpha]$, $f(x(K\alpha)) \leqslant f(x(t))\leqslant f(x(0)) \leqslant f_i+\epsilon \leqslant \Delta - \epsilon \leqslant \Delta - \epsilon' L$. By Lemmas \ref{lemma:containment} and \ref{lemma:decrease}, $x_{K+1} \in [f\leqslant\Delta]$ and $f(x_{K+1}) \leqslant f(x(K\alpha)) + \epsilon'L$. If $f(x(K\alpha)) \leqslant f_i + \epsilon/2$, we have that $f(x_{K+1}) \leqslant f(x(K\alpha)) + \epsilon'L <f_i+\epsilon/2+\epsilon/8 \leqslant f_i+\epsilon$, as desired. Otherwise, we have $f(x(t)) \in [f_i+\epsilon/2,f_i+\epsilon]$ for all $t \in [0,K \alpha]$.
By the Kurdyka-\L{}ojasiewicz inequality, we have $d(0,\partial f(x(t))) \geqslant 1/\psi'(f(x(t)) - f_i) \geqslant 1/\psi'(\epsilon/2)>0$. According to the chain rule of subgradient trajectories (\cref{prop:chain}), it holds that
\begin{subequations}
	\begin{align}
		f(x(K \alpha)) - f_i&= f(x(0)) - f_i- c\int_{0}^{K \alpha} d(0,\partial f (x(s)))^2~ds\\
		&\leqslant f(x(0)) - f_i- cK\alpha/\psi'(\epsilon/2)^2\\[1mm]
		&\leqslant f(x(0)) - f_i - cT/(2\psi'(\epsilon/2)^2)\\[2mm]
		&\leqslant \epsilon -  cT/(2\psi'(\epsilon/2)^2).
		\end{align}
\end{subequations}
Thus $f(x_{K+1}) - f_i \leqslant f(x(K \alpha)) - f_i + f(x_{K+1}) - f(x(K \alpha)) \leqslant \epsilon - cT/(2\psi'(\epsilon/2)^2) + \epsilon'L \leqslant \epsilon$, where we used the fact that $\epsilon' \leqslant (cT)/(2L\psi'(\epsilon/2)^2)$.
If $|f(x_k) - f_i| \leqslant \epsilon$ for all $k \geqslant k'$, then the conclusion of the theorem follows. Otherwise, there exists $\hat{k}\geqslant k'$ such that $f(x_{\hat{k}})< f_i - \epsilon \leqslant \Delta/2 + 3\epsilon/8 - \epsilon \leqslant \Delta/2$. Following the same argument as in the paragraph below \eqref{eq:decrease}, there exists $v'\in \mathbb{N}$ and an absolutely continuous function (again denoted $x(\cdot)$) such that $f(x_k) \leqslant f(x_{\hat{k} + v'K})- c\int_0^{(k - (\hat{k} + v'K))\alpha} d(0,\partial f (x(s)))^2~ds + 2\epsilon' L$ and $\|x_{k}-x(\alpha(k-(\hat{k} + v'K)))\| \leqslant \epsilon'$ for $k = \hat{k}+v'K, \ldots, \hat{k}+(v'+1)K$ where $c\int_0^{K \alpha} d(0,\partial f (x(s)))^2~ds  < 3\epsilon' L$. As before, it follows that there exist $t''\in [0,T]$ and $j \in \{1,2, \ldots, p\}$ such that $|f(x(t'')) - f_j| \leqslant \epsilon/2$. Since $f(x(t'')) \leqslant f(x(0)) \leqslant f(x_{\hat{k} + v'K}) + \epsilon'L\leqslant f(x_{\hat{k}}) + 3\epsilon'L < f_i - \epsilon+ 3\epsilon/8 = f_i - 5\epsilon/8$, it holds that $f_j<f_i$. Replicating \eqref{eq:fvc1}-\eqref{eq:fvc5}, we get $|f(x_{k''}) - f_j| \leqslant \epsilon$ for $k''= \hat{k} + v'K, \ldots, \hat{k} + (v'+1)K$. By the same argument as in the previous paragraph, we have $f(x_k) \leqslant f_j+ \epsilon$ for all $k \geqslant k'':=\hat{k} + (v'+1)K$. Since $f$ only has finitely many critical values that are below $\Delta$, the conclusion of the theorem follows. 
\end{proof}

We are now ready to prove \cref{thm:global_stability}. 

\begin{proof}[Proof of \cref{thm:global_stability}]
Let $f:\mathbb{R}^n \rightarrow \mathbb{R}$ be a locally Lipschitz coercive tame function and let $\mathcal{M}$ be an iterative method with constant step size that is approximated by subgradient trajectories of $f$ (\cref{def:approx_flow_new}). Let $\epsilon>0$  and let $X_0\subset \mathbb{R}^n$ be bounded. By \cref{prop:fv_stab}, there exist $\alpha_1,\Delta>0$ such that for all $(x_k)_{k\in\mathbb{N}} \in \mathcal{M}(f,(0,\alpha_1], X_0, 0)$, $f(x_k) \leqslant \Delta$ for all $k \in \mathbb{N}$.
Let $L$ denote a Lipschitz constant of $f$ on the compact set $[f\leqslant \Delta]$ and consider the quantity
\begin{equation}
    M := \inf\{d(0,\partial f(x)): d(x,X^*)\geqslant \epsilon/2, f(x)\leqslant \Delta\}>0,
\end{equation}
where $X^*$ is the set of critical points of $f$. Fix $T>0$. Since $\mathcal{M}$ is approximated by subgradient trajectories of $f$, by Definition \ref{def:approx_flow_new} there exist $c>0$ and $\alpha_2 \in (0,\alpha_1]$ such that for all $\alpha \in (0,\alpha_2],k_0 \in \mathbb{N}$ and $(x_k)_{k\in \mathbb{N}} \in \mathcal{M}(f,\alpha,[f\leqslant \Delta], k_0)$ for which $x_0,\ldots,x_{k_0} \in [f \leqslant \Delta]$, there exists a subgradient trajectory $x:[0,T]\rightarrow \mathbb{R}^n$ of $cf$ for which $x(0)\in [f\leqslant\Delta]$ and $\|x_{k} - x((k - k_0)\alpha)\|\leqslant \epsilon'$ for $k = k_0, \ldots, k_0 + \lfloor T/\alpha\rfloor$ where $\epsilon' := \min\{\epsilon/4,cM^2T/(16(1+L)),\epsilon^2/(32(1+L)cT)\}$. Again by \cref{prop:fv_stab} there exists $\alpha_3 \in (0,\alpha_2]$ such that for all $(x_k)_{k \in \mathbb{N}} \in \mathcal{M}(f,(0,\alpha_3],X_0,0)$, there exist a critical value $f^*$ of $f$ and $\bar{k} \in \mathbb{N}$ such that $ |f(x_k) - f^*| \leqslant \epsilon'$ for all $k \geqslant \bar{k}$. Let $\bar{\alpha} := \min\{\alpha_3, \epsilon'/(2c(1+L)),T/2\}$. 

Let $\alpha \in (0,\bar{\alpha}]$, $(x_k)_{k \in \mathbb{N}} \in \mathcal{M}(f,\alpha, X_0, 0)$, and fix a corresponding $f^*$ and $\bar{k}$. We fix some $k \geqslant \bar{k}$ from now on and show that $d(x_k,X^*)\leqslant \epsilon$. Since $(x_{k'})_{k'\in \mathbb{N}}\in \mathcal{M}(f,\alpha, [f\leqslant \Delta],k)$ and $\{x_{k'}\}_{k'\in \mathbb{N}} \subset [f\leqslant \Delta]$,  there exists a subgradient trajectory $x:[0,T]\rightarrow \mathbb{R}^n$ of $cf$ for which $x(0)\in [f\leqslant\Delta]$ and $\|x_{k'} - x(\alpha (k'-k))\|\leqslant \epsilon'$ for $k' = k, \ldots, k + K$ where $K:= \lfloor T/\alpha\rfloor$. 
By Lemma \ref{lemma:decrease}, we have
 \begin{equation}
 \label{eq:infgrad}
   c\int_0^{K \alpha} d(0,\partial f (x(s)))^2~ds \leqslant f(x_k) - f(x_{k + K}) + 2\epsilon' L \leqslant 2\epsilon' (1+L).
 \end{equation}
Thus there exists $t' \in [0,K \alpha]$ such that $d(0,\partial f(x(t')))^2\leqslant 2\epsilon' (1+L)/(cK\alpha) \leqslant 2\epsilon' (1+L)/(cT/2) \leqslant M^2/4$, where we use the fact that $\epsilon' \leqslant cM^2T/(16(1+L))$.
As $f(x(t')) \leqslant f(x(0))\leqslant \Delta$, we have $d(x(t'),X^*) \leqslant \epsilon/2$. It now suffices to show that $\|x_k-x(t')\| \leqslant \epsilon/2$. Notice that $\|x_k - x(0)\| \leqslant \epsilon' \leqslant \epsilon/4$ and
\begin{subequations}
    \begin{align}
        \|x(0) - x(t')\| &\leqslant \int_0^{t'} \|x'(s)\|~ds \label{eq:fvl-a}\\
        &= \int_0^{t'} c~d(0,\partial f(x(s)))~ds\label{eq:fvl-b}\\
        &\leqslant \sqrt{\int_0^{t'} c~ds} \sqrt{\int_0^{t'} c~d(0,\partial f(x(s)))^2~ds}\label{eq:fvl-c}\\
        &\leqslant \sqrt{\int_0^{T} c~ds} \sqrt{\int_0^{K \alpha}
        c~d(0,\partial f(x(s)))^2~ds}\label{eq:fvl-d}\\[2mm]
        &\leqslant \sqrt{cT} \sqrt{2\epsilon' (1+L)}\label{eq:fvl-e}\\[2mm]
        &\leqslant \epsilon/4.\label{eq:fvl-f}
    \end{align}
\end{subequations}
Indeed, \eqref{eq:fvl-a} is due to triangular inequality. \eqref{eq:fvl-b} is a consequence of the chain rule of subgradient trajectories (\cref{prop:chain}). \eqref{eq:fvl-c} is due to the Cauchy-Schwarz inequality. \eqref{eq:fvl-f} is due $\epsilon' \leqslant \epsilon^2/(32(1+L)cT)$. Summing up, we have $|d(x_k,X^*) - d(x(t'),X^*)| \leqslant \|x_k-x(t')\|\leqslant \|x_k-x(0)\|+\|x(0)-x(t')\| \leqslant \epsilon/2$ and thus $d(x_k,X^*) \leqslant \epsilon$. 
\end{proof}


\chapter{Sufficient conditions for instability of the subgradient method}\label{ch:instability}
In this chapter, we provide sufficient conditions for instability of the subgradient method with constant step size around a local minimum of a locally Lipschitz tame function. They are satisfied by several spurious local minima arising in robust principal component analysis and neural networks.  The material of this chapter is based on the following article:\vspace*{3mm}

\noindent C\' edric Josz, Lexiao Lai, Sufficient conditions for instability of the subgradient method with constant step size, \emph{SIAM Journal on Optimization}, 2024 [\href{https://arxiv.org/abs/2211.14852}{preprint}] [\href{https://epubs.siam.org/doi/full/10.1137/22M1535723}{journal doi}]
 \vspace*{3mm}
 

In \cref{ch:local_stability}, we proposed and studied a notion of discrete Lyapunov stability for first-order methods that are approximated by subgradient trajectories (\cref{def:approx_flow_new}). Those results apply in particular to the subgradient method with constant step size (\cref{alg:sg}), which satisfies \cref{def:approx_flow_new} by \cref{thm:subg_approx}. Throughout this chapter, we say that a point is stable (resp. unstable) if it is stable (resp. not stable) in the sense of \cref{def:discrete_lyapunov} where the iterative method $\mathcal{M}$ is taken to be the subgradient method (see \cref{def:discrete_lyapunov_subg}).

Let us summarize what do we know so far regarding the (local) stability of the subgradient method. Assume that the objective function is locally Lipschitz and tame. In order for a point to be stable, it is necessary for it to be a local minimum (\cref{thm:necessary}) and it suffices for it to be a strict local minimum (\cref{thm:sufficient_strict}). If the function is additionally differentiable with a locally Lipschitz gradient, then it suffices to be a local minimum \cite[Proposition 3.3]{absil2005convergence}. 

In this chapter, we show that the existence of a Chetaev function \cite{chetaev1961stability} in a neighborhood of a non-strict local minimum satisfying certain geometric properties guarantees instability. Chetaev functions are similar to Lyapunov functions, except that they increase along the dynamics rather than decrease. We check that the geometric properties, which involve higher-order metric subregularity \cite{li2012holder, mordukhovich2015higher,zheng2015holder} and the Verdier condition \cite{verdier1976stratifications}, hold in several applications of interest and exhibit corresponding Chetaev functions. 

The Verdier condition was recently introduced to the field of optimization by Bianchi \textit{et al.} \cite{bianchi2023stochastic} and Davis \textit{et al.} \cite{davis2023active}. Those works extend to the nonsmooth setting the pioneering work by Pemantle \cite{pemantle1990nonconvergence} on the nonconvergence to strict saddle points of the perturbed gradient method with diminishing step size. Precisely, they consider the update rule $x_{k+1} \in x_k - \alpha_k (\partial f(x_k) + \epsilon_k)$ for all $k\in \mathbb{N}$, where there exist $0<c_1<c_2$ and $\gamma \in (1/2,1]$ such that $c_1/k^\gamma \leqslant \alpha_k \leqslant c_2/k^\gamma$ for all $k\in \mathbb{N}^* := \{1,2,3,\hdots\}$. Also, the random variable $\epsilon_k$ is drawn uniformly from a ball of radius $r>0$ centered at the origin. They prove nonconvergence to active strict saddles \cite[Definition 2.3]{davis2023active} satisfying the Verdier condition and an angle/proximal aiming condition \cite[Theorem 3]{bianchi2023stochastic} \cite[Theorem 6.2]{davis2023active}.

As shown by Lee \textit{et al.} \cite[Theorem 4]{lee2016} (see also \cite{panageas2017}), in the smooth setting and with constant step size, adding random noise is actually not necessary to prevent convergence to strict saddle points almost surely. More recently, it was observed \cite[Figure 3]{kleinberg2018alternative} that the gradient method with constant step size can escape spurious local minima after adding uniform random noise. A similar observation on the benefits of noise was made in \cite{keskar2016large} when training neural networks: large batch sizes tend to converge to sharp local minima \cite[Metric 2.1]{keskar2016large}, while small batch sizes tend to converge to flat local minima. We show that critical points can be inherently unstable due to the local geometry of the objective function, without adding any noise.

Let us restate the notion of stability (\cref{def:discrete_lyapunov}) in the setting of the subgradient method.

\begin{definition}
\label{def:discrete_lyapunov_subg}
We say that $x^*\in \mathbb{R}^n$ is a stable point of a locally Lipschitz function $f:\mathbb{R}^n\rightarrow\mathbb{R}$ if for all $\epsilon>0$, there exist $\delta>0$ and $\bar{\alpha}>0$ such that for all $\alpha \in (0,\bar{\alpha}]$, the subgradient method with constant step size $\alpha$ initialized in $B(x^*,\delta)$ has all its iterates in $B(x^*,\epsilon)$.
\end{definition}

According to the above definition, a point $x^* \in \mathbb{R}^n$ is unstable if there exists $\epsilon>0$ such that for all $\delta>0$ and $\bar{\alpha}>0$, there exists $\alpha \in (0,\bar{\alpha}]$ and an initial point $x_0\in B(x^*,\delta)$ such that at least one of the iterates of the subgradient method with constant step size $\alpha$ does not belong to $B(x^*,\epsilon)$. The sufficient conditions proposed in this chapter actually imply instability in a stronger sense.

\begin{definition}
\label{def:strong_unstable}
We say that $x^*\in \mathbb{R}^n$ is a strongly unstable point of a locally Lipschitz function $f:\mathbb{R}^n\rightarrow\mathbb{R}$ if there exists $\epsilon>0$ such that for all but finitely many constant step sizes $\alpha>0$ and for almost every initial point in $B(x^*,\epsilon)$, at least one of the iterates of the subgradient method does not belong to $B(x^*,\epsilon)$.
\end{definition}

 In order to describe the nature of the set of critical points around a non-strict local minimum, we recall the definition of a smooth manifold.



\begin{definition}
\label{def:manifold}
A subset $S$ of $\mathbb{R}^n$ is a $C^p$ manifold with positive $p\in \mathbb{N}$ of dimension $m \in \mathbb{N}$ at $x \in S$ if there exists a Euclidean space $E$ of dimension $n-m$ such that there exists an open neighborhood $U$ of $x$ in $\mathbb{R}^n$ and a $p$ times continuously differentiable function $\varphi: U \rightarrow E$ such that $S\cap U = \varphi^{-1}(0)$ and whose Jacobian is surjective.
\end{definition}


We will use the following notions related to a $C^p$ manifold $S$ at a point $x$. According to \cite[Example 6.8]{rockafellar2009variational}, the tangent cone $T_S(x)$ \cite[6.1 Definition]{rockafellar2009variational} and the normal cone $N_S(x)$ \cite[6.3 Definition]{rockafellar2009variational} at a point $x$ in $S$ are respectively the kernel of $\varphi'(x)$ and the range of $\varphi'(x)^*$ where $\varphi'(x)$ is the Jacobian of the function $\varphi$ in \cref{def:manifold} at $x$ and $\varphi'(x)^*$ is its adjoint. 

In order to describe the variation of the objective function around a non-strict local minimum, we borrow the notion of metric $\theta$-subregularity of a set-valued mapping \cite{li2012holder,mordukhovich2015higher,zheng2015holder}. It is a generalization of metric subregularity \cite[Equation (4)]{van2015metric} \cite[Definition 2.3]{artacho2008characterization} \cite[Definition 3.1]{dontchev2004regularity} that has been used to study the Mordukhovich subdifferential \cite[Theorem 3.4]{mordukhovich2015higher}. 
Given a set-valued mapping $F:\mathbb{R}^n\rightrightarrows\mathbb{R}^m$, let $\mathrm{graph}~F := \{ (x,y) \in \mathbb{R}^n \times \mathbb{R}^m : F(x) \ni y \}$.

\begin{definition}
\label{def:submetric}
\cite[Definition 3.1]{li2012holder} A mapping $F:\mathbb{R}^n\rightrightarrows\mathbb{R}^m$ is metrically $\theta$-subregular at $(\bar{x},\bar{y}) \in \mathrm{graph}~F$ with $\theta\in \mathbb{R}$ if there exist $c>0$ and a neighborhood $U$ of $\bar{x}$ such that $d(x,F^{-1}(\bar{y})) \leqslant c d(\bar{y},F(x))^\theta$ for all $x \in U$.
\end{definition}


We introduce two final definitions in order to further describe the variation of the objective function around a nonstrict local minimum.
\begin{definition}
\label{def:riemannian gradient}
\cite[Definition 3.30]{boumal2023introduction}
Let $f:\mathbb{R}^n\rightarrow \mathbb{R}$ be a locally Lipschitz function and $S\subset \mathbb{R}^n$ be a $C^p$ manifold at $x$. We say that $f$ is $C^p$ on $S$ at $x$ if there exists a neighborhood $U$ of $x$ and a $p$ times continuously differentiable function $\bar{f}:U \rightarrow \mathbb{R}$ such that $f(y) = \bar{f}(y)$ for all $y \in S\cap U$.
\end{definition}

According to \cite[Definition 3.58, Proposition 3.61]{boumal2023introduction}, the Riemannian gradient $\nabla_S f(x)$ of $f$ on $S$ at $x$ is given by $\nabla_S f(x):= P_{T_S(x)}(\nabla \bar{f}(x))$.

\begin{definition}
\label{def:verdier}
\cite[Definition 5 iii)]{bianchi2023stochastic} Let $f:\mathbb{R}^n\rightarrow \mathbb{R}$ be a locally Lipschitz function and let $S\subset \mathbb{R}^n$ be a $C^1$ manifold at a point $x^* \in \mathbb{R}^n$. Assume that $f$ is $C^1$ on $S$ at $x^*$. We say that $f$ satisfies the Verdier condition at $x^*$ along $S$ if there exist a neighborhood $U$ of $x^*$ and $c>0$ such that for all $y \in S\cap U$, $x \in U \setminus S$ and $v \in \partial f(x)$, we have $\left\|P_{T_S (y)}(v)-\nabla_{S} f(y)\right\| \leqslant c\|x-y\|$.
\end{definition}

The Verdier condition \cite[Equation (1.4)]{verdier1976stratifications} was introduced in 1976 to study the relationship between submanifolds arising in the Whitney stratification \cite{whitney1965tangents}. It was later shown that a finite family of definable sets always admits a Verdier stratification \cite[1.3 Theorem]{le1998verdier}, that is, for which the Verdier condition holds at every point on each stratum. Bianchi \textit{et al.} \cite{bianchi2023stochastic} and Davis \textit{et al.} \cite{davis2023active} recently used this condition to guarantee that a perturbed subgradient method on tilted functions with diminishing step size does not converge to active saddle points almost surely.

In the context of optimization, the Verdier condition poses a Lipschitz-like condition on the projection of the subgradients and the Riemannian gradient of the objective function along a $C^1$ manifold. Such a condition is reasonable since the domain of a continuous definable function always admits a Verdier stratification such that the function satisfies the Verdier condition at every point along each stratum \cite[Theorem 1]{bianchi2023stochastic} \cite[Theorem 3.29]{davis2023active}. However, the manifold induced by the critical points around a non-strict local minimum may not be contained in any strata, in which case the Verdier condition need not hold. It is for this reason that the Verdier condition appears as an assumption in Theorem \ref{thm:suff_unstable} below. We illustrate the Verdier condition with the following two examples, where $\|\cdot\|$ is induced by the Euclidean inner product. They are illustrated in Figures \ref{fig:verdier} and \ref{fig:not_verdier} respectively.
\begin{example}
    Let $f:\mathbb{R}^2\rightarrow \mathbb{R}$ be the function defined by $f(x_1,x_2) := |x_1 x_2 - 1|$. It satisfies the Verdier condition at $x^*:= (1,1)$ along its set of critical points $S := \{(x_1,x_2) \in \mathbb{R}^2 :x_1 x_2 = 1\} \cup \{(0,0)\}$. Consider the neighborhood $U:= B(x^*,0.5)$ of $x^*$. For all $(y_1,y_2) \in S\cup U$, we have that $T_S(y_1,y_2) = \{(x_1,x_2)\in \mathbb{R}^2:y_2 x_1 + y_1 x_2 = 0\}$ and $\nabla_S f(y_1,y_2) = (0,0)$. For all $(x_1,x_2) \in U\setminus S$, we have that $\partial f(x_1,x_2) = \{(\mathrm{sign}(x_1x_2 - 1)x_2, \mathrm{sign}(x_1x_2 - 1)x_1)\}$, where $\mathrm{sign}(t) = 1$ if $t>0$ and $\mathrm{sign}(t) = -1$ if $t<0$. Thus, for all $(y_1,y_2) \in S\cap U$, $(x_1,x_2) \in U \setminus S$ and $v \in \partial f(x_1,x_2)$, we have
 \begin{align*}
     \|P_{T_S(y_1,y_2)}(v) - \nabla_S f(y_1,y_2)\| &=\frac{|x_1 y_2 - x_2 y_1|}{\sqrt{y_1^2 + y_2^2}}\\
     &= \frac{|(x_1 - y_1)y_2 - (x_2 - y_2) y_1|}{\sqrt{y_1^2 + y_2^2}} \\
     &\leqslant \sqrt{(x_1 - y_1)^2 + (x_2 - y_2)^2}\\
     &= \|(x_1,x_2) - (y_1,y_2)\|
 \end{align*}
by the Cauchy-Schwarz inequality.
\end{example}
\begin{example}
Let $f:\mathbb{R}^2\rightarrow \mathbb{R}$ be the function defined by $f(x_1,x_2) := \max\{- x_1^2+2x_2 , |x_2|\}$, which is a slight modification of \cite[Example 3.1]{davis2023active}. It does not satisfy the Verdier condition at $x^*:= (0,0)$ along its set of critical points $S := \mathbb{R}\times \{0\}$. Indeed, consider the sequences $y^k := (1/k,0) \in S$, $x^k := (1/k,1/k^2) \notin S$ and $v^k:= (-2/k,2)$ defined for all $k \in \mathbb{N}^*$. They satisfy $y^k \rightarrow x^*$, $x^k \rightarrow x^*$, and $v^k \in \partial f(x^k)$, yet
\begin{equation*}
    \frac{\|P_{T_S(y^k)}(v^k) - \nabla_S f(y^k)\|}{\|x^k - y^k\|} = \frac{\|(-2/k,0) - (0,0)\|}{\|(1/k,1/k^2) - (1/k,0)\|} = \frac{2/k}{1/k^2} \rightarrow \infty.
\end{equation*}
\end{example}

\begin{figure}[ht]
\centering
\begin{subfigure}{.49\textwidth}
  \centering
  \includegraphics[width=1\textwidth]{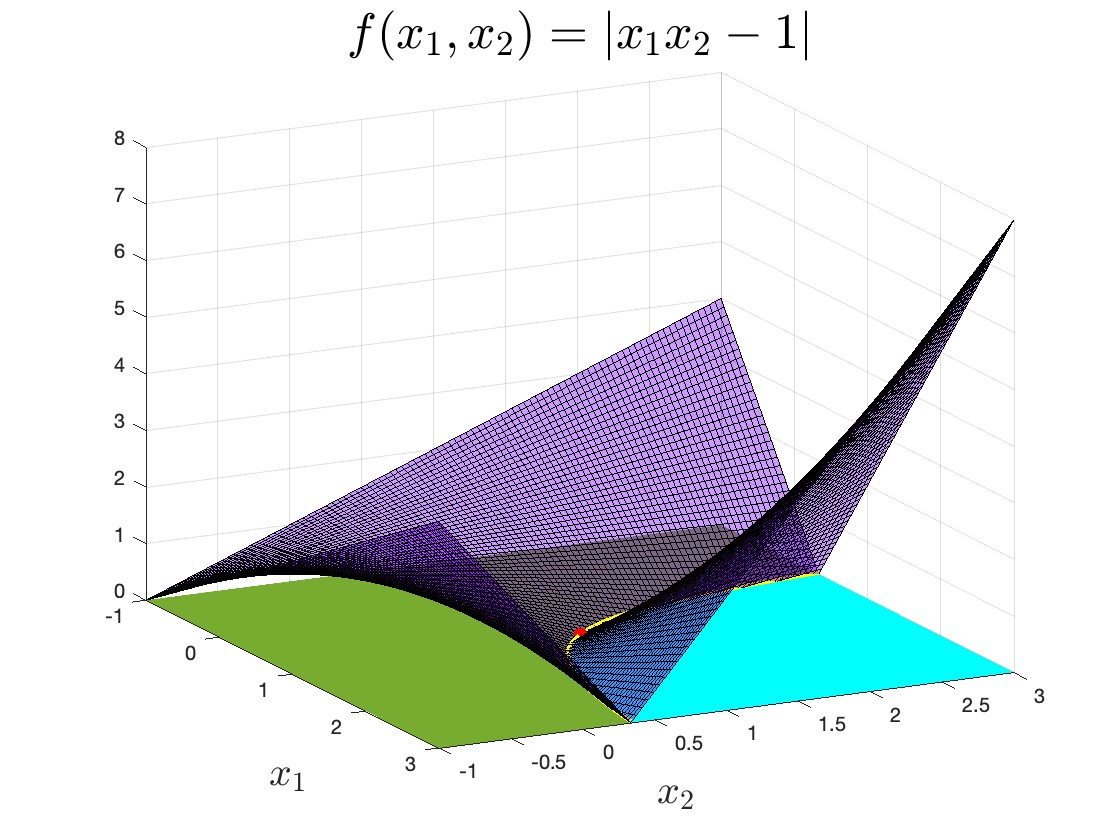}
  \caption{Verdier condition verified at $(1,1)$ along manifold of critical points.}
  \label{fig:verdier}
\end{subfigure}
\begin{subfigure}{.49\textwidth}
  \centering
  \includegraphics[width=1\textwidth]{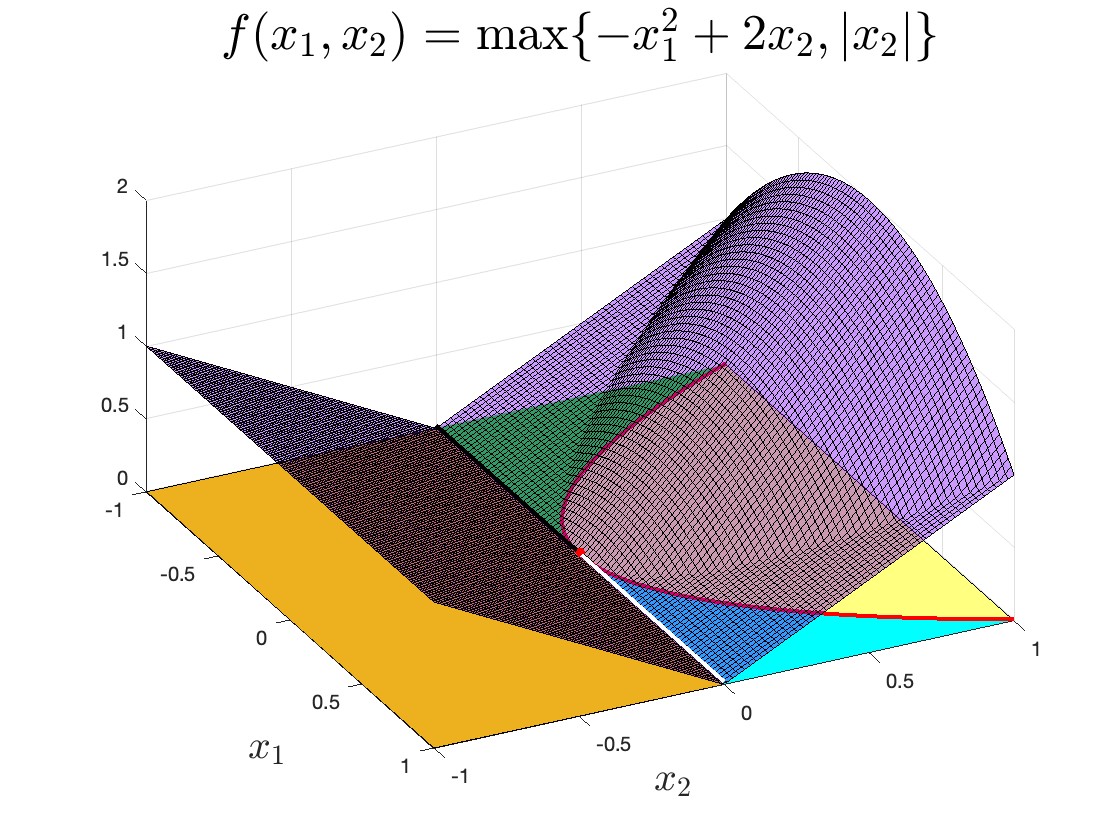}
  \caption{Verdier condition violated at $(0,0)$ along manifold of critical points.}
    \label{fig:not_verdier}
\end{subfigure}
\caption{Verdier stratification of the domain of two continuous semi-algebraic functions.}
\end{figure}


Before we state the main result of this chapter, we show that  the subgradient method can avoid any null set, building on the arguments in \cite{bolte2020mathematical}.

\begin{lemma}\label{lemma:avoidance}
	Let $f:\mathbb{R}^n\rightarrow \mathbb{R}$ be a locally Lipschitz definable function. There exist $\alpha_1,\hdots,\alpha_m>0$ such that for any constant step size $\alpha \in (0,\infty) \setminus \{\alpha_1,\hdots,a_m\}$ and for any null set $S\subset \mathbb{R}^n$, there exists a null subset $I_\alpha \subset \mathbb{R}^n$ such that, for every initial point $x_0 \in \mathbb{R}^n \setminus I_\alpha$, none of the iterates $x_0,x_1,x_2,\hdots$ of the subgradient method belong to $S$.
\end{lemma}
\begin{proof}
	Consider the set-valued mapping $G_\alpha:\mathbb{R}^n \rightrightarrows \mathbb{R}^n$ defined by $G_\alpha(x) := x - \alpha \partial f(x)$ for some $\alpha>0$, as well as its preimage $G_\alpha^{-1}(S) := \{ x \in \mathbb{R}^n : G_\alpha(x) \cap S \neq \emptyset \}$. We may then view $I_\alpha$ as the union of $G^{-k}_\alpha(S)$ over all $k \in \mathbb{N}$, where $G_\alpha^{-k}(S) := G_\alpha^{-1}(\hdots G_\alpha^{-1}(S) \hdots )$ is obtained by taking $k$ times the preimage of $S$. The Lebesgue measure of $I_\alpha$ is thus bounded above by the sum of the Lebesgue measures of $G^{-k}_\alpha(S)$. Below, we will show that they simultaneously have zero Lebesgue measure for all but finitely many $\alpha$.
	
	By the cell decomposition theorem \cite[(2.11) p. 52]{van1998tame}, there exists a definable open dense subset $\Omega$ of $\mathbb{R}^n$ such that $f$ is twice continuously differentiable on $\Omega$.	  Denote by $\lambda_{i}:\Omega \rightarrow \mathbb{R}$ the function that maps any $x\in \Omega$ to the $i$th largest eigenvalue of $\nabla^2 f(x)$ for $i = 1,\ldots,n$. Since $\lambda_1,\ldots,\lambda_n$ are definable, again by the cell decomposition theorem, they are continuously differentiable on a definable open dense subset of $\Omega$ (again denoted by $\Omega$). 
	
	Let $g_\alpha:\Omega\rightarrow \mathbb{R}^n$ be the restriction of $G_\alpha$ on $\Omega$ (it is single-valued by continuous differentiability of $f$ on $\Omega$).  We next show that for all but finitely many step size $\alpha$, it holds that if $S\subset \mathbb{R}^n$ is a null set, then so is $g_{\alpha}^{-1}(S)$, following arguments similar to \cite[Claim 3]{bolte2020mathematical}. Indeed, by the definable Morse-Sard theorem \cite[Corollary 9]{bolte2007clarke}, $\lambda_i$ has finitely many critical values. Thus, $F:=\{\alpha >0: 1/\alpha\text{ is a critical value of $\lambda_i$ for some $i$}\}$ is finite. Let $\alpha_1,\ldots,\alpha_m$ be the elements of $F$. Fix any $\alpha \in (0,\infty)\setminus \{\alpha_1,\ldots,\alpha_m\}$ and any null subset $S$ of $\mathbb{R}^n$. We next show that $g_\alpha^{-1}(S)$ is null. Consider
	\begin{subequations}
		\begin{align*}
			K_\alpha :&= \left\{x\in \Omega: g'_\alpha(x)\text{ is not invertible}\right\}\\
			&=\left\{x\in \Omega: \mathrm{det}\left(I -\alpha \nabla^2 f(x) \right)= 0\right\}\\
			&=\bigcup_{i = 1}^n \left\{x\in \Omega: 1-\alpha\lambda_i(x) = 0\right\},
		\end{align*}
	\end{subequations}
	where $I$ is the identity matrix of order $n$ and the last equality follows from the diagonalization of $g'_\alpha(x)$. As $1/\alpha$ is not a critical value of $\lambda_i$, $\left\{x\in \Omega: 1-\alpha\lambda_i(x) = 0\right\}$ is null for $i = 1,\ldots,n$, according to the regular level set theorem \cite[Corollary 5.14]{lee2013smooth}. Thus, $K_\alpha$ is a null subset of $\Omega$. Note that $\Omega\setminus K_\alpha$ is open in $\Omega$ (and thus in $\mathbb{R}^n$) by continuity of $g'_\alpha$. Therefore, $g_\alpha^{-1}(S)\cap (\Omega\setminus K_\alpha)$ is null, as $g_\alpha$ is a diffeomorphism when restricted to $\Omega\setminus K_\alpha$. It follows that $g_\alpha^{-1}(S) = (g_\alpha^{-1}(S)\cap K_\alpha) \cup (g_\alpha^{-1}(S)\cap (\Omega\setminus K_\alpha)) \subset K_\alpha \cup  (g_\alpha^{-1}(S)\cap (\Omega\setminus K_\alpha))$ is null.

	  This property continues to hold without restricting the domain of $G_\alpha$ on $\Omega$. Indeed, for any null set $S\subset \mathbb{R}^n$, $G_\alpha^{-1}(S)=  (G_\alpha^{-1}(S)\cap \Omega) \cup (G_\alpha^{-1}(S)\cap(\mathbb{R}^n\setminus \Omega))\subset g_\alpha^{-1}(S)\cup (\mathbb{R}^n\setminus \Omega)$ is null. By induction, so are $G_\alpha^{-1}(S)$, $G_\alpha^{-2}(S), \hdots$ This concludes the proof.
\end{proof}

We are now ready to state the main theorem of this chapter, which provides sufficient conditions for strong instability.

\begin{theorem}
\label{thm:suff_unstable}
Let $f:\mathbb{R}^n\rightarrow \mathbb{R}$ be a locally Lipschitz tame function whose set of critical points we denote by $S$. Assume that $S$ is a $C^2$ manifold at some $x^*\in S$ of dimension less than $n$. Assume that there exist $\theta_1 \geqslant 0$, a neighborhood $U$ of $x^*$, and a continuous function $C:\mathbb{R}^n \rightarrow \mathbb{R}$ such that for all $\alpha>0$, there exist $c_1>0$ such that for any sequence $x_0,x_1,\hdots \in U \setminus S$ generated by the subgradient method with constant step size $\alpha$, we have $C(x_{k+1}) -  C(x_k) \geqslant c_1 d(x_k,S)^{\theta_1}$ for all $k\in \mathbb{N}$. The point $x^*$ is strongly unstable if 1) $\theta_1 = 0$ or 2) $\partial f$ is metrically $\theta_2$-subregular at $(x^*,0)$ with $\theta_2>1$ and $f$ satisfies the Verdier condition at $x^*$ along $S$.
\end{theorem}

\begin{proof} We begin with an outline of the proof. In order to establish instability, we reason by contradiction and assume that the iterates of the subgradient method remain in a neighborhood of a fixed critical point. We show that this implies that the function $C$ becomes unbounded along the iterates, which is impossible since this function is continuous. The key to showing unboundedness is to prove divergence of a series whose terms depend on the distance of the iterates to the manifold of critical points. For the proof to work, this distance should be positive for all iterates. We hence begin the proof by ensuring that this holds almost surely, using \cref{lemma:avoidance}. After treating an easy case, the majority of the proof is devoted to showing that the distance to the set of critical points does not converge to zero.

We seek to show that there exists $\epsilon>0$ such that for all but finitely many constant step sizes $\alpha>0$, there exists a null subset $I_\alpha \subset \mathbb{R}^n$ such that for every initial point $x_0 \in B(x^*,\epsilon) \setminus I_\alpha$, at least one of the iterates of the subgradient method does not belong to $B(x^*,\epsilon)$. Since $S$ is a $C^2$ manifold at $x^*$ of dimension less than $n$, we have that $S\cap U$ is a null set after possibly reducing $U$. As we assume the iterates remain in $U$, we may assume that $f$ is definable (by replacing $f$ with a definable extension of $f_{|U}$). By \cref{lemma:avoidance}, there exist $\alpha_1,\hdots,\alpha_m>0$ such that for any constant step size $\alpha \in (0,\infty) \setminus \{\alpha_1,\hdots,a_m\}$, there exists a null subset $I_\alpha \subset \mathbb{R}^n$ such that, for every initial point $x_0 \in \mathbb{R}^n \setminus I_\alpha$, none of the iterates $x_0,x_1,x_2,\hdots$ of the subgradient method belong to the null set $S\cap U$.

\underline{Case 1: Assume that $\theta_1 = 0$.} Let $\epsilon >0$ such that $B(x^*,\epsilon) \subset U$. Let $\alpha \in (0,\infty) \setminus \{\alpha_1,\hdots,a_m\}$ and consider a sequence of iterates $x_0,x_1,x_2,\hdots \in \mathbb{R}^n$ of the subgradient method with constant step size $\alpha$ such that $x_0 \in B(x^*,\epsilon) \setminus I_\alpha$. We reason by contradiction and assume that $x_k \in B(x^*,\epsilon)$ for all $k \in \mathbb{N}$. Thus $x_k \notin S$ for all $k \in \mathbb{N}$. We have $C(x_{k+1}) - C(x_k) \geqslant c_1 d(x_{k},S)^{\theta_1}$ and
\begin{equation}
\label{eq:diverge}
    C(x_K) - C(x_0) = \sum\limits_{k = 0}^{K-1} C(x_{k+1}) - C(x_k) \geqslant \sum\limits_{k = 0}^{K-1} c_1 d(x_{k},S)^{\theta_1} = \sum\limits_{k = 0}^{K-1} c_1,
\end{equation}
which converges to $+\infty$ as $K$ converges to $+\infty$. Since $C$ is continuous and $x_K \in B(x^*,\epsilon)$, this yields a contradiction.

\underline{Case 2: Assume that $\theta_1>0$.} We proceed in four steps. We begin by choosing $\epsilon>0$ sufficiently small so that the objective function admits favorable geometric properties in $B(x^*,2\epsilon)$ (step 1). We then use these properties, including metric $\theta_2$-subregularity, to show that $d(x_{k+1},S) \geqslant d(x_k,S)$ whenever $d(x_k,S)$ is small enough (step 2). This prevents $d(x_k,S)$ from converging to zero. Similar to \eqref{eq:diverge}, this leads to a divergent series $\sum_{k = 0}^{\infty} c_1 d(x_{k},S)^{\theta_1}$ and hence to a contradiction. A computation reveals that proving the inequality on the distances reduces to showing that a certain ratio is bounded (step 3), at which point we invoke the Verdier condition. This in turn requires showing that the projection 
is preserved when taking a step of a slight modification of the subgradient method (step 4). 

\underline{Step 1} We begin by choosing $\epsilon>0$ such that the projection $P_S$ onto $S$ is Lipschitz continuous and identifies on $B(x^*,2\epsilon)$ with the preimage of a mapping related to the normal cone $N_S(x)$, among other properties.

Since $S$ is a $C^2$ manifold at $x^*$
, $S\cap U$ is strongly amenable \cite[10.23 Definition (b)]{rockafellar2009variational} after possibly reducing $U$. It follows that $S\cap U$ is prox-regular \cite[13.31 Exercise, 13.32 Proposition]{rockafellar2009variational} and locally closed \cite[p. 28]{rockafellar2009variational}. Therefore, there exists a closed neighborhood $V \subset U$ of $x^*$ such that $S\cap V$ is closed and prox-regular at $x^*$. By \cite[Theorem 1.3 (j)]{poliquin2000local}, there exists $\epsilon>0$ such that the projection $P_{S\cap V}$ onto $S\cap V$ is single-valued and Lipschitz continuous with some constant $L>0$ on $B(x^*,2\epsilon)$. After possibly reducing $\epsilon>0$, we have $P_{S\cap V}(x) = P_S(x)$ for all $x\in B(x^*,2\epsilon)$. (Indeed, if $B(x^*,5\epsilon) \subset V$, then $\|x-y\|\geqslant 3\epsilon$ for all $y \in S \setminus V$ while $\|x-x^*\|\leqslant 2\epsilon$.) Again by \cite[Theorem 1.3 (j)]{poliquin2000local}, there exists $c>0$ such that $P_S(x) = (I+N_S^c)^{-1}(x)$ for all $x \in B(x^*,2\epsilon)$, where $N_S^c$ is a set-valued mapping defined from $\mathbb{R}^n$ to the subsets of $\mathbb{R}^n$ by
\begin{equation}
\label{eq:N_S^c}
     N_S^c(x) := \left\{\begin{array}{ll}
         N_S(x) \cap \mathring{B}(0,c) & \text{if $x\in S$}, \\
         \emptyset & \text{else}.
    \end{array}\right.
\end{equation}
After possibly reducing $\epsilon$, we may assume that $(2+L)\epsilon < c$. 

In the following, we further reduce $\epsilon$ whenever necessary.
Since $\partial f$ is metrically $\theta_2$-subregular at $(x^*,0)$, there exists $c_2>0$ such that $d(x,S) \leqslant c_2 d(0,\partial f(x))^{\theta_2}$ for all $x \in B(x^*,\epsilon)$. Since $f$ satisfies the Verdier condition at $x^*$ along $S$, there exists $c_3>0$ such that for all $y \in B(x^*,2\epsilon)\cap S$, $x \in B(x^*,2\epsilon) \setminus S$ and $v \in \partial f(x)$, we have $\left\|P_{T_{
S}(y)}(v)\right\| \leqslant c_3\|x-y\|$. Indeed, $\nabla_S f(y) = 0$ for all $y\in B(x^*,2\epsilon) \cap S$ because $f$ agrees with a constant function along $S$ around $x^*$ by the definable Morse-Sard theorem \cite[Corollary 9]{bolte2007clarke}.

\underline{Step 2} Having chosen $\epsilon>0$, let $\alpha \in (0,\infty) \setminus \{\alpha_1,\hdots,a_m\}$. Consider a sequence of iterates $x_0,x_1,x_2,\hdots$ of the subgradient method with constant step size $\alpha$ such that $x_0 \in B(x^*,\epsilon) \setminus I_\alpha$. As in the case when $\theta_1 = 0$, we reason by contradiction and assume that $x_k \in B(x^*,\epsilon)$ for all $k \in \mathbb{N}$. Thus $x_k \notin S$ for all $k \in \mathbb{N}$. Also, for all $K \in \mathbb{N}$, we have $C(x_K) - C(x_0) \geqslant \sum_{k=0}^{K-1} c_1 d(x_{k},S)^{\theta_1}$. In order to show that $\sum_{k=0}^{K-1} c_1 d(x_{k},S)^{\theta_1}$ diverges, it suffices to show that $d(x_k,S)$ does not converge to zero. To this end, we next show that $d(x_{k+1},S) \geqslant d(x_k,S)$ whenever $d(x_k,S)>0$ is sufficiently small (it is non-zero because $x_k \notin S$ and $S$ is locally closed).

\begin{figure}[ht]
\centering
    \begin{tikzpicture}[scale=2.4]
\fill[gray!50] (5.5,.9) circle (38pt); 
    \draw[line width=.2mm,dashed] (5.5,.9) circle (38pt);
    \fill (3.15,0.85) node{\normalsize $S$};
    \draw[line width=.2mm] plot [smooth,tension=0.689] coordinates {(3.2,0.6)(3.5,1.0)(4.17,0.95)};
    \draw[magenta,line width=.4mm, name path=A] plot [smooth,tension=0.689] coordinates {(4.17,0.95) (5.5,.9) (6.57,1.7)};
    \draw[line width=.2mm] plot [smooth,tension=0.689] coordinates {(6.57,1.7)(6.8, 2)(7,2.5)};
    \draw[magenta,->] (6.55,2.1)-- (6.4,1.6);
    \fill (6.55,2.2)  node {\normalsize $\textcolor{magenta}{S\cap U}$};

    \fill (4.25,1.8)  node {\normalsize $U$};
    \filldraw (5.5,.9) circle (.3pt);
    \fill (5.54,1.01)  node {\normalsize $x^*$};

     \filldraw (5.622,0.945) circle (.3pt);
    \fill (5.5,0.65)  node {\small $P_S(x_{k+1})$};
     \draw[->](5.522,0.695)--(5.61,0.895);
     \filldraw (5.8,1.04) circle (.3pt);
     \fill (5.9,1.39)  node {\small $P_S(x_k)$};
     \draw[->](5.89,1.305)--(5.82,1.105);
     \draw[blue,opacity=0.5,line width=.4mm] plot [smooth,tension=1] coordinates {(5,2.4)(6.5,-0.15)};
     \fill (5.5,2.4)  node {\normalsize $\textcolor{blue}{N_k + x_k}$};
     \draw[blue,->] (5.5,2.3)-- (5.3,2);
     \draw[blue,opacity=0.5,line width=.4mm] plot [smooth,tension=1] coordinates {(4.5,-0.1824)(7,1.2882)};
     \fill (4,0) node {\normalsize $\textcolor{blue}{T_k + x_k}$};
     \draw[blue,->] (4.3,0)-- (4.7,0);
       \draw[line width=.4mm,violet,->] (6,.7)-- (6.4,.16);
       \fill (6.34,.48)  node {\normalsize $\textcolor{violet}{\boldsymbol{v_k}}$}; 
       \draw[yellow,dashed,line width=.2mm] plot [smooth, tension=1] coordinates {(5.3,1.6450)(6,.7)};
       
    \filldraw[yellow] (6,.7) circle (.5pt);
      \fill (5.96,.54)  node {\normalsize $x_k$};
       \filldraw[yellow] (5.3,1.6450) circle (.5pt);
       \fill (5.1,1.6)  node {\normalsize $x_{k+1}$};
    \end{tikzpicture}
\caption{Illustration of $d(x_{k+1},S) \geqslant d(x_k,S)$ for $d(x_k,S)$ sufficiently small.}
\label{fig:chetaev}
\end{figure}
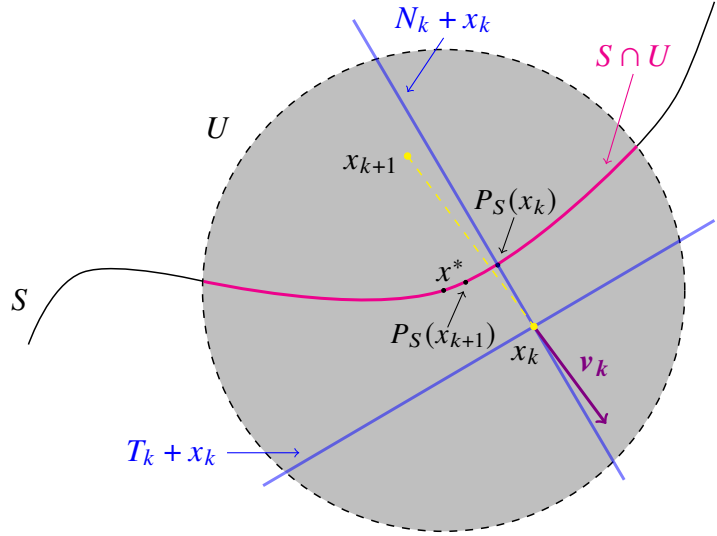

Since $(x_k)_{k\in \mathbb{N}}$ is generated by the subgradient method with constant step size $\alpha$, for all $k \in \mathbb{N}$ there exists $v_k \in \partial f(x_k)$ such that $x_{k+1} = x_k - \alpha v_k$. 
As illustrated in Figure \ref{fig:chetaev}, we have
\begin{subequations}
    \begin{align}
        d(x_{k+1},S) & = \|x_{k+1} - P_S(x_{k+1})\| \label{dist_a} \\
        & = \|x_k - P_S(x_{k+1}) - \alpha v_k\| \label{dist_b} \\
        & \geqslant \alpha \|v_k\| - \|x_k-P_S(x_{k+1})\| \label{dist_c} \\
        & \geqslant \alpha c_2^{-1/\theta_2} d(x_k,S)^{1/\theta_2}- \|x_k-P_S(x_{k+1})\| \label{dist_d} \\
        & = d(x_k,S)\left(\alpha c_2^{-1/\theta_2} d(x_k,S)^{1/\theta_2-1} - \frac{\|x_k-P_S(x_{k+1})\|}{d(x_k,S)}\right) \label{dist_e}\\
        & \geqslant d(x_k,S) \label{dist_f}
    \end{align}
\end{subequations}
provided that $d(x_k,S)$ sufficiently small and that $\|x_k-P_S(x_{k+1})\|/d(x_k,S)$ is upper bounded on $B(x^*,\epsilon) \setminus S$ if $d(x_k,S)$ sufficiently small. Indeed, in \eqref{dist_a} $P_S(x_{k+1})$ is a singleton because $x_{k+1} \in B(x^*,\epsilon)$. \eqref{dist_b}-\eqref{dist_c} are deduced from the update rule and the triangular inequality. \eqref{dist_d} follows from the metric $\theta_2$-subregularity of $\partial f$ at $x^*$ for $0$. In the second factor of \eqref{dist_e}, the first term $c_2^{-1/\theta_2} d(x_k,S)^{1/\theta_2-1}$ diverges as $d(x_k,S)$ nears zero because $1/\theta_2-1< 0$. Hence, if the second term $\|x_k-P_S(x_{k+1})\|/d(x_k,S)$ is bounded, then the lower bound \eqref{dist_f} holds. 

\underline{Step 3} We next focus on proving that $\|x_k-P_S(x_{k+1})\|/d(x_k,S)$ is bounded. Since $x_k \in B(x^*, \epsilon)\setminus S$, $P_S(x_k) \in B(x^*,2\epsilon)\cap S$, and $v_k \in \partial f(x_k)$, by the Verdier condition we have $\|P_{T_S(P_S(x_k))}(v_k)\| \leqslant c_3 \|x_k - P_S(x_k)\|$ for all $k\in \mathbb{N}$. For notational convenience, let $T_k:= T_S(P_{S}(x_k))$ and $N_k:= N_S(P_{S}(x_k))$ respectively be the tangent and normal cones of $S$ at $P_S(x_k)$. Also, let $v^T_k := P_{T_k}(v_k)$ and $v^N_k := P_{N_k}(v_k)$ respectively be the projections of $v_k$ on $T_k$ and $N_k$. With these notations, we have $\|v_k^T\| \leqslant c_3\|x_k - P_{S}(x_k)\|$ for all $k\in \mathbb{N}$. Observe that
\begin{subequations}
    \begin{align}
        \frac{\|x_k-P_S(x_{k+1})\|}{d(x_k,S)} & \leqslant \frac{\|x_k-P_S(x_k)\| + \|P_S(x_k)-P_S(x_{k+1})\|}{d(x_k,S)} \label{proj_a} \\
        & = 1 + \frac{\|P_S(x_k)-P_S(x_{k+1})\|}{\|x_k-P_S(x_k)\|} \label{proj_b} \\
        & = 1 + \frac{\|P_S(x_k-\alpha v_k^N)-P_S(x_k-\alpha v_k)\|}{\|x_k-P_S(x_k)\|} \label{proj_c} \\
        & \leqslant 1 + L \alpha \frac{\|v_k^N-v_k\|}{\|x_k-P_S(x_k)\|} \label{proj_d} \\
        & \leqslant 1 + L \alpha c_3 \label{proj_e}
    \end{align}
\end{subequations}
provided that $d(x_k,S)$ sufficiently small. Indeed, \eqref{proj_a} follows from the triangular inequality. \eqref{proj_b} holds because $d(x_k,S) = \|x_k-P_S(x_k)\|$.  
\eqref{proj_c} holds because of the update rule $x_{k+1} = x_k-\alpha v_k$ and the fact that $P_S(x_k) = P_S(x_k-\alpha v_k^N)$, which is the object of the next step. \eqref{proj_d} holds because $P_S$ is $L$-Lipschitz continuous in $B(x^*,2\epsilon)$. Finally, \eqref{proj_e} follows from the Verdier condition and the fact that $v_k = v_k^T + v_k^N$.

\underline{Step 4} It remains to prove that $P_S(x_k) = P_S(x_k-\alpha v_k^N)$ when $d(x_k,S)$ is sufficiently small. We may thus assume that $d(x_k,S) \leqslant \epsilon/(\alpha c_3)$, which guarantees that $x_k-\alpha v_k^N \in B(x^*,2\epsilon)$. Indeed, 
\begin{subequations}
    \begin{align}
        \| x_k-\alpha v_k^N - x^*\| & \leqslant \| x_k - \alpha v_k - x^*\| + \alpha \| v_k^T\| \\
        & \leqslant \epsilon + \alpha c_3 \|x_k - P_S(x_k)\| \\
        & \leqslant \epsilon + \alpha c_3 d(x_k,S) \\
        & \leqslant 2\epsilon.
    \end{align}
\end{subequations}
Recall that $P_S(x) = (I+N_S^c)^{-1}(x)$ for all $x \in B(x^*,2\epsilon)$. We have
\begin{subequations}
    \begin{align}
        P_S( x_k-\alpha v_k^N ) & = (I+N_S^c)^{-1}(x_k-\alpha v_k^N) \label{normal_a} \\
        & = (I+N_S^c)^{-1}(P_S(x_k) + x_k - P_S(x_k) - \alpha v_k^N) \label{normal_b} \\
        & = P_S(x_k). \label{normal_c}
    \end{align}
\end{subequations}
Indeed, \eqref{normal_c} is equivalent to $P_S(x_k) + x_k - P_S(x_k) - \alpha v_k^N \in (I+N_S^c)(P_S(x_k))$, that is to say, $x_k - P_S(x_k) - \alpha v_k^N \in N_S^c(P_S(x_k))$. Since $P_S(x_k) \in S$, by definition of $N_S^c$ in \eqref{eq:N_S^c}, $N_S^c(P_S(x_k)) = N_k \cap \mathring{B}(0,c)$. To see why $x_k - P_S(x_k) - \alpha v_k^N \in N_k$, observe that $P_S(x_k) = P_S(x_k - P(x_k) + P(x_k)) = (I+N_S^c)^{-1}(x_k - P(x_k) + P(x_k))$. Thus $x_k - P(x_k) + P(x_k) \in (I+N_S^c)(P_S(x_k))  = P_S(x_k) + N_S^c(P_S(x_k))$, that is to say, $x_k - P(x_k) \in N_k$. Since $N_k$ is a linear subspace, it follows that $x_k - P_S(x_k) - \alpha v_k^N \in N_k$. Finally, 
    \begin{align*}
        \|x_k - P_S(x_k)-\alpha v_k^N\| & \leqslant \|x_k -\alpha v_k^N - x^*\| + \|x^* - P_S(x_k)\|\\
        & \leqslant 2\epsilon +  \|P_S(x^*) - P_S(x_k)\|\\
        & \leqslant 2\epsilon + L\|x^* - x_k\|\\
        & \leqslant (2 + L) \epsilon<c.\qedhere
    \end{align*}
\end{proof}

 Recall that a local minimum $x^* \in \mathbb{R}^n$ of $f$ is spurious if $f(x^*)>\inf\{f(x):x\in \mathbb{R}^n\}$. In \cref{sec:Applications}, Theorem \ref{thm:suff_unstable} will be used to prove instability of spurious local minima in two practical problems; see Propositions \ref{prop:neural_network_unstable} and \ref{prop:rpca_rankr_unstable}. Recall that Lyapunov functions are used in the theory of ordinary differential equations (or inclusions) to prove the stability of an equilibrium point \cite{liapounoff1907probleme}. For example, a locally Lipschitz tame objective function $f$ is a Lyapunov function for the continuous-time subgradient dynamics $x'\in -\partial f(x)$ around a strict local minimum $x^*$ \cite[Theorem 5.16 1.]{sastry2013nonlinear}. Indeed, $f$ is positive around $x^*$ and $f\circ x$ is decreasing along any trajectory $x$. The objective function is however not monotonic along discrete-time dynamics, in which case it ceases to be a Lyapunov function.

In contrast to Lyapunov functions, Chetaev functions are used to prove instability \cite{chetaev1961stability}. By \cite[Theorem 2.14]{braun2021stability}, an equilibrium point of ordinary differential equation is unstable if there exists a continuous function with positive values in any neighborhood of the equilibrium where it is equal to zero, and it is increasing along any trajectory (see also \cite[Theorems 5.29 and 5.30]{sastry2013nonlinear}). Chetaev functions have gained renewed interest recently in the context of obstacle avoidance in control, where one seeks to render the obstacles unstable by feedback \cite[Section III. B.]{braun2019uniting} \cite[Section IV. B.]{braun2018complete}. 
The function $C:\mathbb{R}^n \rightarrow \mathbb{R}$ in Theorem \ref{thm:suff_unstable} plays the role of a Chetaev function in a neighborhood of the point $x^*$. So long as the iterates $x_k$ stay near $x^*$ and avoid the critical points, the Chetaev function values $C(x_k)$ increase
. If the increase is lower bounded by a positive constant at every iteration (i.e., when $\theta_1=0$), then we may readily conclude. Otherwise, the local geometry of the objective function comes into play. 

The fact that the exponent in the metric subregularity of the subdifferential is greater than one prevents the objective function from having a locally Lipschitz gradient if it is differentiable. The Verdier condition characterizes how fast subgradients become normal to the set of critical points in the vicinity of $x^*$. Together, these two conditions ensure that the iterates of the subgradient method do not converge to the set of critical points around $x^*$. Then the values $C(x_k)$ converge to plus infinity if the iterates remain near $x^*$, resulting in instability. In contrast, Bianchi \textit{et al.} \cite[Proposition 4]{bianchi2023stochastic} and Davis \textit{et al.} \cite[Proposition 5.2]{davis2023active} use the Verdier condition to ensure that the projection of the iterates on an active manifold containing a saddle point correspond to an inexact Riemannian gradient method with an implicit retraction. This technique is thus not suitable for proving instability of local minima.

In order to avoid assuming that the inequality $C(x_{k+1}) -  C(x_k) \geqslant c_1 d(x_k,S)^{\theta_1}$ holds for all $k \in \mathbb{N}$ in Theorem \ref{thm:suff_unstable}, one may require the Chetaev function to be convex and $\langle s , s' \rangle \leqslant  -d(x,S)^{\theta_1}$ for all $x\in U \setminus S, s \in \partial C(x),$ and $s' \in \partial f(x)$. Indeed, we then have $C(x_{k+1}) - C(x_k) \geqslant \langle s_k,x_{k+1}-x_k\rangle = \langle s_k,-\alpha s_k'\rangle \geqslant \alpha d(x,S)^{\theta_1}$ where $s_k \in \partial C(x_k)$ and $s_k' \in \partial f(x_k)$. These slightly stronger conditions hold in the first example in the next section.

\section{Applications}
\label{sec:Applications}


In this section, we apply Theorem \ref{thm:suff_unstable} to two examples using the Euclidean inner product. We first show that instability occurs in an example of ReLU neural network with $\ell_1$ loss, namely $(x_1,x_2,x_3)\in \mathbb{R}^3 \mapsto  |x_3 \max\{x_2,0\} - 1| + |x_3 \max\{x_1 + x_2,0\}|$. Indeed, it is the loss function when one seeks to fit the ReLU neural network $(a_1,a_2) \in \mathbb{R}^2 \mapsto x_3 \max\{x_1 a_1 + x_2 a_2,0\}$ over two data points $(0,1)$ and $(1,1)$ with corresponding labels $1$ and $0$. Figure \ref{fig:nn_relative} reveals that the iterates of the subgradient method move away from a fixed spurious local minimum despite being initialized nearby. Five trials are displayed, each corresponding to a uniform choice of constant step size in $[0.05,0.15]$ and a random initial point within $10^{-3}$ relative distance of the local minimum. Figure \ref{fig:nn_lyapunov} shows the corresponding values of an associated Chetaev function $C:\mathbb{R}^3\rightarrow\mathbb{R}$ defined by $C(x_1,x_2,x_3) := 1-x_1$. The fact that this function must increase indefinitely if the iterates remain near the local minimum is at the root of the instability (see Proposition \ref{prop:neural_network_unstable}). Figure \ref{fig:nn_fv} shows that the objective function values eventually stabilize around the global minimum value.

\begin{figure}[ht]
    \centering
     \begin{subfigure}{.49\textwidth}
  \centering
  \includegraphics[width=.95\textwidth]{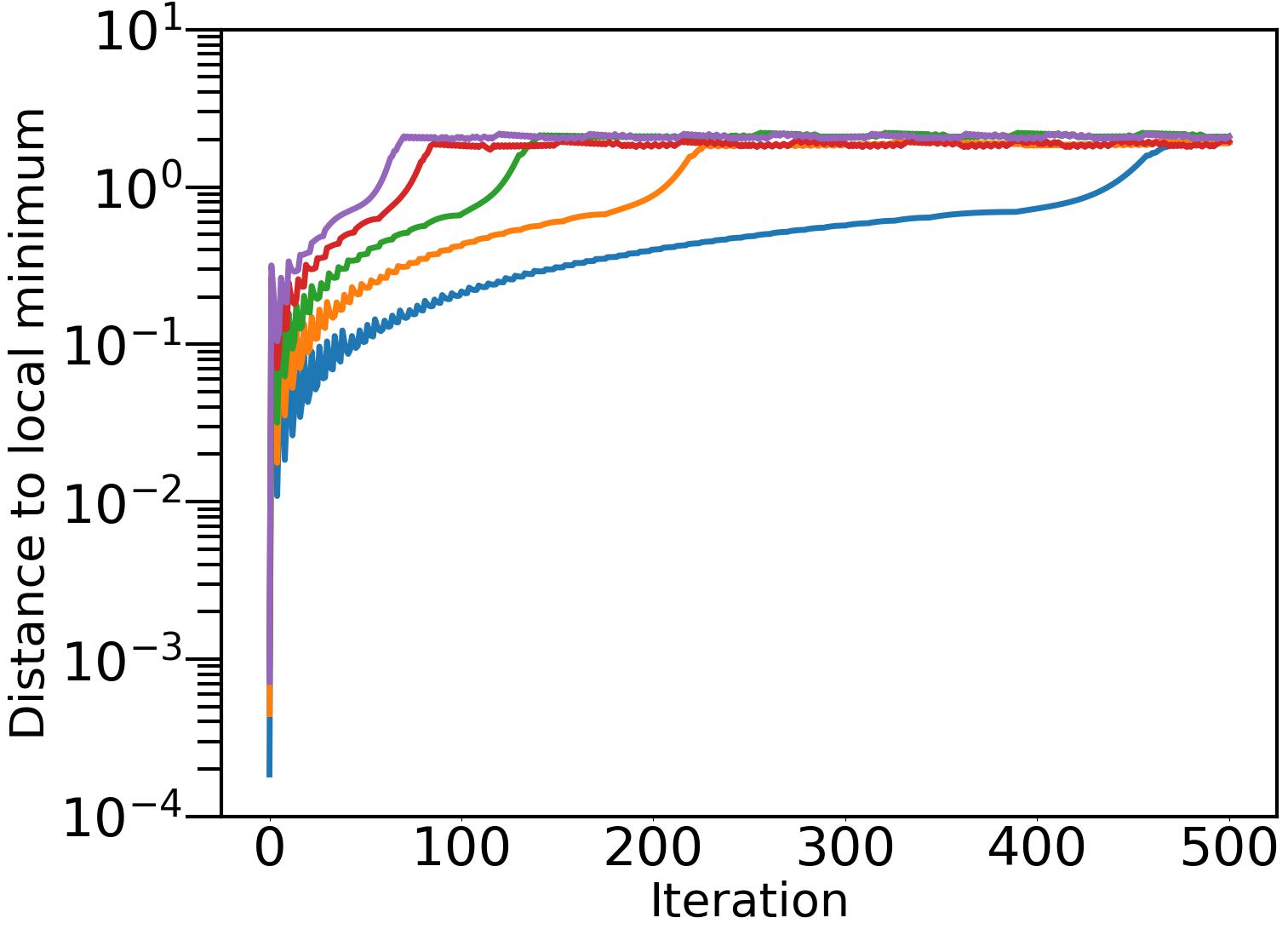}
  \caption{}
  \label{fig:nn_relative}
\vspace*{1mm}
  \label{fig:nn_distance}
 \end{subfigure}
 \begin{subfigure}{.49\textwidth}
  \centering
  \includegraphics[width=.95\textwidth]{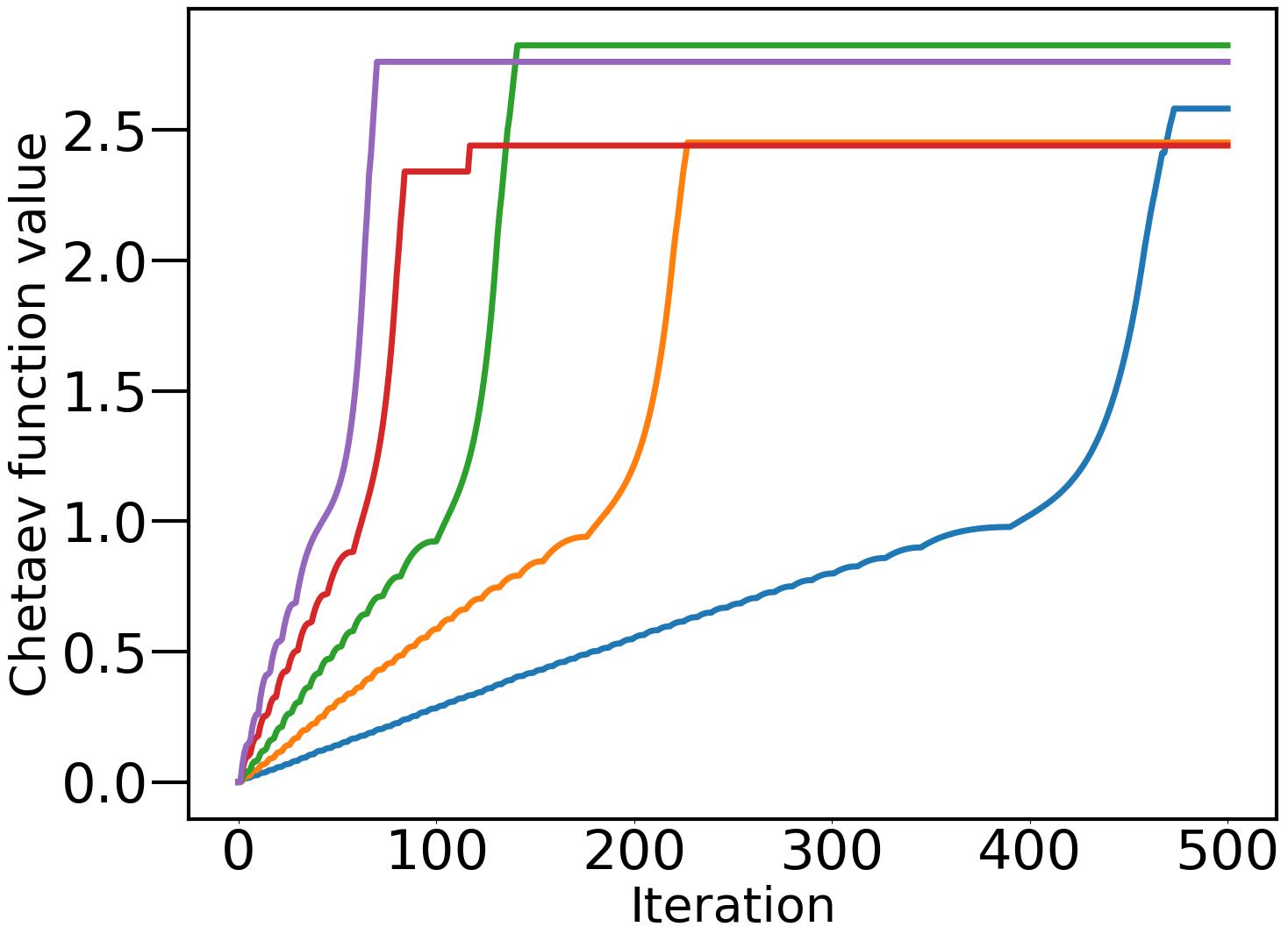}
  \caption{}
  \label{fig:nn_lyapunov}
 \end{subfigure}
  \begin{subfigure}{.49\textwidth}
  \centering
  \includegraphics[width=.95\textwidth]{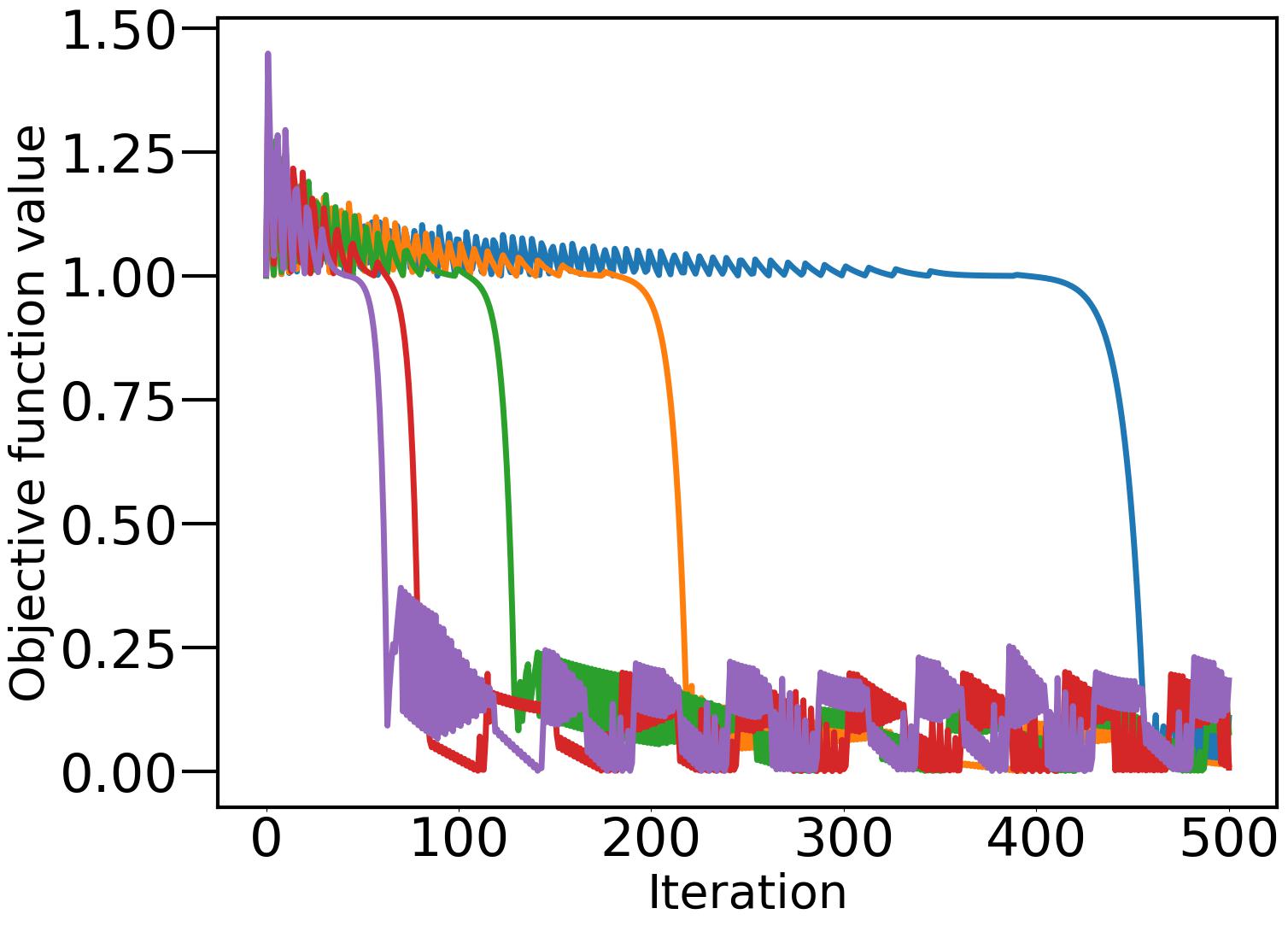}
  \caption{}
  \label{fig:nn_fv}
 \end{subfigure}
    \caption{Subgradient method randomly initialized near a spurious local minimum of a ReLU neural network with $\ell_1$ loss (5 trials with different step sizes).}
     \label{fig:nn}
\end{figure}

\begin{proposition}
\label{prop:neural_network_unstable}
The point $(1,1,0)$ is a strongly unstable spurious local minimum of the function defined from $\mathbb{R}^3$ to $\mathbb{R}$ by $f(x_1,x_2,x_3) :=  |x_3 \max\{x_2,0\} - 1| + |x_3 \max\{x_1 + x_2,0\}|$.
\end{proposition}
\begin{proof}
There exists a neighborhood $U$ of $(1,1,0)$ such that for all $(x_1,x_2,x_3)\in U$, we have $x_1\geqslant 1/2$, $x_2 \geqslant 1/2$ and $x_2 x_3 < 1$. Thus inside $U$ we have $f(x_1,x_2,x_3) = |x_3 x_2-1|+|x_3(x_1+x_2)| = 1-x_3 x_2 + |x_3|(x_1+x_2) = 1 + x_2(|x_3|-x_3)+x_1|x_3|\geqslant f(1,1,0) > f(-1,1,1)$, with equality in the inequality if and only if $x_3=0$. It follows that $(1,1,0)$ is a spurious local minimum. We next show that it is strongly unstable using Theorem \ref{thm:suff_unstable}. The function $f$ is locally Lipschitz and tame. Let $S$ denote the set of critical points of $f$. By the definable Morse-Sard theorem \cite[Corollary 9]{bolte2007clarke} and by shrinking the neighborhood $U$ if necessary, $S \cap U = \{(x_1,x_2,x_3)\in U : x_3 = 0\}$ is a $C^2$ manifold of dimension $2$ at $(1,1,0)$. Let  $\theta_1 := 1$ and $C:\mathbb{R}^3\rightarrow\mathbb{R}$ be the continuous function defined by $C(x_1,x_2,x_3) := 1 - x_1$. Let $\alpha>0$ and consider a sequence $(x_1^k,x_2^k,x_3^k)_{k \in \mathbb{N}}$ generated by the subgradient method with constant step size $\alpha$ such that $(x_1^k,x_2^k,x_3^k) \in U\setminus S$ for all $k\in \mathbb{N}$. Letting $c_1 :=\alpha$, we have $C(x_1^{k+1},x_2^{k+1},x_3^{k+1}) - C(x_1^{k},x_2^{k},x_3^k) = x_1^k - x_1^{k+1} =  \alpha |x_3^k| = c_1 d((x_1^{k},x_2^{k},x_3^k),S)^{\theta_1}$ for all $k \in \mathbb{N}$. Letting $\theta_2 := 1$, and $c_2:= 1$,
we have $d((x_1,x_2,x_3),S) \leqslant c_2 d(0,\partial f(x_1,x_2,x_3))^{\theta_2}$ for all $(x_1,x_2,x_3) \in U$, so $\partial f$ is metrically $\theta_2$-subregular at $((1,1,0),(0,0,0))$. Finally, letting $c_3 := \sqrt{5}$, for all $(x_1,x_2,x_3) \in  U\setminus S$, $(y_1,y_2,y_3) \in S\cap U$, and $(v_1,v_2,v_3) \in \partial f((x_1,x_2,x_3))$, we have $\|P_{T_S(y_1,y_2,y_3)}(v_1,v_2,v_3) - \nabla_S f(y_1,y_2,y_3)\|^2 = \|(v_1,v_2,0)\|^2 = |x_3|^2 + (|x_3| - x_3)^2 \leqslant 5x_3^2 \leqslant c_3^2\|(x_1,x_2,x_3) - (y_1,y_2,0)\|^2$. Thus $f$ satisfies the Verdier condition at $(1,1,0)$ along $S$.
\end{proof}

Second, we show that instability occurs in an example of robust principal component analysis with real-world data. The objective function $f:\mathbb{R}^{m\times r} \times \mathbb{R}^{n\times r}\rightarrow\mathbb{R}$ is defined by $f(X,Y) := \|XY^T-M\|_1$ \cite[Equation (4)]{gillis2018complexity} where $\|\cdot\|_1$ is the entrywise $\ell_1$-norm of a matrix and $M \in \mathbb{R}^{m\times n}$ is a data matrix. The goal is to decompose $M$ as a low rank matrix plus a sparse matrix. Figure \ref{fig:rpca_relative} reveals that the iterates of the subgradient method move away from a fixed spurious local minimum $(X^*,Y^*)\in \mathbb{R}^{m\times r} \times \mathbb{R}^{n\times r}$ despite being initialized nearby ($(m,n,r)=(62400,3417,10)$ in the experiment). Five trials are displayed, each corresponding to a uniform choice of constant step size in $[0.0000025,0.0000075]$ and a random initial point within $10^{-3}$ relative distance of the local minimum. Figure \ref{fig:rpca_lyapunov} shows the corresponding values of an associated Chetaev function $C:\mathbb{R}^{m\times r} \times \mathbb{R}^{n\times r}\rightarrow\mathbb{R}$ defined by $C(X,Y) := \|X^*\|_F^2-\|Y^*\|_F^2+\|Y\|_F^2-\|X\|_F^2$ where $\|\cdot\|_F$ is the Frobenius norm. The function increases as long as the iterates remain near the local minimum, but ceases to do so once the iterates are far enough. This is sufficient to prove instability (see Proposition \ref{prop:rpca_rankr_unstable}). Figure \ref{fig:rpca_fv} shows that the objective function values eventually drop below the spurious critical value and stabilize around a new value. 

\begin{figure}[ht]
    \centering
     \begin{subfigure}{.49\textwidth}
  \centering
  \includegraphics[width=.95\textwidth]{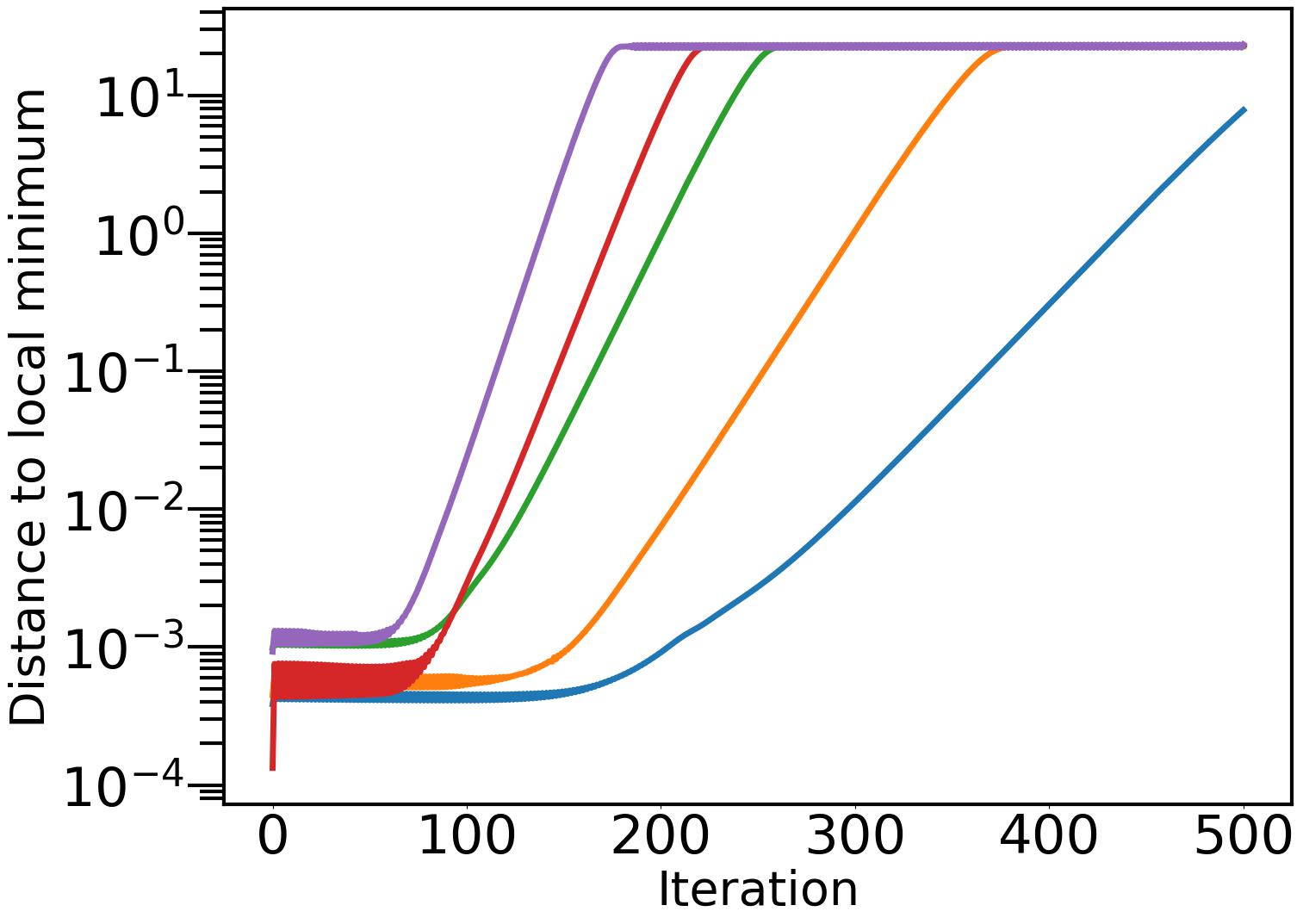}
  \caption{}
  \label{fig:rpca_relative}
\vspace*{1mm}
  \label{fig:video_distance}
 \end{subfigure}
 \begin{subfigure}{.49\textwidth}
  \centering
  \includegraphics[width=.95\textwidth]{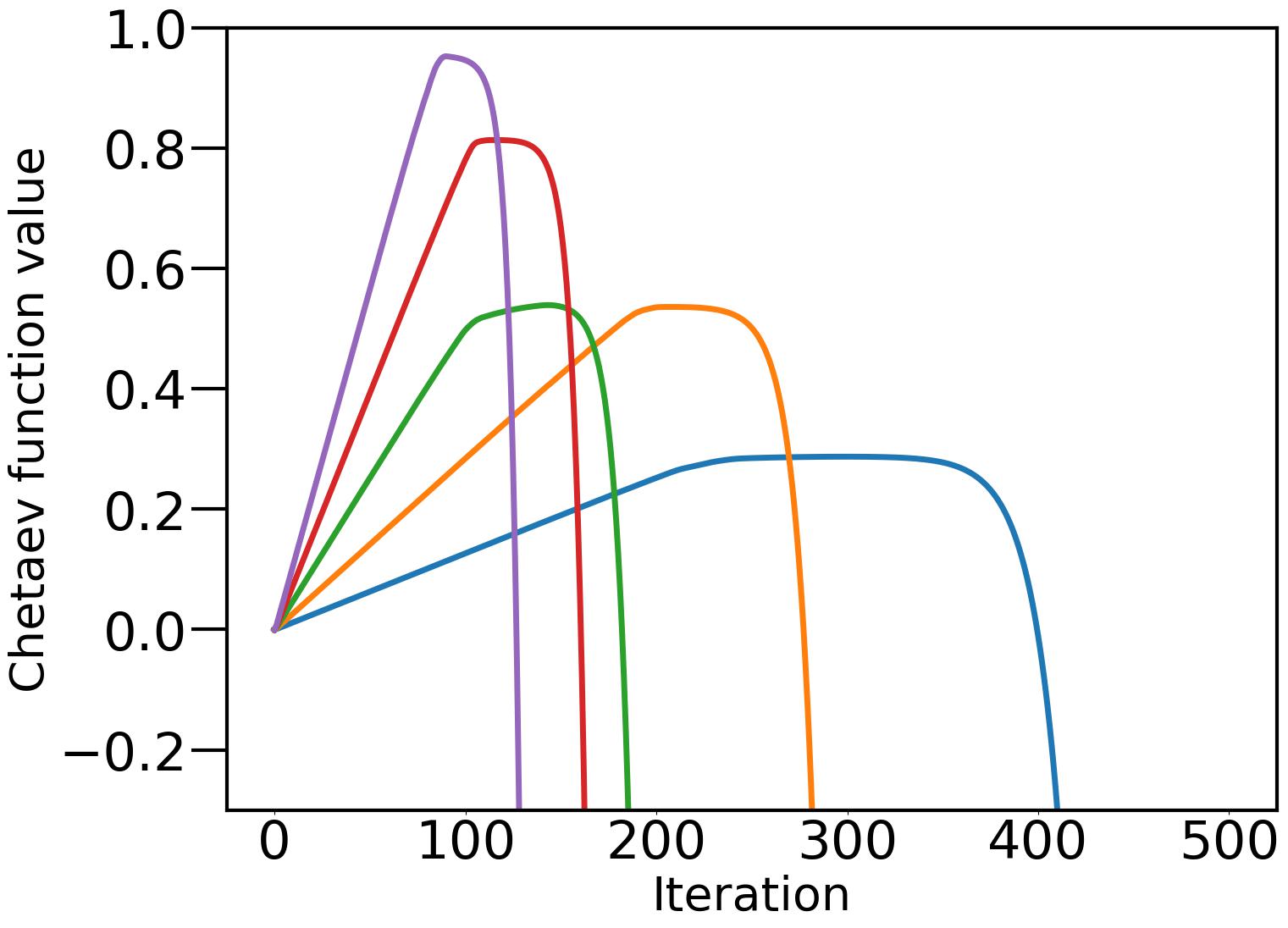}
  \caption{}
  \label{fig:rpca_lyapunov}
 \end{subfigure}
 \begin{subfigure}{.49\textwidth}
  \centering
  \includegraphics[width=.95\textwidth]{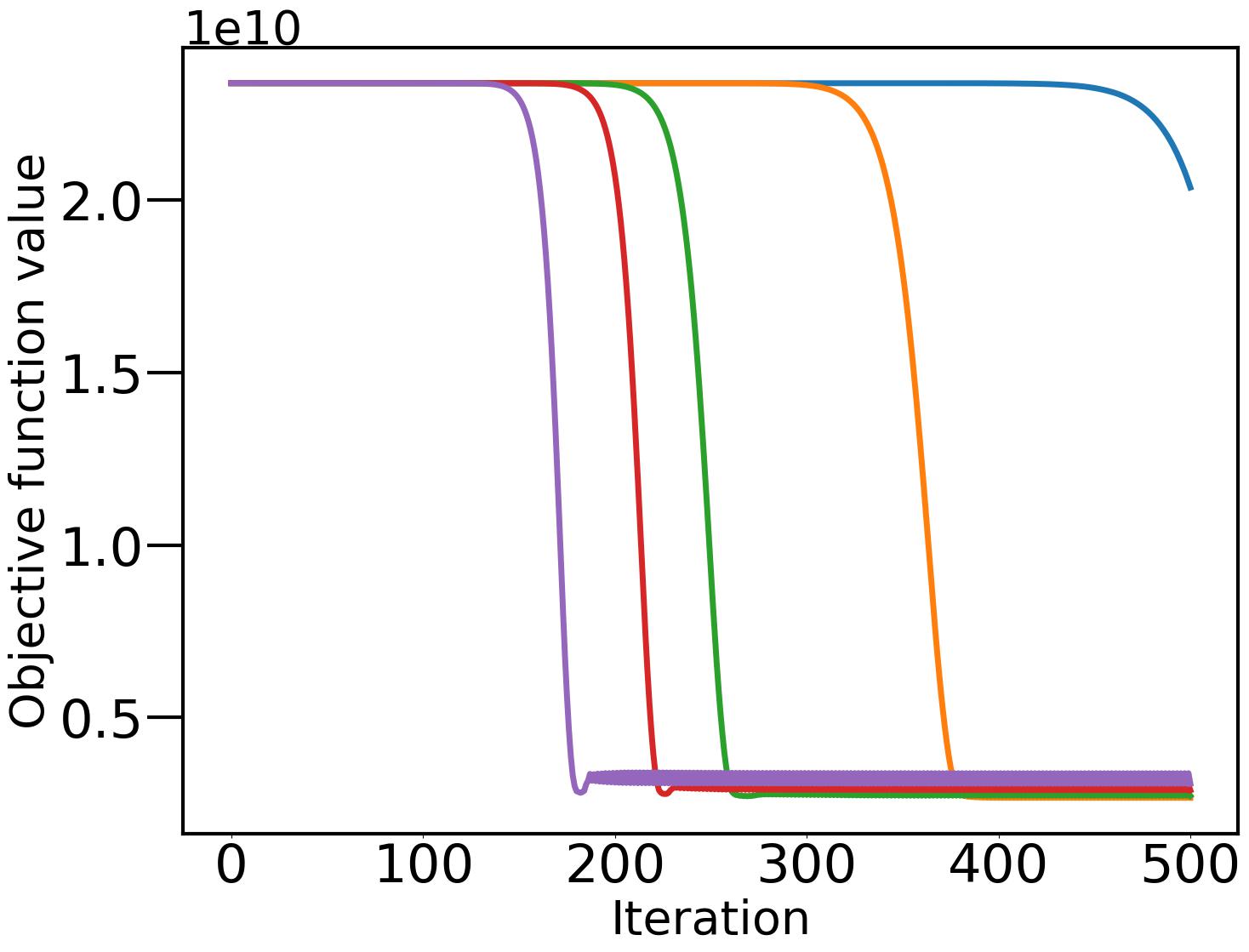}
  \caption{}
  \label{fig:rpca_fv}
 \end{subfigure}
    \caption{Subgradient method randomly initialized near a spurious local minimum of robust principal component analysis (5 trials with different step sizes).}
     \label{fig:rpca}
\end{figure}

The data used in Figure \ref{fig:rpca} is used in \cite[Figure 3]{netrapalli2014} to illustrate \textit{Non-convex Alternating Projections based Robust PCA} \cite[Algorithm 1]{netrapalli2014} and comes from the same dataset as the one used to illustrate \textit{Principal Component Pursuit} \cite[Equation (1.1)]{candes2011}. The application in those works consists of detecting moving objects in a surveillance video. Spurious local minima exist because the data matrix has zero rows, which corresponds to pixels that are composed of at most two of the three primary colors (red, green, and blue) throughout the video. It is crucial that the iterates of the subgradient method do not remain near a spurious local minimum like the one in Figure \ref{fig:rpca}. Otherwise, no moving object would be detected. In contrast, at the lower value obtained in Figure \ref{fig:rpca_fv}, all moving objects are detected. This can be seen in Figure \ref{fig:video_frames} and at the link \href{https://www.youtube.com/playlist?list=PLIR8kg8LvAmuPPZPD4CxjoqcYskoUzjzN}{[video]}.
 \begin{figure}[ht!]
 \centering
 \begin{subfigure}{.327\textwidth}
  \centering
  \includegraphics[width=1\textwidth]{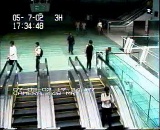}
  \caption{Original frame.}
  \label{fig:original_140}
 \end{subfigure}
 \begin{subfigure}{.327\textwidth}
  \centering
  \includegraphics[width=1\textwidth]{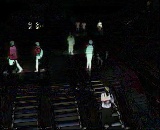}
  \caption{Moving objects.}
  \label{fig:s_140}
 \end{subfigure}
 \begin{subfigure}{.327\textwidth}
  \centering
  \includegraphics[width=1\textwidth]{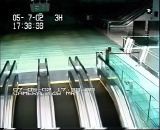}
  \caption{Background.}
  \label{fig:l_140}
 \end{subfigure}
 \vspace*{-1mm}
 \caption{Output after 500 iterations of the subgradient method with step size $\alpha = 0.000005$ when initialized within $10^{-3}$ of a spurious local minimum in relative distance.}
 \label{fig:video_frames}
 \vspace*{-3mm}
 \end{figure}
\begin{proposition}
\label{prop:rpca_rankr_unstable}
	The function $f$ defined from $\mathbb{R}^{m\times r} \times \mathbb{R}^{n\times r}$ to $\mathbb{R}$ by $f(X,Y):= \|XY^T - M\|_1$ admits strongly unstable spurious local minima if $M \in \mathbb{R}^{m\times n} \setminus \{0\}$ and $M$ contains at least $r$ zero rows or $r$ zero columns.
\end{proposition}
\begin{proof}
Without loss of generality, we assume that the first $r$ rows of $M$ are equal to zero. Let $\tilde{M}$ be the matrix containing the $m-r$ remaining rows, one of which is non-zero. We seek to show that $(X^*,Y^*) \in \mathbb{R}^{m\times r} \times \mathbb{R}^{n \times r}$ is a strongly unstable spurious local minimum of $f(X,Y):=\|XY^T-M\|_1$, where the first $r$ rows of $X^*$ form an invertible matrix $\hat{X}^* \in \mathbb{R}^{r \times r}$, the remaining rows are zero, and $Y^* = 0$. Given $(H,K) \in \mathbb{R}^{m\times r} \times \mathbb{R}^{n \times r}$, let $\hat{H}$ be the first $r$ rows of $H$ and let $\tilde{H}$ be the remaining $m-r$ rows. For all $(H,K)$ sufficiently small, we have
\begin{subequations}
\label{eq:hk}
    \begin{align}
        f(X^*+H,Y^*+K) & = \| (X^*+H)K^T - M \|_1  \\
         & =   \|(\hat{X}^* + \hat{H})K^T\|_1 + \|\tilde{H}K^T - \tilde{M}\|_1 \label{split} \\
         & \geqslant  \|\hat{X}^*K^T\|_1 - \|\hat{H}K^T\|_1 + \|\tilde{M}\|_1 - \|\tilde{H}K^T\|_1 \label{triangle}\\
         & = \|\hat{X}^*K^T\|_1 ~-~ \|HK^T\|_1 ~+~ \|M\|_1 \label{group}\\
         & \geqslant c\|K\|_1 ~-~ \|HK^T\|_1 ~+~ \|M\|_1 \label{equivalence_norms} \\
         & \geqslant c/2\|K\|_1 ~+~ \|M\|_1 \label{dual_norm} \\
         & \geqslant \|M\|_1 = f(X^*,Y^*) > f(\bar{X},\bar{Y}). \label{H_small}
    \end{align}
\end{subequations}
Above, the first term in \eqref{split} is the $\ell_1$-norm of the first $r$ rows of $(X^*+H)K^T - M$, while the second term in \eqref{split} is the $\ell_1$-norm of the remaining rows. \eqref{triangle} follows from the triangular inequality. \eqref{group} holds because the first $r$ rows of $HK^T$ are $\hat{H}K^T$, while the remaining rows are $\tilde{H}K^T$. The existence of a positive constant $c$ in \eqref{equivalence_norms} is due to the equivalence of norms ($K\mapsto \|\hat{X}^*K^T\|_1$ is a norm because $\hat{X}^*$ is invertible). \eqref{dual_norm} holds because we may take $\|H\|_\infty \leqslant c/(2m)$ where $\|\cdot\|_\infty$ is the dual norm of $\|\cdot\|_1$. Then $\|HK^T\|_1 = \sum_{i=1}^m \sum_{j=1}^n | \langle h_i , k_j \rangle | \leqslant \sum_{i=1}^m \sum_{j=1}^n \|h_i\|_\infty \|k_j\|_1 \leqslant \sum_{i=1}^m \sum_{j=1}^n \|H\|_\infty \|k_j\|_1 \leqslant c/2 \sum_{j=1}^n \|k_j\|_1 = c/2 \|K\|_1$ where $h_i^T$ and $k_j^T$ respectively denote the rows of $H$ and $K$. Finally, we may choose $(\bar{X},\bar{Y})$ in \eqref{H_small} to be factors of a rank-one matrix $\bar{M} \in \mathbb{R}^{m\times n}$ which has all zero entries, apart from one where $\bar{M}_{ij} = M_{ij} \neq 0$. Then $f(\bar{X},\bar{Y}) = \|\bar{X}\bar{Y}^T-M\|_1 = \|\bar{M}-M\|_1 = \|M\|_1 - |M_{ij}| < \|M\|_1$.

We next show that $(X^*,Y^*)$ is strongly unstable using Theorem \ref{thm:suff_unstable}. The function $f$ is locally Lipschitz and tame. Let $S$ denote the set of critical points of $f$. By the definable Morse-Sard theorem \cite[Corollary 9]{bolte2007clarke}, there exists a bounded neighborhood $U$ of the local minimum $(X^*,Y^*)$ such that $S \cap U = \{(X,Y) \in U: f(X,Y) = f(X^*,Y^*)\} = \{(X,Y) \in U:Y = Y^*\}$, where the second setwise equality is due to \eqref{dual_norm}. 
As a result, $S$ is a $C^2$ manifold at $(X^*,Y^*)$. Let $\theta_1 := 0$ and $C:\mathbb{R}^{m \times r}\times \mathbb{R}^{n \times r}\rightarrow \mathbb{R}$ be the continuous function defined by $C(X,Y):= \|X^*\|_F^2-\|Y^*\|_F^2+\|Y\|_F^2-\|X\|_F^2$. Let $\alpha>0$ and consider a sequence $(X_k,Y_k)_{k \in \mathbb{N}}$ generated by the subgradient method with constant step size $\alpha$ such that $(X_k,Y_k) \in U\setminus S$ for all $k\in \mathbb{N}$. Let $\mathrm{sign}(\cdot)$ be the function defined by $\mathrm{sign}(t) = 1$ if $t>0$, $\mathrm{sign}(t) = -1$ if $t<0$, and $\mathrm{sign}(t) = [-1,1]$ if $t=0$. When the input is a matrix, it is applied entrywise. Letting 
\begin{equation*}
    c_1 :=\alpha^2 \inf \{ \|\Lambda^T X\|_F^2 - \|\Lambda Y\|_F^2 : (X,Y) \in U\setminus S, ~ \Lambda \in \mathrm{sign}(X Y^T - M) \},
\end{equation*}
we have $C(X_{k+1},Y_{k+1}) - C(X_k,Y_k) = \hdots$
\begin{subequations}
\begin{align}
    & = \|Y_{k+1}\|_F^2 - \|X_{k+1}\|_F^2 - \|Y_{k}\|_F^2 + \|X_{k}\|_F^2 \label{eq:chpca_b}\\
    & = \mathrm{trace}(Y_{k+1}^T Y_{k+1} - X_{k+1}^T X_{k+1} - Y_k^T Y_k + X_k^T X_k) \label{eq:chpca_c}\\
    & = \alpha^2 \mathrm{trace}(X_k^T \Lambda_k \Lambda_k^T X_k - Y_k^T \Lambda_k^T\Lambda_k Y_k) \\
    & = \alpha^2 (\|\Lambda_k^T X\|_F^2 - \|\Lambda_k Y_k\|_F^2)  \geqslant c_1 d((X_k,Y_k),S)^{\theta_1}, \label{eq:chpca_d}
\end{align}
\end{subequations}
where $\Lambda_k \in \mathrm{sign}(X_k Y_k^T - M)$. It remains to show that $c_1>0$.
Given $(\Lambda,X,M) \in \mathbb{R}^{m\times n} \times \mathbb{R}^{m\times r}\times \mathbb{R}^{m\times n}$, let $(\hat{\Lambda},\hat{X},\hat{M})$ be the first $r$ rows of $(\Lambda,X,M)$ and let $(\tilde{\Lambda},\tilde{X},\tilde{M})$ be the remaining $m-r$ rows. It suffices to show that
\begin{equation*}
\label{eq:lambdah}
    \inf \{ \|\hat{\Lambda}^T \hat{X}\|_F : (X,Y) \in U \setminus S,~\hat{\Lambda} \in \mathrm{sign}(\hat{X}Y^T - \hat{M}) \} > 0
\end{equation*}
after possibly reducing the neighborhood $U$ of $(X^*,Y^*)$. Indeed, for all $(X,Y) \in U \setminus S$ and $\Lambda \in \mathrm{sign}(X Y^T - M) $, we then have $\|\Lambda^T X\|_F^2 - \|\Lambda Y\|_F^2 = \|\hat{\Lambda}^T \hat{X} + \tilde{\Lambda}^T \tilde{X}\|_F^2 - \|\Lambda Y\|_F^2 \geqslant (\|\hat{\Lambda}^T \hat{X}\|_F - \|\tilde{\Lambda}^T \tilde{X}\|_F)^2 - \|\Lambda Y\|_F^2 \geqslant \|\hat{\Lambda}^T \hat{X}\|_F^2/2$ since $\tilde{X}^* = Y^* = 0$. We next reason by contradiction and assume that the infimum in \eqref{eq:lambdah} is equal to zero.
Let $(X_i,Y_i,\Lambda_i)_{i \in \mathbb{N}}$ be a minimizing sequence. Since it is contained in the bounded set $U \times [-1,1]^{m \times n}$, there exists a subsequence (again denoted $(X_i,Y_i,\Lambda_i)_{i \in \mathbb{N}}$) that converges to some $(X^\circ,Y^\circ,\Lambda^\circ)$. Naturally we have $(\hat{\Lambda}^\circ)^T \hat{X}^\circ = 0$. On the one hand, since $\hat{X}^* \in \mathbb{R}^{r\times r}$ is invertible, so is any matrix in its neighborhood $\bar{U}$, in particular $\hat{X}^\circ,\hat{X}_0,\hat{X}_1,\hdots$ after possibly reducing $U$. Hence $\hat{\Lambda}^\circ = 0$. On the other hand, since $(X_i,Y_i) \in U \setminus S$, $S \cap U = \{ (X,Y)\in U : Y=Y^*=0\}$, and $\hat{M} = 0$, we have $Y_i \neq 0$ and $\hat{X}_iY_i^T-\hat{M} \neq 0$ for all $i \in \mathbb{N}$. Hence the matrix $\hat{\Lambda}_i \in \mathrm{sign}(\hat{X}Y^T - \hat{M})$ has at least one entry equal to either $1$ or $-1$. Thus $\|\hat{\Lambda}_i\|_\infty \geqslant 1$ for all $i \in \mathbb{N}$. Passing to the limit, we obtain the contradiction $0 =\|\hat{\Lambda}^\circ\|_\infty \geqslant 1$.
\end{proof}



\clearpage
\phantomsection 
\titleformat{\chapter}[display]
{\normalfont\bfseries\filcenter}{}{0pt}{\large\bfseries\filcenter{#1}}  
\titlespacing*{\chapter}
  {0pt}{0pt}{30pt}

\begin{singlespace}  
	\printbibliography[title={References}]
\end{singlespace}


%
%


\end{document}